\documentclass[11pt]{amsart}
\usepackage{amsmath}
\usepackage{amssymb}
\usepackage{amsthm}
\usepackage{latexsym}
\usepackage{graphicx}
\usepackage{hyperref}
\usepackage{enumerate}
\usepackage[all]{xy}

\newtheorem{thm}{Theorem}[section]
\newtheorem{cor}[thm]{Corollary}
\newtheorem{lem}[thm]{Lemma}
\newtheorem{prop}[thm]{Proposition}
\newtheorem{conj}[thm]{Conjecture}
\newtheorem{cond}[thm]{Condition}
\newtheorem{thmintro}{Theorem}
\newtheorem{condintro}[thmintro]{Condition}
\newtheorem{conjintro}[thmintro]{Conjecture}

\theoremstyle{definition}

\newtheorem{rem}[thm]{Remark}
\newtheorem{ex}[thm]{Example}

\newcommand{\N}{\mathbb N}
\newcommand{\Z}{\mathbb Z}
\newcommand{\Q}{\mathbb Q}
\newcommand{\R}{\mathbb R}
\newcommand{\C}{\mathbb C}
\newcommand{\mf}{\mathfrak}
\newcommand{\mc}{\mathcal}
\newcommand{\mb}{\mathbf}
\newcommand{\mh}{\mathbb}
\def\Irr{{\rm Irr}}
\newcommand{\mr}{\mathrm}
\newcommand{\enuma}[1]{\begin{enumerate}[\textup{(}a\textup{)}] {#1} \end{enumerate}}
\newcommand{\Fr}{\mathrm{Frob}}
\newcommand{\Sc}{\mathrm{sc}}
\newcommand{\ad}{\mathrm{ad}}
\newcommand{\cusp}{\mathrm{cusp}}
\newcommand{\nr}{\mathrm{nr}}
\newcommand{\Wr}{\mathrm{wr}}
\newcommand{\cpt}{\mathrm{cpt}}
\newcommand{\Rep}{\mathrm{Rep}}
\newcommand{\Mod}{\mathrm{Mod}}

\newcommand{\Hom}{\mathrm{Hom}}
\newcommand{\Aut}{\mathrm{Aut}}

\newcommand{\unip}{\mathrm{unip}}
\newcommand{\der}{\mathrm{der}}
\newcommand{\matje}[4]{\left(\begin{smallmatrix} #1 & #2 \\ 
#3 & #4 \end{smallmatrix}\right)}
\def\cE{{\mathcal E}}
\newcommand{\lmapsto}{\mathrel{\rotatebox[origin=c]{180}{$\mapsto$}}}

\begin{document}

\title{Langlands parameters, functoriality\\ and Hecke algebras}
\author{Maarten Solleveld}
\address{IMAPP, Radboud Universiteit Nijmegen, Heyendaalseweg 135, 
6525AJ Nijmegen, the Netherlands}
\email{m.solleveld@science.ru.nl}
\date{\today}
\thanks{The author is supported by a NWO Vidi grant "A Hecke algebra approach to the 
local Langlands correspondence" (nr. 639.032.528).}
\maketitle

\begin{abstract}
Let $G$ and $\tilde G$ be reductive groups over a local field $F$. Let $\eta : \tilde G \to
G$ be a $F$-homomorphism with commutative kernel and commutative cokernel. We
investigate the pullbacks of irreducible admissible $G$-representations $\pi$ along $\eta$. 
Following Borel, Adler--Korman and Xu, we pose a conjecture on the decomposition of 
the pullback $\eta^* \pi$. It is formulated in terms of enhanced Langlands parameters
and includes multiplicities. This can be
regarded as a functoriality property of the local Langlands correspondence.

We prove this conjecture for three classes: 
principal series representations of split groups (over non-archimedean local fields), 
unipotent representations (also with $F$ non-archimedean) and inner twists of 
$GL_n, SL_n, PGL_n$. 

Our main techniques involve Hecke algebras associated to Langlands parameters.
We also prove a version of the pullback/functoriality conjecture for those.
\end{abstract}

\thispagestyle{empty}
\tableofcontents

\section*{Introduction}

Let $F$ be a local field, let $\mc G$ be a connected reductive algebraic
$F$-group and write $G = \mc G (F)$. The (conjectural) local Langlands correspondence asserts
that there exists a ``nice'' map 
\begin{equation}\label{eq:1}
\Irr (G) \to \Phi (G) 
\end{equation}
from the set of irreducible admissible $G$-representations to the set of Langlands parameters for $G$.
This map is supposed to satisfy several ``nice'' conditions, listed in \cite{Bor,Vog}. In this
paper we discuss the functoriality of \eqref{eq:1} with respect to homomorphisms of reductive
groups. Any $F$-homomorphism $\eta : \tilde{\mc G} \to \mc G$ gives rise to a pullback functor
\[
\eta^* : \mr{Rep}(G) \to \mr{Rep}(\tilde G) , 
\]
and to a map 
\[
\Phi (\eta) : \Phi (G) \to \Phi (\tilde G) . 
\]
For $\pi \in \Irr (G)$ with L-parameter $\phi$, we would like to decompose $\eta^* \pi \in 
\mr{Rep}(\tilde G)$ and to relate it to the L-packet $\Pi_{\Phi (\eta) \phi}(\tilde G)$.
Of course, this gets exceedingly difficult when $\eta$ is far from surjective. Also, simple
factors in the kernel of $\eta$ are hardly relevant. Taking these restrictions into account,
Borel \cite[\S 10.3.5]{Bor} conjectured:
\begin{equation}\label{eq:3}
\begin{split}
\text{if } \eta \text{ has commutative kernel and cokernel and } \pi \in \Pi_\phi (G), \\
\text{ then } \eta^* \pi \text{ is a finite direct sum of members of } 
\Pi_{\Phi (\eta) \phi} (\tilde G) .
\end{split}
\end{equation}
That $\eta^* \pi$ is completely reducible and has finite length was shown around the same
time by Silberger \cite{Sil}. 

For a more precise version of this conjecture, we involve enhancements of L-parameters. 
Let $\mc S_\phi$ be the component group associated to $\phi$ in \cite{Art1,HiSa}. As usual,
an enhancement of $\phi$ is an irreducible representation $\rho$ of $\mc S_\phi$. 
Let $\Phi_e (G)$ be the set of $G$-relevant enhanced L-parameters. It is expected 
\cite{Vog,ABPSLLC} that \eqref{eq:1} can be enhanced to a bijection
\begin{equation}\label{eq:2}
\begin{array}{ccc}
\Irr (G) & \longrightarrow & \Phi_e (G) \\
\pi (\phi,\rho) & \longleftrightarrow & (\phi,\rho)
\end{array} .
\end{equation}
In particular the L-packet $\Pi_\phi (G)$ is then parametrized by the set of irreducible
$\mc S_\phi$-representations that are $G$-relevant (see Section \ref{sec:1} for details).
In general the map \eqref{eq:2} will not be unique, but the desired conditions render
it close to canonical.

In Borel's conjecture one must be careful with inseparable homomorphisms. These can be 
surjective as morphisms of algebraic groups, yet at the same time have a large, noncommutative 
cokernel as homomorphisms between groups of $F$-rational points. To rule that out, we impose:

\begin{condintro}\label{cond:1}
The homomorphism of connected reductive $F$-groups $\eta : \tilde{\mc G} \to \mc G$ satisfies
\begin{itemize}
\item[(i)] the kernel of d$\eta : \mr{Lie}(\tilde{\mc G}) \to \mr{Lie}(\mc G)$ is central;
\item[(ii)] the cokernel of $\eta$ is a commutative $F$-group.
\end{itemize}
\end{condintro}

Let ${}^L \eta = \eta^\vee \rtimes \mr{id} : G^\vee \rtimes \mb W_F \to \tilde{G}^\vee 
\rtimes \mb W_F$ be a L-homomorphism dual to $\eta$. (It is unique up to
$\tilde{G}^\vee$-conjugation.) For any $\phi \in \Phi (G)$ we get 
\[
\tilde \phi := {}^L \eta \circ \phi \in \Phi (\tilde G) .
\]
Then $\eta$ gives rise to an injective algebra homomorphism
\begin{equation}\label{eq:4}
{}^S \eta : \C [\mc S_\phi] \to \C [\mc S_{\tilde \phi}] , 
\end{equation}
which under mild assumptions is canonical. It is a twist of the injection 
\begin{equation}\label{eq:5}
{}^L \eta : \mc S_\phi \to \mc S_{\tilde \phi} 
\end{equation}
by a character of $\mc S_\phi$ (see Proposition \ref{prop:5.1}). Combining ideas from
\cite{Bor,AdPr,Xu1}, we pose:

\begin{conjintro}\label{conj:A}
Suppose that $\eta : \tilde{\mc G} \to \mc G$ satisfies Condition \ref{cond:1}. Assume that
a local Langlands correspondence exists for sufficiently large classes of representations
of $G$ and $\tilde G$. Then, for any $(\phi,\rho) \in \Phi_e (G)$:
\[
\eta^* (\pi (\phi,\rho)) = \bigoplus_{\tilde \rho \in \Irr (\mc S_{{}^L \eta \circ \phi})} 
\mr{Hom}_{\mc S_\phi} \big( \rho, {}^S \eta^* (\tilde \rho) \big) 
\otimes \pi ({}^L \eta \circ \phi, \tilde \rho) .
\]
\end{conjintro}

Here ${}^S \eta^* (\tilde \rho) = \tilde \rho \circ {}^S \eta$ denotes the pullback along 
\eqref{eq:4}. When \eqref{eq:4} is just the $\C$-linear extension of \eqref{eq:5} (which happens 
often), Conjecture \ref{conj:A} can be reformulated as 
\[
\eta^* (\pi (\phi,\rho)) = \bigoplus_{\tilde \rho \in \Irr (\mc S_{\tilde \phi})} 
\mr{Hom}_{\mc S_{\tilde \phi}} \big( \mr{ind}_{\mc S_\phi}^{\mc S_{\tilde \phi}} \rho, \tilde \rho
\big) \otimes \pi (\tilde \phi, \tilde \rho) . 
\]
Briefly, Conjecture \ref{conj:A} says that the LLC is functorial with respect to homomorphisms
with commutative (co)kernel. With the Langlands classification \cite{BoWa,Ren,SiZi,ABPSnontemp} 
one can see that validity for tempered representations and enhanced bounded L-parameters would
imply the conjecture in general \cite[\S 4]{AdPr}.

Let us list some interesting applications. Firstly, Conjecture \ref{conj:A} readily entails 
(Corollary \ref{cor:5.8}) that $\Pi_{{}^L \eta \circ \phi}(\tilde G)$ consists precisely of the 
irreducible direct summands of the $\eta^* \pi$ with $\pi \in \Pi_\phi (G)$. In particular this 
implies Borel's conjecture \eqref{eq:3}.

It is believed that every tempered L-packet $\Pi_\phi (G)$ supports a unique (up to scalars)
stable distribution $J(\phi)$ on $G$, a linear combination of the traces of the members of
$\Pi_\phi (G)$. Then $\eta^* J(\phi)$ is a stable distribution on $\tilde G$ and Conjecture
\ref{conj:A} implies, as checked in \cite[\S 2]{AdPr}, that $\eta^* J(\phi)$ is a scalar multiple
of $J({}^L \eta \circ \phi)$.

Further, Conjecture \ref{conj:A} can be used to quickly find multiplicity one results for the 
pullback of $G$-representations to $\tilde G$. Namely, $\eta^* \pi (\phi,\rho)$ is 
multiplicity-free precisely when 
\[
{}^S \eta^* : \mr{Rep} (\mc S_{\tilde \phi}) \to  \mr{Rep} (\mc S_\phi)
\]
(or equivalently ${}^L \eta^*$) is multiplicity-free on $\tilde G$-relevant irreducible 
representations of $\mc S_{\tilde \phi}$. This happens in particular when $\mc S_{\tilde \phi}$
is abelian, as is the case for many groups \cite[\S 5]{AdPr}.

Granting some comparison results between $\mc S_\phi$ and appropriate R-groups
(see \cite{ABPSnontemp,BaGo}), one can reformulate the multiplicity aspect of Conjecture \ref{conj:A} 
entirely in terms of $p$-adic groups (without L-parameters). This has been investigated in 
\cite{Cho,Key,BCG}.

It is interesting to speculate about global versions of \eqref{eq:3} and Conjecture \ref{conj:A}.
Assume that $\mc G, \tilde{\mc G}$ and $\eta$ are defined over a global field $k$ and satisfy
Condition \ref{cond:1} with $k$ instead of $F$. Let $\mh A_k$ be the ring of adeles of $k$.
Recall \cite[\S 5]{Kna} that every irreducible admissible representation $\pi$ of $\mc G (\mh A_k)$ 
factorizes as a restricted tensor product $\bigotimes'_v \pi_v$, where $v$ runs over the places 
of $k$ and $\pi_v$ is an irreducible admissible representation of $\mc G (k_v)$.
Moreover almost every $\pi_v$ is unramified, in the sense it contains a nonzero vector fixed
by a special maximal compact subgroup of $\mc G (k_v)$. Then 
\[
\eta^* (\pi) = \bigotimes'\nolimits_v \eta^* (\pi_v) \qquad
\text{as } \tilde{\mc G}(\mh A_k) \text{-representations,}
\]
We already know from \cite{Sil} that every $\eta^* \pi_v$ is a finite direct sum of irreducible
$\tilde{\mc G} (k_v)$-representations, so $\eta^* \pi$ is also completely reducible. One can 
determine the decomposition of $\eta^* \pi$ by applying Conjecture \ref{conj:A} at every
place of $k$. But, since $k$ has infinitely many places, 
it is possible that in this generality $\eta^* (\pi)$ has infinite length.

Let us suppose in addition that $\pi$ is automorphic. Given the particular shape of $\eta$,
it is easy to check, with the standard characterizations
\cite[\S 7]{Kna}, that every vector of $\eta^* \pi$ is an automorphic
form for $\tilde{\mc G}(\mh A_k)$. The best one can hope for is that the automorphicity
implies that $\eta^* \pi_v$ is irreducible for almost all $v$ (or equivalently,
at almost all places $v$ where $\pi_v$ is unramified). Assuming that, $\eta^* \pi$
is a finite direct sum of irreducible automorphic representations of $\tilde{\mc G}(\mh A_k)$.

Suppose further that a global Langlands correspondence exists for $\mc G (k)$ and 
$\tilde{\mc G}(k)$, and that $\phi$ is a L-parameter (of some kind) 
associated to $\pi$. Then one may conjecture that every (or just one) irreducible constituent 
of $\eta^* \pi$ has L-parameter $\Phi (\eta) \phi$. This is an instance of global functoriality 
\cite[\S 10]{Kna}, and it appears to be wide open. Of course, if the Langlands correspondences
would satisfy a nice local-global compability, such a conjecture might be a consequence of
\eqref{eq:3}. Nevertheless, it seems unlikely that the decomposition of $\eta^* \pi$ can be 
described with global component groups for $\phi$ and $\Phi (\eta) \phi$.\\

The main results of the paper can be summarized as follows:

\begin{thmintro}\label{thm:B} 
(see Theorem \ref{thm:6.3}, Theorem \ref{thm:6.10} and Paragraphs \ref{par:GLn}--\ref{par:SLn}) \\
Conjecture \ref{conj:A} holds for the following 
classes of $F$-groups and representations. (That is, whenever $\mc G$ and $\tilde{\mc G}$
belong to one of these classes, $\eta : \tilde{\mc G} \to \mc G$ satisfies Condition 
\ref{cond:1} and the admissible representations are of the indicated kind).
\enuma{
\item Split reductive groups over non-archimedean local fields and irreducible representations
in the principal series, with the LLC from \cite{ABPSprin}.
\item Unipotent representations of groups over non-archimedean local fields which split over 
an unramified extension, with the LLC from \cite{LusUni1,LusUni2,FOS,SolLLCunip}.
\item Inner twists of $GL_n, SL_n$ and $PGL_n$ over local fields, following \cite{HiSa,ABPSSL1}.
}
\end{thmintro}

For quasi-split classical groups over local fields of characteristic zero, Conjecture
\ref{conj:A} has been established with endoscopic methods. To provide a proper perspective, 
we collect all those instances in Paragraph \ref{par:classical}. We also mention that Conjecture
\ref{conj:A} for real reductive groups should be related to parts of \cite{ABV}, which however
rather treat Arthur packets.

The maps \eqref{eq:4} depend multiplicatively on $\eta$, hence Conjecture \ref{conj:A} is 
transitive in $\eta$. This enables us (cf. Section \ref{sec:5}) to reduce the verification 
to four classes of homomorphisms, which we discuss now. 
\begin{itemize}
\item \textbf{Inclusions $\tilde{\mc G} \to \tilde{\mc G} \times \mc T$, where $\mc T$ is a $F$-torus.}\\
This case is trivial. 

\item \textbf{Quotient maps $q : \tilde{\mc G} \to \mc G = \tilde{\mc G} / \mc N$, where $\mc N
\subset \tilde{\mc G}$ is central.} \\
Together with inclusions $\tilde{\mc G} \to \tilde{\mc G} \times \mc T$, these account for all
inclusions. For instance, $SL_n \hookrightarrow GL_n$ can be factorized as 
$SL_n \to SL_n \times Z(GL_n) \to GL_n$, where the second map is surjective as a homomorphism
of algebraic groups. 

The failure of $q : \tilde G \to G$ to be surjective (in general) entails that $q^*$ need
not preserve irreducibility and that it may enlarge the finite groups attached to Bernstein
components. That makes this case very technical.

\item \textbf{Inner automorphisms $\mr{Ad}(g)$ with $g \in \mc G_\ad (F) = G_\ad$.}\\
Regarded as automorphisms of the abstract group $G = \mc G (F)$, these are not necessarily
inner, so in principle they can act nontrivially on $\Irr (G)$. When $F$ is non-archimedean
and $g$ lies in a compact subgroup of $G_\ad$, $\mr{Ad}(g)^*$ typically acts on $\Irr (G)$
via permutations of the cuspidal supports. In general, for a parabolic subgroup $P = MU$ 
the action of $\mr{Ad}(g)^*$ on constituents of $I_P^G (\sigma)$ with $\sigma \in \Irr (M)$
essentially discrete series will also involve a character of the arithmetic R-group 
associated to $I_P^G (\sigma)$ \cite[\S 1]{ABPSnontemp}. 

The corresponding action on $\Phi_e (G)$ stabilizes all L-parameters, it only permutes the
enhancements. This could be expected, as L-packets should be the minimal subsets of
$\Irr (G)$ that are stable under conjugation by $G_\ad$. To any $g \in G_\ad$ and
$\phi \in \Phi (G)$ we canonically associate (in Paragraph \ref{par:2.1}) a character 
$\tau_\phi (g)$ of $Z_{G^\vee}(\phi (\mb W_F))$, which induces a character of $\mc S_\phi$.
In all cases considered in this paper:
\[
\mr{Ad}(g)^* \pi (\phi,\rho) = \pi (\phi, \rho \otimes \tau_\phi (g)^{-1}) . 
\]
This equality is responsible for the character twists in the definition of 
${}^S \mr{Ad}(g) \in \mr{Aut} \, \C [\mc S_\phi]$. See Example \ref{ex:SLn} for characters
$\tau_\phi (g)$ of high order.
\item \textbf{Isomorphisms of reductive $F$-groups}

Recall that to any connected reductive $F$-group $\mc G$ one can associate a based root
datum $\mc R (\mc G,\mc T)$, endowed with an action of the Weil group $\mb W_F$. The
$\mb W_F$-automorphisms of $\mc R (\mc G,\mc T)$ are the source of all 
elements of the outer automorphism group of $\mc G$, and $\mb W_F$-equivariant isomorphisms
of based root data give rise to isomorphisms between reductive groups. 
This is well-known for split reductive groups, we make it precise in general. 
\end{itemize}

\begin{thmintro}\label{thm:C}
(see Theorem \ref{thm:3.2} and Proposition \ref{prop:3.4}) \\
Let $\mc G$ be an inner twist of a quasi-split reductive $F$-group $\mc G^*$. Let 
$\zeta \in \Irr (Z({G^\vee}_\Sc)^{\mb W_F})$ be the Kottwitz parameter of $\mc G$.
Assume that analogous objects are given for $\tilde{\mc G}$, with tildes. Let
\[
\tau : \mc R (\tilde{\mc G}^* , \tilde{\mc T}^* ) \to \mc R (\mc G^*,\mc T^*)
\]
be a $\mb W_F$-equivariant isomorphism of based root data. The following are equivalent:
\begin{itemize}
\item[(i)] $\tau (\tilde \zeta) = \zeta$;
\item[(ii)] there exists an isomorphism of $F$-groups $\eta : \tilde{\mc G} \to \mc G$
which lifts $\tau$. 
\end{itemize}
Every isomorphism $\tilde{\mc G} \to \mc G$ arises in this way.
\enuma{
\item When (ii) holds, the group $\mc G_\ad (F)$ acts simply transitively on the
collection of such $\eta$ (by composition).
\item When $\mc G$ and $\tilde{\mc G}$ are quasi-split, (i) and (ii) hold for every $\tau$.
The isomorphism $\eta$ can be determined uniquely by requiring that it sends a chosen
$\mb W_F$-stable pinning of $\tilde{\mc G}$ to a chosen $\mb W_F$-stable pinning of $\mc G$.
}
\end{thmintro}

With Theorem \ref{thm:C} we reduce the verification of Conjecture \ref{conj:A} for
isomorphisms to inner automorphisms and to one $\eta$ for every $\tau$ as above. For
such $\eta$ it is usually easy, because one can choose them so that they preserve 
the entire setup.\\

Now we focus on non-archimedean local fields $F$.
The main players in our proof of cases of Conjecture \ref{conj:A} are Hecke algebras.
In previous work \cite{AMS3} a twisted affine Hecke algebra $\mc H (\mf s^\vee, 
\vec{\mb z})$ was attached to every Bernstein component of enhanced L-parameters 
$\Phi_e (G)^{\mf s^\vee}$. Its crucial property is that, for every specialization of
$\vec{\mb z}$ to an array $\vec z$ of parameters in $\R_{\geq 1}$, there exists a
canonical bijection 
\[
\Phi_e (G)^{\mf s^\vee} \longleftrightarrow
\Irr \big( \mc H (\mf s^\vee, \vec{\mb z}) / (\vec{\mb z} - \vec z) \big) .
\]
The entire study of the algebras $\mc H (\mf s^\vee, \vec{\mb z})$ takes place in the
realm of complex groups with a $\mb W_F$-action, it is not conditional on the existence
of types or a LLC. 

It is expected that every Bernstein block $\mr{Rep}(G)^{\mf s}$ in $\mr{Rep}(G)$
is equivalent to the module category of an affine Hecke algebra (or a very similar
kind of algebra), say $\mc H (\mf s)$. The above Hecke algebras for L-parameters
essentially allow one to reduce a proof of the LLC to two steps:
\begin{itemize}
\item a LLC on the cuspidal level, which in particular matches $\mf s$ with a
unique $\mf s^\vee$;
\item for all matching inertial equivalence classes $\mf s$ and $\mf s^\vee$, a Morita
equivalence between $\mc H (\mf s)$ and $\mc H (\mf s^\vee, \vec{\mb z}) / 
(\vec{\mb z} - \vec z)$ (for suitable parameters $\vec z$).
\end{itemize}
With this in mind, Conjecture \ref{conj:A} can be translated to a statement about
representations of the algebras $\mc H (\mf s^\vee, \vec{\mb z})$. We note that 
the pullback ${}^L \eta (\mf s^\vee)$ is in general not a single inertial equivalence
class for $\Phi_e (\tilde G)$, rather a finite union thereof. Consequently
$\mc H ({}^L \eta (\mf s^\vee),\vec{\mb z})$ is a finite direct sum of (twisted) affine
Hecke algebras associated to parts of $\Phi_e (\tilde G)$. This reflects that an L-packet
$\Pi_{\tilde \phi}(\tilde G)$ need not be contained in a single Bernstein component.

\begin{thmintro}\label{thm:D}
Let $\eta : \tilde{\mc G} \to \mc G$ be as in Condition \ref{cond:1}, and let 
$\Phi_e (G)^{\mf s^\vee}$ be a Bernstein component of $\Phi_e (G)$. Assume that for
every involved $\tau$ (as in Theorem \ref{thm:C}):
\begin{itemize}
\item there exists a canonical choice of $\eta$ (also as in Theorem \ref{thm:C}),
\item the group $W_{\mf s^\vee}$ attached to $\mf s^\vee$ fixes a point of the
torus $\mf s_L^\vee$ attached to $\mf s^\vee$ (see Section \ref{sec:1} for background),
\item $(w - 1) X^* (X_\nr (L_{AD})) \subset X^* (\mf s_L^\vee)$ for all $w \in
W_{\mf s^\vee}$ (confer Lemma \ref{lem:2.7}).
\end{itemize}
Then Conjecture \ref{conj:A} holds for representations of $\mc H (\mf s^\vee, \vec{\mb z})$
and $\mc H ({}^L \eta (\mf s^\vee), \vec{\mb z})$. 
\end{thmintro}

For principal series representation (of $F$-split groups) and unipotent representations
(of groups splitting over an unramified field extension), Theorem \ref{thm:D} constitutes
a large part of the proof of Conjecture \ref{conj:A}.\\

We conclude the introduction with some clarification of the structure of the paper. 
Throughout Sections \ref{sec:1}--\ref{sec:2}, \ref{sec:4} and \ref{sec:principal}--\ref{sec:unip}, 
the local field $F$ is supposed to be non-archimedean. In general all statements involving
Hecke algebras only apply when $F$ is non-archimedean, whereas most other results will
be established over all local fields.

In the first section we recall some notions and results about enhanced L-parameters and the 
associated Hecke algebras (affine and graded). In Sections \ref{sec:2}--\ref{sec:4} we 
investigate the action of homomorphisms $\tilde G \to G$ on these algebras. To every
Ad$(g) \in \mr{Aut}(G)$ with $g \in G_\ad$ we associate (in Paragraph \ref{par:2.3}) an
algebra isomorphism 
\[
\mc H (\mf s^\vee, \vec{\mb z}) \to \mc H (\mf s^\vee \otimes \tau_{\phi_L}(g),\vec{\mb z}) ,
\]
which has the desired effect on representations. For every central quotient map
$q : \tilde{\mc G} \to \mc G$ we would like to construct a homomorphism
\begin{equation}\label{eq:6}
\mc H ({}^L q (\mf s^\vee), \vec{\mb z}) \to \mc H (\mf s^\vee,\vec{\mb z}) .
\end{equation}
Unfortunately, this is in general not possible directly, only via some intermediate algebras.
We show that nevertheless there is a canonical notion of pullback of modules along
\eqref{eq:6}.

In Section \ref{sec:5} we provide solid footing to formulate Conjecture \ref{conj:A} precisely.
We also wrap up the findings from the previous sections to establish Theorem \ref{thm:D}. Up to 
this point, our results do not use any knowledge of a local Langlands correspondence (beyond
the case of tori). 

The remaining sections are dedicated to the proofs of our main results. Building upon
\cite{ABPSprin}, we verify Theorem \ref{thm:B}.a in Section \ref{sec:principal}. 
Hecke algebras for unipotent representations were studied mainly in \cite{LusUni1,SolLLCunip}.
In Section \ref{sec:unip} we combine these sources with the first half of the paper to prove 
Theorem \ref{thm:B}.b.
In Section \ref{sec:8} we first recall some background of the LLC for inner twists of
$GL_n, PGL_n$ and $SL_n$. With the appropriate formulations at hand, we settle Theorem
\ref{thm:B}.c. This case is easier than the previous two, no Hecke algebras are required.\\

\textbf{Acknowledgements.}
The author thanks Dipendra Prasad, Jeff Adler, Santosh Nadimpalli and the referee 
for some useful comments.

\numberwithin{equation}{section}
\renewcommand{\theequation}{\arabic{section}.\arabic{equation}}

\section{Hecke algebras for Langlands parameters}
\label{sec:1}

Let $F$ be a non-archimedean local field with ring of integers $\mf o_F$ and
a uniformizer $\varpi_F$. Let $k_F = \mf o_F / \varpi_F \mf o_F$ be its residue
field, of cardinality $q_F$. We fix a separable closure $F_s$ and assume that all
finite extensions of $F$ are realized in $F_s$. 
Let $\mb W_F \subset \mr{Gal}(F_s / F)$ be the Weil group of $F$ and let $\Fr$ be a
geometric Frobenius element. Let $\mb I_F  \subset \mb W_F$ be the 
inertia subgroup, so that $\mb W_F / \mb I_F \cong \Z$ is generated by $\Fr$.

Let $\mc G$ be a connected reductive $F$-group. Let $\mc T$ be a maximal torus 
of $\mc G$, and let $\Phi (\mc G, \mc T)$ 
be the associated root system. We also fix a Borel subgroup $\mc B$ of $\mc G$ 
containing $\mc T$, which determines a basis $\Delta$ of $\Phi (\mc G, \mc T)$.
Let $\mc S$ be a maximal $F$-split torus in $\mc G$. By \cite[Theorem 13.3.6.(i)]{Spr} 
applied to $Z_{\mc G}(\mc S)$, we may assume that $\mc T$ is defined over $F$ and 
contains $\mc S$. Then $Z_{\mc G}(\mc S)$ is a minimal Levi $F$-subgroup 
of $\mc G$ and $\mc B Z_{\mc G}(\mc S)$ is a minimal parabolic $F$-subgroup of $\mc G$.

We denote the complex dual group of $\mc G$ by $G^\vee$ or $\mc G^\vee$. 
Let ${G^\vee}_\ad$ be the adjoint group of $G^\vee$, and let 
${G^\vee}_\Sc = (\mc G_\ad)^\vee$ be its simply connected cover. 

We write $G = \mc G(F)$ and similarly for other $F$-groups. 
Recall that a Langlands parameter for $G$ is a homomorphism 
\[
\phi : \mb W_F \times SL_2 (\C) \to {}^L G = G^\vee \rtimes \mb W_F ,
\]
with some extra requirements. In particular $\phi |_{SL_2 (\C)}$ has to be algebraic, 
$\phi (\mb W_F)$ must consist of semisimple elements and $\phi$ must respect the
projections to $\mb W_F$. 

We say that a L-parameter $\phi$ for $G$ is 
\begin{itemize}
\item discrete if there does not exist any proper L-Levi subgroup of ${}^L G$ containing the
image of $\phi$;
\item bounded if $\phi (\Fr) = (s,\Fr)$ with $s$ in a bounded subgroup of $G^\vee$;
\item unramified if $\phi (w) = (1,w)$ for all $w \in \mb I_F$.
\end{itemize}
Let $\mc G^*$ be the unique $F$-quasi-split inner form of $\mc G$. 
We consider $\mc G$ as an inner twist of $\mc G^*$, so endowed with a $F_s$-isomorphism
$\mc G \to \mc G^*$. Via the Kottwitz homomorphism $\mc G$ is labelled by character
$\zeta_{\mc G}$ of $Z({G^\vee}_\Sc)^{\mb W_F}$ (defined with respect to $\mc G^*$).

Both ${G^\vee}_\ad$ and ${G^\vee}_\Sc$ act on $G^\vee$ by conjugation. As
\[
Z_{G^\vee}(\text{im } \phi) \cap Z(G^\vee) = Z(G^\vee)^{\mb W_F} ,
\]
we can regard $Z_{G^\vee}(\text{im } \phi) / Z(G^\vee)^{\mb W_F}$ as a subgroup of 
${G^\vee}_\ad$.  Let $Z^1_{{G^\vee}_\Sc}(\text{im } \phi)$ be its inverse image in
${G^\vee}_\Sc$ (it contains $Z_{{G^\vee}_\Sc}(\text{im } \phi)$ with finite index). The
S-group of $\phi$ is
\begin{equation}\label{eq:1.1}
\mc S_\phi := \pi_0 \big( Z^1_{{G^\vee}_\Sc}(\text{im } \phi) \big) .
\end{equation}
An enhancement of $\phi$ is an irreducible representation $\rho$ of $\mc S_\phi$.
Via the canonical map $Z({G^\vee}_\Sc)^{\mb W_F} \to \mc S_\phi$, $\rho$ determines
a character $\zeta_\rho$ of $Z({G^\vee}_\Sc)^{\mb W_F}$. 
We say that an enhanced L-parameter $(\phi,\rho)$ is relevant for $G$ if $\zeta_\rho =
\zeta_{\mc G}$. This can be reformulated with $G$-relevance of $\phi$ in terms of Levi
subgroups \cite[Lemma 9.1]{HiSa}. To be precise, in view of \cite[\S 3]{Bor}
there exists an enhancement $\rho$ such that $(\phi,\rho)$ is $G$-relevant if and only if 
every L-Levi subgroup of ${}^L G$ containing the image of $\phi$ is $G$-relevant. 
The group $G^\vee$ acts naturally on the collection of $G$-relevant enhanced L-parameters, by 
\[
g \cdot (\phi,\rho) = (g \phi g^{-1},\rho \circ \mr{Ad}(g)^{-1}) .
\]
We denote the set of $G^\vee$-equivalence classes of $G$-relevant L-parameters by $\Phi (G)$.
The subset of unramified (resp. bounded, resp. discrete) $G$-relevant L-parameters is
denoted by $\Phi_\nr (G)$ (resp. $\Phi_{\mr{bdd}}(G)$, resp. $\Phi_{\mr{disc}}(G)$). 

For certain topics, the above gives too many enhancements. To fix that, we choose
an extension $\zeta_{\mc G}^+ \in \Irr (Z({G^\vee}_\Sc))$ of $\zeta_{\mc G}$.
Let $\mc Z_\phi$ be the image of $Z({G^\vee}_\Sc)$ in $\mc S_\phi$. Via
$Z({G^\vee}_\Sc) \to \mc Z_\phi \to \mc S_\phi$, any enhancement $\rho$ also determines
a character $\zeta_\rho^+$ of $Z({G^\vee}_\Sc)$. In the more precise sense, we say that 
\begin{equation}\label{eq:1.4}
(\phi,\rho) \text{ is relevant for } (G,\zeta_{\mc G}^+) \text{ if } 
\zeta_\rho^+ = \zeta_{\mc G}^+ .
\end{equation}
This can be interpreted in terms of rigid inner forms of $\mc G$ \cite{Kal2}.

The set of $G^\vee $-equivalence classes of relevant L-parameters should be denoted 
$\Phi_e (G,\zeta_{\mc G}^+)$. However in practice, we will usually omit $\zeta_{\mc G}^+$ 
from the notation, and we write simply $\Phi_e (G)$. A local Langlands correspondence 
for $G$ (in its modern interpretation) should be a bijection between $\Phi_e (G)$ and 
the set of irreducible smooth $G$-representations, with several nice properties.

Let $H^1 (\mb W_F, Z(G^\vee))$ be the first Galois cohomology group of $\mb W_F$ with
values in $Z(G^\vee)$. It acts on $\Phi (G)$ by
\begin{equation}\label{eq:2.9}
(z \phi) (w,x) = z' (w) \phi (w,x) \qquad \phi \in \Phi (G), w \in \mb W_F, x \in SL_2 (\C) ,
\end{equation}
where $z' : \mb W_F \to Z(G^\vee)$ represents $z \in H^1 (\mb W_F, Z(G^\vee))$. 
This extends to an action of $H^1 (\mb W_F, Z(G^\vee))$ on $\Phi_e (G)$, which does 
nothing to the enhancements.

Let us focus on cuspidality for enhanced L-parameters \cite[\S 6]{AMS1}. Consider
\[
G^\vee_\phi := Z^1_{{G^\vee}_\Sc}( \phi |_{\mb W_F}),
\]
a possibly disconnected complex reductive group. Then $u_\phi := \phi \big( 1, 
\big( \begin{smallmatrix} 1 & 1 \\ 0 & 1 \end{smallmatrix} \big)\big)$ can be regarded as a 
unipotent element of $(G^\vee_\phi )^\circ$ and 
\begin{equation}\label{eq:1.3}
\mc S_\phi \cong \pi_0 (Z_{G^\vee_\phi}(u_\phi)).
\end{equation}
We say that $(\phi,\rho) \in \Phi_e (G)$ is cuspidal if $\phi$ is discrete and $(u_\phi,\rho)$ 
is a cuspidal pair for $G^\vee_\phi$. The latter means that $(u_\phi,\rho)$ determines a 
$G^\vee_\phi$-equivariant cuspidal local system on the $(G^\vee_\phi)^\circ$-conjugacy class 
of $u_\phi$. Notice that a L-parameter alone does not contain enough information to detect 
cuspidality, for that we really need an enhancement. Therefore we will often say 
"cuspidal L-parameter" for an enhanced L-parameter which is cuspidal. 

The set of $G^\vee$-equivalence classes of $G$-relevant cuspidal L-parameters is denoted
$\Phi_\cusp (G)$. It is conjectured that under the LLC $\Phi_\cusp (G)$ corresponds to
the set of supercuspidal irreducible smooth $G$-representations.

The cuspidal support of any $(\phi,\rho) \in \Phi_e (G)$ is defined in \cite[\S 7]{AMS1}. 
It is unique up to $G^\vee$-conjugacy and consists of a $G$-relevant L-Levi subgroup
${}^L L$ of ${}^L G$ and a cuspidal L-parameter  $(\phi_v, q\epsilon)$ for ${}^L L$. 
By \cite[Corollary 1.3]{SolLLCunip} this ${}^L L$ corresponds to a unique (up to $G$-conjugation)
Levi $F$-subgroup $\mc L$ of $\mc G$. This allows us to express the aforementioned 
cuspidal support map as
\begin{equation}\label{eq:2.1}
\mb{Sc}(\phi,\rho) = (\mc L (F), \phi_v, q \epsilon), \quad \text{where }
(\phi_v, q \epsilon) \in \Phi_\cusp (\mc L (F)) .
\end{equation}
It is conjectured that under the LLC this map should correspond to Bernstein's cuspidal
support map for irreducible smooth $G$-representations.

Sometimes we will be a little sloppy and write that $L = \mc L (F)$ is a Levi subgroup of $G$. 
Let $X_\nr (L)$ be the group of unramified characters $L \to \C^\times$.
As worked out in \cite[\S 3.3.1]{Hai}, it is naturally isomorphic to 
$(Z (L^\vee)^{\mb I_F} )_\Fr^\circ \subset H^1 (\mb W_F, Z(L^\vee))$. As such it acts on 
$\Phi_e (L)$ and on $\Phi_\cusp (L)$ by \eqref{eq:2.9}.
A cuspidal Bernstein component of $\Phi_e (L)$ is a set of the form
\[
\Phi_e (L)^{\mf s_L^\vee} := 
X_\nr (L) \cdot (\phi_L,\rho_L) \text{ for some } (\phi_L,\rho_L) \in \Phi_\cusp (L) . 
\]
The group $G^\vee$ acts on the set of cuspidal Bernstein components for all Levi subgroups
of $G$. The $G^\vee$-action is just by conjugation, but to formulate it precisely, more 
general L-Levi subgroups of ${}^L G$ are necessary. We prefer to keep those out of the notations, 
since we do not need them to get all classes up to equivalence. With that convention, we can 
define an inertial equivalence class for $\Phi_e (G)$ as 
\[
\mf s^\vee \text{ is the } G^\vee \text{-orbit of } 
(L, X_\nr (L) \cdot (\phi_L,\rho_L) ), \text{ where } (\phi_L,\rho_L) \in \Phi_\cusp (L).
\]
The underlying inertial equivalence class for $\Phi_e (L)$ is $\mf s_L^\vee = (L,X_\nr (L) \cdot
(\phi_L,\rho_L))$. Here it is not necessary to take the $L^\vee$-orbit, for 
$(\phi_L,\rho_L) \in \Phi_e (L)$ is fixed by $L^\vee$-conjugation.

We denote the set of inertial equivalence classes for $\Phi_e (G)$ by $\mf{Be}^\vee (G)$.
Every $\mf s^\vee \in \mf{Be}^\vee (G)$ gives rise to a Bernstein component in $\Phi^e (G)$
\cite[\S 8]{AMS1}, namely
\begin{equation}\label{eq:2.10}
\Phi_e (G)^{\mf s^\vee} = \{ (\phi,\rho) \in \Phi_e (G) : \mb{Sc}(\phi,\rho) \in \mf s^\vee \} .
\end{equation}
The set of such Bernstein components is also parametrized by $\mf{Be}^\vee (G)$, and forms a 
partition of $\Phi_e (G)$.

Notice that $\Phi_e (L)^{\mf s^\vee_L} \cong \mf s_L^\vee$ has a canonical topology, 
coming from the transitive action of $X_\nr (L)$. More precisely, let $X_\nr (L,\phi_L)$ 
be the stabilizer in $X_\nr (L)$ of $\phi_L$. Then the complex torus 
\[
T_{\mf s_L^\vee} := X_\nr (L) / X_\nr (L,\phi_L)
\] 
acts simply transitively on $\mf s_L^\vee$. This endows $\mf s_L^\vee$ with the structure of an 
affine variety. (There is no canonical group structure on $\mf s_L^\vee$ though, for that one 
still needs to choose a basepoint.) 

To $\mf s^\vee$ we associate a finite group $W_{\mf s^\vee}$, in many
cases a Weyl group. For that, we choose $\mf s_L^\vee = (L, X_\nr (L) \cdot (\phi_L,\rho_L) )$
representing $\mf s^\vee$ (up to isomorphism, the below does not depend on this choice). We define 
$W_{\mf s^\vee}$ as the stabilizer of $\mf s_L^\vee$ in $N_{G^\vee}(L^\vee \rtimes \mb W_F) / L^\vee$. 
In this setting we write $T_{\mf s^\vee}$ for $T_{\mf s_L^\vee}$. Thus $W_{\mf s^\vee}$ acts on 
$\mf s_L^\vee$ by algebraic automorphisms and on $T_{\mf s^\vee}$ by group automorphisms (but the 
bijection $T_{\mf s^\vee} \to \mf s_L^\vee$ need not be $W_{\mf s^\vee}$-equivariant). 

Next we quickly review the construction of an affine Hecke algebra from a Bernstein component
of enhanced Langlands parameters. We fix a basepoint $\phi_L$ for $\mf s_L^\vee$ as in
\cite[Proposition 3.9]{AMS3}, and use that to identify $\mf s_L^\vee$ with $T_{\mf s_L^\vee}$.
Consider the possibly disconnected reductive group
\[
G^\vee_{\phi_L} = Z^1_{{G^\vee}_\Sc} (\phi_L |_{\mb W_F}) .
\]
Let $L_c^\vee$ be the Levi subgroup of ${G^\vee}_\Sc$ determined by $L^\vee$. There
is a natural homomorphism 
\begin{equation}\label{eq:2.2}
Z(L_c^\vee)^{\mb W_F,\circ} \to X_\nr (L) \to T_{\mf s_L^\vee}
\end{equation}
with finite kernel \cite[Lemma 3.7]{AMS3}. Using that and \cite[Lemma 3.10]{AMS3}, 
$\Phi (G_{\phi_L}^\circ, Z(L_c^\vee)^{\mb W_F,\circ})$ gives rise to a reduced root system 
$\Phi_{\mf s^\vee}$ in $X^* (T_{\mf s^\vee})$. The coroot system $\Phi^\vee_{\mf s^\vee}$ 
is contained in $X_* (T_{\mf s^\vee})$. That gives a root datum $\mc R_{\mf s^\vee}$, whose
basis can still be chosen arbitrarily. The group $W_{\mf s^\vee}$ acts naturally on 
$\mc R_{\mf s^\vee}$ and contains the Weyl group of $\Phi_{\mf s^\vee}$. 

The construction of label functions $\lambda$ and $\lambda^*$ for $\mc R_{\mf s^\vee}$ 
consists of several steps. The numbers $\lambda (\alpha), \lambda^* (\alpha) \in \Z_{\geq 0}$
will be defined for all $\alpha \in \Phi_{\mf s^\vee}$. First, we pick
$t \in (Z(L_c^\vee)^{\mb I_F})^\circ_\Fr$ such that the reflection $s_\alpha$ fixes 
$t \phi_L (\Fr)$. Then $q \alpha$ lies in $\Phi \big( (G^\vee_{t \phi_L})^\circ,
Z(L_c^\vee)^{\mb W_F,\circ} \big)$ for some $q \in \Q_{>0}$. The labels
$\lambda (q \alpha), \lambda^* (q \alpha)$ are related $\Q$-linearly to the labels 
$c(q \alpha), c^* (q \alpha)$ for a graded Hecke algebra \cite[\S 1]{AMS3} associated to
\begin{equation}\label{eq:2.3}
(G^\vee_{t \phi_L})^\circ = Z_{{G^\vee}_\Sc}(t \phi_L (\mb W_F))^\circ, 
Z(L_c^\vee)^{\mb W_F,\circ}, u_{\phi_L} \text{ and } \rho_L .
\end{equation}
These integers $c(q \alpha), c^* (q \alpha)$ were defined in \cite[Propositions 2.8, 2.10 and 
2.12]{LusCusp}, in terms of the adjoint action of $\log (u_{\phi_L})$ on
\[
\mr{Lie} (G^\vee_{t \phi_L})^\circ = \mr{Lie} \big( Z_{{G^\vee}_\Sc}(t \phi_L (\mb W_F)) \big) .
\]
In \cite[Proposition 3.13 and Lemma 3.14]{AMS3} it is described which 
$t \in (Z(L_c^\vee)^{\mb I_F})^\circ_\Fr$ we need to determine all labels: just one with 
$\alpha (t) = 1$, and sometimes one with $\alpha (t) = -1$. 

Finally, we choose an array $\vec{\mb z}$ of $d$ invertible variables, one $\mb z_j$ for every 
$W_{\mf s^\vee}$-orbit of irreducible components of $\Phi_{\mf s^\vee}$. To these data one can 
attach an affine Hecke algebra $\mc H (\mc R_{\mf s^\vee}, \lambda, \lambda^*, \vec{\mb z})$, 
as in \cite[\S 2]{AMS3}. 

The group $W_{\mf s^\vee}$ acts on $\Phi_{\mf s^\vee}$ and contains the Weyl group 
$W_{\mf s^\vee}^\circ$ of that root system. It admits a semidirect factorization
\[
W_{\mf s^\vee} = W_{\mf s^\vee}^\circ \rtimes \mf R_{\mf s^\vee} ,
\]
where $\mf R_{\mf s^\vee}$ is the stabilizer of a chosen basis of $\Phi_{\mf s^\vee}$.

Using the above identification of $T_{\mf s^\vee}$ with $\mf s_L^\vee$, we can reinterpret 
$\mc H (\mc R_{\mf s^\vee}, \lambda, \lambda^*, \vec{\mb z})$ as an algebra 
$\mc H (\mf s_L^\vee, W_{\mf s^\vee}^\circ, \lambda, \lambda^*,\vec{\mb z})$ 
whose underlying vector space is 
\[
\mc O (\mf s_L^\vee) \otimes \C [W_{\mf s^\vee}^\circ] \otimes \C [ \vec{\mb z}, \vec{\mb z}^{-1}].
\]
Here $\C [ \vec{\mb z}, \vec{\mb z}^{-1}]$ is a central subalgebra, generated by elements
$\mb z_j, \mb z_j^{-1} (j = 1, \ldots,d)$. The group $\mf R_{\mf s^\vee}$ acts naturally on the 
based root datum $\mc R_{\mf s^\vee}$, and hence on \\
$\mc H (\mf s_L^\vee, W_{\mf s^\vee}^\circ, \lambda, \lambda^*,\vec{\mb z})$ by algebra 
automorphisms \cite[Proposition 3.15.a]{AMS3}. From  \cite[Proposition 3.15.b]{AMS3} 
we get a 2-cocycle $\natural : \mf R_{\mf s^\vee}^2 \to \C^\times$ and a twisted group 
algebra $\C[\mf R_{\mf s^\vee},\kappa_{\mf s^\vee}]$. By definition, such a twisted group 
algebra has a vector space basis $\{N_r : r \in \mf R_{\mf s^\vee}\}$ 
and multiplication rule 
\[
N_r N_{r'} = \kappa_{\mf s^\vee}(r,r') T_{r r'} .
\]
Now we can define the twisted affine Hecke algebra
\begin{equation}\label{eq:2.4}
\mc H (\mf s^\vee, \vec{\mb z}) := \mc H (\mf s_L^\vee, W_{\mf s^\vee}^\circ, \lambda, 
\lambda^*, \vec{\mb z}) \rtimes \C[\mf R_{\mf s^\vee},\kappa_{\mf s^\vee}] .
\end{equation}
Up to isomorphism it depends only on $\mf s^\vee$ \cite[Lemma 3.16]{AMS3}.

The multiplication relations in $\mc H (\mf s^\vee, \vec{\mb z})$ are based on the Bernstein
presentation of affine Hecke algebras, let us make them explicit. The vector space
\[
\C [W_{\mf s^\vee}^\circ] \otimes \C [\vec{\mb z}, \vec{\mb z}^{-1}] \subset 
\mc H (\mf s^\vee, \vec{\mb z})
\]
is the Iwahori--Hecke algebra $\mc H (W_{\mf s^\vee}^\circ, \vec{\mb z}^{\, 2 \lambda})$, where 
$\vec{\mb z}^{\, \lambda} (\alpha) = \mb z_j^{\lambda (\alpha)}$ for the entry $\mb z_j$ of 
$\vec{\mb z}$ specified by $\alpha$. The conjugation action of $\mf R_{\mf s^\vee}$ on 
$W_{\mf s^\vee}^\circ$ induces an action on $\mc H (W_{\mf s^\vee}^\circ, \vec{\mb z}^{\, 2 \lambda})$.

The vector space $\mc O (\mf s_L^\vee) \otimes \C [ \vec{\mb z}, \vec{\mb z}^{-1}]$ is embedded in 
$\mc H (\mf s^\vee, \vec{\mb z})$ as
a maximal commutative subalgebra. The group $W_{\mf s^\vee}$ acts on it via its action
of $\mf s_L^\vee$, and every root $\alpha \in \Phi_{\mf s^\vee} \subset X^* (T_{\mf s^\vee})$ 
determines an element $\theta_\alpha \in \mc O (\mf s_L^\vee )^\times$, which does not
depend on the choice of the basepoint $\phi_L$ of $\mf s_L^\vee$ by \cite[Proposition 3.9.b]{AMS3}. 
For $f \in \mc O (\mf s_L^\vee)$ and a simple reflection $s_\alpha \in W_{\mf s^\vee}^\circ$ 
the following version of the Bernstein--Lusztig--Zelevinsky relation holds:
\[
f N_{s_\alpha} - N_{s_\alpha} s_\alpha (f) = \big( (\mb z_j^{\lambda (\alpha)} \! - 
\mb z_j^{-\lambda (\alpha)}) + \theta_{-\alpha} (\mb z_j^{\lambda^* (\alpha)} \! - 
\mb z_j^{-\lambda^* (\alpha)}) \big) (f - s_\alpha \cdot f) / (1 - \theta^{\, 2}_{-\alpha}) .
\]
Thus $\mc H (\mf s^\vee, \vec{\mb z})$ depends on the following objects:
$\mf s_L^\vee, W_{\mf s^\vee}$ and the simple reflections therein, the label functions
$\lambda, \lambda^*$ and the functions $\theta_\alpha : \mf s_L^\vee \to \C^\times$ for 
$\alpha \in \Phi_{\mf s^\vee}$. When $W_{\mf s^\vee} \neq W_{\mf s^\vee}^\circ$, 
we also need the 2-cocycle $\kappa_{\mf s^\vee}$ on $\mf R_{\mf s^\vee}$.

As in \cite[\S 3]{Lus-Gr}, the above relations entail that the centre of 
$\mc H (\mf s^\vee, \vec{v})$ is $\mc O (\mf s_L^\vee)^{W_{\mf s^\vee}}$. In other words,
the space of central characters for $\mc H (\mf s^\vee, \vec{v})$-representations
is $\mf s_L^\vee / W_{\mf s^\vee}$.

We note that when $\mf s^\vee$ is cuspidal,
\begin{equation}\label{eq:2.11}
\mc H (\mf s^\vee, \vec{\mb z}) = \mc O (\mf s^\vee)
\end{equation}
and every element of $\mf s^\vee$ determines a character of $\mc H (\mf s^\vee,\vec{\mb z})$.

For $\vec{z} \in (\C^\times)^d$ we let $\Irr_{\vec{z}} \big( \mc H (\mf s^\vee, \vec{\mb z}) \big)$
be the subset of $\Irr \big( \mc H (\mf s^\vee, \vec{\mb z}) \big)$ on which every $\mb z_j$ acts
as $z_j$. The main reason for introducing $\mc H (\mf s^\vee, \vec{\mb z})$ is the next result.
(See \cite[Definition 2.6]{AMS3} for the definition of tempered and essentially 
discrete series representations.)

\begin{thm}\label{thm:1.1} \textup{\cite[Theorem 3.18]{AMS3}} \\
Let $\mf s^\vee$ be an inertial equivalence class for $\Phi_e (G)$ and fix
parameters $\vec{z} \in \R_{>1}^d$. Then there exists a canonical bijection
\[
\begin{array}{ccc}
\Phi_e (G)^{\mf s^\vee} & \to & \Irr_{\vec{z}} \big( \mc H (\mf s^\vee, \vec{\mb z}) \big) \\
(\phi,\rho) & \mapsto & \bar M (\phi,\rho,\vec{z}) 
\end{array}
\]
with the following properties.
\begin{itemize}
\item $\bar M (\phi,\rho,\vec{z})$ is tempered if and only if $\phi$ is bounded.
\item $\phi$ is discrete if and only if $\bar M (\phi,\rho,\vec{z})$ is essentially
discrete series and the rank of $\Phi_{\mf s^\vee}$ equals 
$\dim_\C (T_{\mf s^\vee} / X_\nr (G))$.
\item The central character of $\bar M (\phi,\rho,\vec{z})$ is the product of
$\phi (\Fr)$ and a term depending only on $\vec{z}$ and a cocharacter associated to $u_\phi$.
\item Suppose that $\mb{Sc}(\phi,\rho) = (L, \chi_L \phi_L, \rho_L)$, where 
$\chi_L \in X_\nr (L)$. Then $\bar M (\phi,\rho,\vec{z})$ is a constituent of 
$\mr{ind}_{\mc H (\mf s^\vee_L, \vec{v})}^{\mc H (\mf s^\vee, \vec{z})} 
(L, \chi_L \phi_L,\rho_L)$.
\end{itemize}
\end{thm}

The irreducible module $\overline{M} (\phi,\rho,\vec{z})$ in Theorem \ref{thm:1.1}
is a quotient of a ``standard module" $\overline{E} (\phi,\rho,\vec{z})$, also studied
in \cite[Theorem 3.18]{AMS3}. By \cite[Lemma 3.19.a]{AMS3} every such standard module
is a direct summand of a module obtained by induction from a standard module associated
to a discrete enhanced L-parameter for a Levi subgroup of $G$.\\

Suppose that $(\phi_b, q\epsilon)$ is a bounded cuspidal L-parameter for a Levi subgroup
$L = \mc L (F)$ of $G$. (It can be related to the above be requiring that $\phi_b = t \phi_L$
with $t \in X_\nr (L)$ and $q\cE = \rho_L$.) In \cite[\S 3.1]{AMS3} we associated to 
$(G,L,\phi_b ,q\epsilon)$ a twisted graded Hecke algebra
\begin{equation}\label{eq:2.24}
\mh H (\phi_b ,q\epsilon, \vec{\mb r}) = \mh H (G^\vee_{\phi_b},M^\vee, q \cE, \vec{\mb r})
\otimes \mc O (X_\nr (G)) .
\end{equation}
Let us describe this algebra in some detail. Firstly, $M^\vee = L^\vee \cap G^\vee_{\phi_b}$ 
is a quasi-Levi subgroup of $G^\vee_{\phi_b}$ and
\begin{equation}\label{eq:1.2}
Z_{G^\vee \rtimes \mb W_F}(Z(M^\vee)^\circ) = L^\vee \rtimes \mb W_F ,\qquad
Z(L^\vee_c )^{\mb W_F,\circ} = Z(M^\vee)^\circ . 
\end{equation}
From $u_\phi$ we get a unipotent class in $(M^\vee )^\circ$, and $q\cE$ denotes the canonical
extension of $q\epsilon \in \Irr \big( \pi_0 (Z_{M^\vee}(u_\phi)) \big)$ to an $M^\vee$-equivariant
cuspidal local system on the $M^\vee$-conjugacy class of $u_\phi$. We write 
\[
\mf t = \text{Lie}(Z(M^\vee)^\circ) = X_* (Z(M^\vee)^\circ) \otimes_\Z \C 
\quad \text{and} \quad \mf t_\R =  X_* (Z(M^\vee)^\circ) \otimes_\Z \R .
\]
The algebra \eqref{eq:2.24} comes with a root system $R_{q\cE} =  
R \big( (G^\vee_{\phi_b})^\circ, Z(M^\vee)^\circ \big)$ with Weyl group 
$W_{q\cE}^\circ = N_{(G^\vee_{\phi_b})^\circ} ((M^\vee)^\circ) / (M^\vee)^\circ$. 
It is a normal subgroup of the finite group $W_{q\cE} = N_{G^\vee_{\phi_b}}(M^\vee, q\cE) / 
M^\vee$. There exists a subgroup $\mf R_{q\cE} \subset W_{q\cE}$ such that
\begin{equation}\label{eq:2.23}
W_{q\cE} = W^\circ_{q\cE} \rtimes \mf R_{q\cE} .
\end{equation} 
If $(\phi_b,q\epsilon) \in \mf s_L^\vee$, then $W^\circ_{q\cE} \subset W_{\mf s^\vee}^\circ$ 
and $W_{q\cE} \subset W_{\mf s^\vee}$ are the subgroups stabilizing $\phi_b$.

The array of complex parameters $\mb{\vec r}$ yields a function on $R_{q\cE}$, which is
constant on irreducible components. As vector spaces
\begin{equation}\label{eq:2.22}
\mh H (\phi_b,q\epsilon, \mb{\vec r}) = \C [\mf R_{q\cE}] \otimes \C [W^\circ_{q\cE}] 
\otimes S (\mf t^*) \otimes \C[\vec{\mb r}] \otimes S (\mr{Lie}^* (X_\nr (G))) .
\end{equation}
Here the first tensor factor on the right hand side is embedded in 
$\mh H (\phi_b,q\epsilon, \vec{\mb r})$ as a twisted group algebra 
$\C [\mf R_{q\cE},\natural_{q\cE}]$, the second and third factors are embedded as
subalgebras, while the fourth and fifth tensor factors are central subalgebras.
The cross relations between $S(\mf t^*)$ and $\C[W^\circ_{q\cE}]$ are those of a standard 
graded Hecke algebra \cite[\S 4]{Lus-Gr}, for a simple reflection $s_\alpha$:
\begin{equation}\label{eq:2.25}
f T_{s_\alpha} - T_{s_\alpha} s_\alpha (f) = \mb r_j (f - s_\alpha (f)) \alpha^{-1} 
\qquad f \in \mc O ( \mf t) ,
\end{equation}
where $\alpha$ lies in the component of the root system 
$R \big( (G^\vee_{\phi_b})^\circ, Z(M^\vee)^\circ \big)$ labelled by the $j$-th entry of 
$\vec{\mb r}$. For $r \in \mf R_{q\cE}$:
\begin{equation}\label{eq:2.26}
T_r f T_r^{-1} = r (f) \qquad f \in \mc O ( \mf t). 
\end{equation}
Finally, inside $\mh H (\phi_b,q\epsilon, \vec{\mb r})$ there is an identification
\begin{equation}\label{eq:2.27}
\C [\mf R_{q\cE},\natural_{q\cE}] \otimes \C [W^\circ_{q\cE}] \cong
\C [W_{q\cE},\natural_{q\cE}] ,
\end{equation}
where the 2-cocycle $\natural_{q\cE}$ is lifted to $W_{q\cE}^2 \to \C^\times$ via \eqref{eq:2.23}.

We write $X_\nr (L)_{rs} = \Hom (L, \R_{>0})$, the real split part of the complex torus $X_\nr (L)$.
When $\mb{Sc}(\phi,\rho) \in X_\nr (L)_{rs} \cdot (\phi_b,q\epsilon)$, we defined in \cite[p. 37]{AMS3}
a ``standard" $\mh H (\phi_b,q\epsilon, \vec{\mb r})$-module $E(\phi,\rho,\vec{r})$. It has one
particular irreducible quotient called $M(\phi,\rho,\vec{r})$.
The main feature of \eqref{eq:2.24} is

\begin{thm}\label{thm:1.2} \textup{\cite[Theorem 3.8]{AMS3}} \\ 
Fix $\vec r \in \C^d$. The map $(\phi,\rho) \mapsto M(\phi,\rho,\vec{r})$ is natural bijection between:
\begin{itemize}
\item $\{ (\phi,\rho) \in \Phi_e (G) : \mb{Sc}(\phi,\rho) \in X_\nr (L)_{rs} \cdot (\phi_b,q\epsilon) \}$
\item $\big\{ \pi \in \Irr \big( \mh H (\phi_b,q\epsilon, \vec{\mb r}) \big) : \vec{\mb r}$ acts as
$\vec{r}$ and all $\mc O ( \mf t \times \mr{Lie}(X_\nr (G)))$-weights of $\pi$ are contained in 
$\mf t_\R \times \mr{Lie} (X_\nr (G)_{rs}) \big\}$.
\end{itemize}
\end{thm}

The algebras \eqref{eq:2.24} and Theorem \ref{thm:1.2} play an important role in the proof of 
Theorem \ref{thm:1.1}. Namely, $\mc H (\mf s^\vee,\vec{\mb z}) = 
\mc H (L,X_\nr (L) \phi_b, q\epsilon,\vec{\mb z})$ is ``glued" from
the algebras $\mh H (t \phi_b ,q\cE,\vec{\mb r})$ with $t \in X_\nr (L)$ unitary. In particular
$W_{q\cE} \subset W_{\mf s^\vee}$ for every $t \phi_b$, and the 2-cocycle $\natural_{q\cE}$ for
$\phi_L$ (not necessarily for $\phi_b$) is the restriction of $\kappa_{\mf s^\vee}$ to $W_{q\cE}$.

The modules $\overline{M} (\phi,\rho,\vec{z})$ and $\overline{E} (\phi,\rho,\vec{z})$ mentioned 
in and after Theorem \ref{thm:1.1} are obtained from, respectively, $M(\phi,\rho,\log \vec{z})$ 
and $E(\phi,\rho,\log \vec{z})$ by unravelling the glueing procedure 
(see \cite[Theorem 2.5 and 2.9]{AMS3}).
\vspace{4mm}

\section{Automorphisms from $G_\ad / G$}
\label{sec:2}

Let $\mc G_\ad$ be the adjoint group of $\mc G$. 
The conjugation action of $G_\ad$ on $G$ induces an action of $G_\ad$
on $\Irr (G)$. Although this action comes from conjugation in $\mc G (\overline F)$,
that is not necessarily conjugation in $\mc G (F)$, and therefore $G_\ad$ can
permute $\Irr (G)$ nontrivially. When $G$ is quasi-split, it is expected that this 
replaces a generic (with respect to a certain Whittaker datum) member of an L-packet 
by a member which is generic with respect to another Whittaker datum \cite[(1.1)]{Kal}.

Although $\mc G \to \mc G_\ad$ is an epimorphism of algebraic groups, 
the map on $F$-rational points,
\[
G = \mc G (F) \to \mc G_\ad (F) = G_\ad ,
\]
need not be surjective. More precisely, the machine of Galois cohomology yields an exact
sequence
\begin{equation}\label{eq:2.5}
1 \to Z(\mc G)(F) \to \mc G (F) \to \mc G_\ad (F) \to H^1 (F,Z(\mc G)) .
\end{equation}
We abbreviate the group
\[
\mc G_\ad (F) / \mr{im}(\mc G (F) \to \mc G_\ad (F)) \text{ to } G_\ad / G .
\]
Notice that the action of $G_\ad$ on $\Irr (G)$ factors through $G_\ad / G$.
Let $\mc T_{AD} = \mc T / Z(\mc G)$ be the image of $\mc T$ in $\mc G_\ad$. 
It is known \cite[Lemma 16.3.6]{Spr} that the group
\begin{equation}\label{eq:2.33}
T_{AD} / T := \mc T_{AD}(F) / \mr{im}(\mc T(F) \to \mc T_{AD}(F))
\end{equation}
is naturally isomorphic to $G_\ad / G$. Since $\mc T$ centralizes the 
maximal $F$-split torus $\mc S$ of $\mc G$, $\mc T$ is contained in any standard
(w.r.t. $\mc S$) Levi $F$-subgroup of $\mc G$. Consequently $G_\ad / G$ can
be represented by elements that lie in every standard Levi $F$-subgroup of $\mc G_\ad$.

When $G$ is quasi-split, $G_\ad / G$ acts simply transitively on the set of $G$-orbits
of Whittaker data for $G$ \cite{Kal}. For general $G$ we are not aware of such a precise
characterization, but we do note that $G_\ad / G$ acts naturally on the collection of
$G$-orbits of vertices in the Bruhat--Tits building $\mc B (\mc G,F)$. From the 
classification of simple $p$-adic groups \cite{Tit} one can deduce that this action of 
$G_\ad / G$ is transitive on the $G$-orbits of hyperspecial vertices and on the 
$G$-orbits of special, non-hyperspecial vertices.

Denote the (unique) maximal compact subgroup of $T$ by $T_\cpt$,
and let $X_* (T)$ be its latttice of $F$-rational cocharacters.
The fixed uniformizer $\varpi_F$ of $F$ determines a group isomorphism
\begin{equation}\label{eq:2.51}
\begin{array}{ccc}
T_\cpt \times X_* (T) & \to & T \\
(t,\lambda) & \mapsto & t \lambda (\varpi_F) 
\end{array}.
\end{equation}
The same goes for $T_{AD}$, so
\begin{multline}\label{eq:2.53}
T_{AD} / T := T_{AD} / \mr{im}(T \to T_{AD}) \cong \\
T_{AD,\cpt} / \mr{im}(T_\cpt \to T_{AD}) 
\times X_* (T_{AD}) / \mr{im}(X_* (T) \to X_* (T_{AD})) .
\end{multline}
Accordingly, we can write any $g \in T_{AD}$ as $g = g_c g_x$ with a compact part
$g_c \in T_{AD,\cpt}$ and an "unramified" part $g_x \in X_* (T_{AD})$. 

Then $g_c$ fixes an apartment of $\mc B (\mc G,F)$ pointwise, so it stabilizes relevant
vertices in that building. Still, $g_c$ may permute Bernstein components of $\Irr (G)$.
On the other hand, we expect that the action of $g_x$ on $\Irr (G)$ stabilizes all
Bernstein components.

\subsection{Action on enhanced L-parameters} \
\label{par:2.1}

If one believes in the local Langlands correspondence, then the action of $G_\ad / G$
on $\Irr (G)$ should be reflected in an action of $G_\ad / G$ on $\Phi (G)$, and even on
$\Phi_e (G)$. But $G_\ad / G$ does not act in any interesting way on $G^\vee$. Indeed,
representing $G_\ad / G$ inside $T_{AD}$, it fixes the root datum of $(\mc G,\mc T)$,
so it should also fix the root datum of $(\mc G^\vee, \mc T^\vee)$. Any automorphism of
$\mc G^\vee$ fixing that root datum is inner, and L-parameters are only considered up to
$G^\vee$-conjugation. Therefore the only reasonable action $G_\ad / G$ on $\Phi (G)$
is the trivial action.
As the action on $\Irr (G)$ can be nontrivial, the desired functoriality in the LLC tells
us that $G_\ad / G$ should act on $\Phi_e (G)$ by fixing L-parameters and permuting their
enhancements. 

A way to achieve the above in general can be found in \cite[\S 3]{Kal} and \cite[\S 4]{Cho}. 
(We formulate this only for non-archimedean local fields, but apart from Lemmas \ref{lem:2.1}.b
and \ref{lem:2.3} all the results in this paragraph are just as well valid over $\R$ and $\C$.)

Fix a L-parameter
\[
\phi : \mb W_F \times SL_2 (\C) \to G^\vee \rtimes \mb W_F
\]
and let $\mb W_F$ act on $G^\vee$ via $\phi$ and conjugation in $G^\vee \rtimes \mb W_F$.
Similarly $w \mapsto \mr{Ad}(\phi (w))$ defines an action of $\mb W_F$ on $G^\vee_\Sc$.
Consider the short exact sequences of $\mb W_F$-modules
\begin{equation}\label{eq:2.6}
\begin{array}{c@{\;\to\;}c@{\;\to\;}c@{\;\to\;}c@{\;\to\;}c}
1 & Z (G^\vee) & G^\vee & {G^\vee}_\ad & 1 ,\\
1 & Z({G^\vee}_\Sc) & {G^\vee}_\Sc & {G^\vee}_\ad & 1 .
\end{array}
\end{equation}
They induce exact sequences in Galois cohomology:
\begin{equation}\label{eq:2.7}
\begin{array}{c@{\;\to\;}c@{\;\to\;}c@{\;\to\;}c}
H^0 (\mb W_F, Z (G^\vee)) & H^0 (\mb W_F ,G^\vee) & H^0 (\mb W_F ,{G^\vee}_\ad) &
H^1 (\mb W_F, Z(G^\vee)) ,\\
H^0 (\mb W_F, Z({G^\vee}_\Sc )) & H^0(\mb W_F, {G^\vee}_\Sc) & H^0 (\mb W_F ,{G^\vee}_\ad) & 
H^1 (\mb W_F , Z({G^\vee}_\Sc )) .
\end{array}
\end{equation}
Splicing these, we get a map
\begin{equation}\label{eq:2.8}
H^0 (\mb W_F,G^\vee) = (G^\vee )^{\phi (\mb W_F)} \to H^1 (\mb W_F , Z({G^\vee}_\Sc)) 
\end{equation}
which factors via $H^0 (\mb W_F,{G^\vee}_\ad)$. Sometimes we will make use of the explicit
construction of this map, which involves the connecting maps in Galois cohomology. Namely,
choose a map $\phi_\Sc : \mb W_F \to {G^\vee}_\Sc \rtimes \mb W_F$ which lifts $\phi |_{\mb W_F}$. 
For $h \in (G^\vee )^{\phi (\mb W_F)}$
\begin{equation}\label{eq:2.16}
\mb W_F \to Z({G^\vee}_\Sc) : \gamma \mapsto h \phi_\Sc (\gamma) h^{-1} \phi_\Sc (\gamma)^{-1} 
\end{equation}
is well-defined (since the conjugation action of ${G^\vee}_\Sc \rtimes \mb W_F$ on itself 
descends to an action of $G^\vee$). This determines an element 
$c_h \in H^1 (\mb W_F , Z({G^\vee}_\Sc))$, which does not depend on the choice of the lift 
of $\phi$ and is the image of $h$ under \eqref{eq:2.8}.

By \eqref{eq:2.7} the postcomposition of 
\eqref{eq:2.8} with the natural map $H^1 (\mb W_F , Z({G^\vee}_\Sc)) \to H^1 (\mb W_F , Z(G^\vee))$ 
is trivial, as is the precomposition of \eqref{eq:2.8} with $H^0 (\mb W_F, Z({G^\vee}_\Sc)) \to
H^0 (\mb W_F,G^\vee)$. This enables us to rewrite \eqref{eq:2.8} as 
\begin{multline*}
(G^\vee )^{\phi (\mb W_F)} \to
(G^\vee )^{\phi (\mb W_F)} \big/ \mr{im} \big( ({G^\vee}_\Sc)^{\phi (\mb W_F)} \to 
(G^\vee )^{\phi (\mb W_F)} \big) \\
\to \ker \big( H^1 (\mb W_F , Z({G^\vee}_\Sc)) \to H^1 (\mb W_F , Z(G^\vee)) \big) .
\end{multline*}
Recall the natural homomorphism \cite[\S 10.2]{Bor} 
\begin{equation}\label{eq:2.12}
\begin{array}{ccc}
H^1 (\mb W_F, Z(G^\vee)) & \longrightarrow & \Hom (G,\C^\times) \\
c & \mapsto & \big( g \mapsto \langle g, c \rangle \big)
\end{array} .
\end{equation}
On a semisimple element $g \in G$, one can evaluate $\pi_c$ by regarding $g$ as an element of a 
torus $T' \subset G$, $c$ as a L-parameter for $T'$ and applying the LLC for tori.
From \eqref{eq:2.12} for $G_\ad$ and $G$, we get a natural homomorphism \cite[Lemma 3.1]{Kal}
\begin{equation}\label{eq:2.13}
\ker \big( H^1 (\mb W_F , Z({G^\vee}_\Sc)) \to H^1 (\mb W_F , Z(G^\vee)) \big) \longrightarrow
\Hom (G_\ad / G, \C^\times ) .
\end{equation}
Thus \eqref{eq:2.8} can also be regarded as a homomorphism
\begin{equation}\label{eq:2.14}
\begin{array}{ccc}
(G^\vee )^{\phi (\mb W_F)} \big/ \mr{im} \big( ({G^\vee}_\Sc)^{\phi (\mb W_F)} \to
(G^\vee )^{\phi (\mb W_F)} \big) & \longrightarrow & \Hom (G_\ad / G, \C^\times ) \\
h & \mapsto & \big( g \mapsto \langle g, c_h \rangle \big)
\end{array}.
\end{equation}
By duality, we get a natural homomorphism
\begin{equation}\label{eq:2.15}
\tau_{\phi,\mc G} : G_\ad / G \to \Hom \big( (G^\vee )^{\phi (\mb W_F)}, \C^\times \big) ,
\end{equation}
which can be expressed as $\tau_{\phi,\mc G}(g)(h) = \langle g, c_h \rangle$.

\begin{lem}\label{lem:2.1}
\enuma{
\item The image of $\tau_{\phi,\mc G}$ consists of characters of \\
$Z_{G^\vee}(\phi (\mb W_F)) = (G^\vee )^{\mb W_F}$ that are trivial on:
\begin{itemize}
\item $Z(G^\vee)^{\mb W_F}$,
\item the image of $({G^\vee}_\Sc)^{\phi (\mb W_F)} \to (G^\vee )^{\phi (\mb W_F)}$,
\item the identity component $Z_{G^\vee}(\phi (\mb W_F))^\circ$. 
\end{itemize}
\item Let $t \mapsto z_t$ be a continuous map
\[
[0,1] \to X_\nr \cong Z (G^\vee)^{\mb I_F})^\circ_{\mb W_F} ,
\]
and consider the path of L-parameters $t \mapsto z_t \phi$.
For every $g \in G_\ad / G$ the characters $\tau_{z_0 \phi, \mc G}(g)$ and 
$\tau_{z_1 \phi, \mc G}(g)$ agree on $\cap_{t \in [0,1]} (G^\vee)^{z_t \phi (\mb W_F)}$.
}
\end{lem}
\begin{proof}
(a) By the exactness of \eqref{eq:2.7}, $Z(G^\vee)^{\mb W_F}$ lies in the kernel of
\eqref{eq:2.14}. It is also clear \eqref{eq:2.14} that the image of 
$({G^\vee}_\Sc)^{\phi (\mb W_F)} \to (G^\vee )^{\phi (\mb W_F)}$ lies in the kernel of
$\tau_{\phi,\mc G}(g)$, for every $g \in G_\ad / G$.

The group $G_\ad / G \cong T_{AD} / T$ is totally disconnected and compact, so in 
\eqref{eq:2.7} the image consists of characters $G_\ad / G \to \C^\times$ with finite
image. Hence in \eqref{eq:2.8} the image consists of continuous characters from
$Z_{G^\vee}(\phi (\mb W_F))$ to the group of finite order elements in $\C^\times$.
The latter group is totally disconnected, so every such character
$Z_{G^\vee}(\phi (\mb W_F)) \to \C^\times$ is trivial on the identity component 
$Z_{G^\vee}(\phi (\mb W_F))^\circ$.\\
(b) Let $a \in \cap_{t \in [0,1]} (G^\vee)^{z_t \phi (\mb W_F)}$ and let 
$c_{a_t} \in H^1 (\mb W_F,Z( {G^\vee}_\Sc ))$ be its image under \eqref{eq:2.8}
for the L-parameter $z_t \phi$. From \eqref{eq:2.16} we see that, for any $w \in \mb W_F$,
\[
[0,1] \to Z({G^\vee}_\Sc) : t \mapsto c_{a_t} (w)
\]
is a continuous map. Since $Z({G^\vee}_\Sc)$ is finite, $c_{a_t}(w)$ does not depend on 
$t \in [0,1]$, and hence $c_{a_0} = c_{a_1} \in H^1 (\mb W_F, Z({G^\vee}_\Sc))$. In view of
\eqref{eq:2.14} and \eqref{eq:2.15}, this implies $\tau_{z_0 \phi, \mc G}(g) (a) =
\tau_{z_1 \phi, \mc G}(g) (a)$.
\end{proof}

Now we construct an action of $G_\ad / G$ on the set of enhancements of $\phi$. Recall that
$G^\vee_\phi =  Z^1_{{G^\vee}_\Sc}(\phi |_{\mb W_F})$ is the inverse image of 
$(G^\vee)^{\phi (\mb W_F)} / Z(G^\vee )^{\mb W_F}$ in ${G^\vee}_\Sc$. By Lemma \ref{lem:2.1}.a 
$\tau_{\phi,\mc G}(g)$ can be lifted to a character of $G^\vee_\phi$, which is trivial on 
$(G^\vee_\phi )^\circ$ and on $Z({G^\vee}_\Sc)$. In this
way $\tau_{\phi,\mc G}(g)$ naturally determines a character $\tau_{\mc S_\phi}(g)$ of
\begin{equation}\label{eq:2.17}
\mc S_\phi \cong \pi_0 \big( Z_{G^\vee_\phi}(u_\phi) \big) = 
Z_{G^\vee_\phi}(u_\phi) / Z_{G^\vee_\phi}(u_\phi )^\circ .
\end{equation}
Then we can define an action $\tau_{\mc G}$ of $G_\ad / G$ on enhanced L-parameters by
\begin{equation}\label{eq:2.20}
\tau_{\mc G}(g) (\phi,\rho) = (\phi, \rho \otimes \tau_{\mc S_\phi}(g)) .
\end{equation}
We note that $\tau_{\mc S_\phi}(g)$ factors through 
\[
\mc S_\phi / Z({G^\vee}_\Sc) \cong \pi_0 ( Z_{{G^\vee}_\ad} (\phi) ) ,
\]
the component group for $\phi$ as L-parameter for the quasi-split inner form of $\mc G (F)$.
Hence the action \eqref{eq:2.20} preserves the $Z({G^\vee}_\Sc)$-character of any enhancement,
which, as explained after \eqref{eq:1.1}, means that the action of $G_\ad / G$ stabilizes the
set $\Phi_e (G)$ of $G$-relevant enhanced L-parameters.

Next we will investigate the compatibility of the characters $\tau_{\phi,\mc G}(g)$ with Levi
subgroups and with the cuspidal support map for enhanced L-parameters. Let $\mc L$ be a Levi
$F$-subgroup of $\mc G$, such that the image of $\phi$ is contained in $L^\vee \rtimes \mb W_F$.
Write $\mc L_\ad = \mc L / Z(\mc L)$ and $\mc L_{AD} = \mc L / Z(\mc G)$, so that $\mc L_\ad$
is a quotient of $\mc L_{AD}$ and $L_{AD} := \mc L_{AD}(F)$ maps naturally to $L_\ad :=
\mc L_\ad (F)$.
Then \eqref{eq:2.15} for $\mc L$ gives 
\[
\tau_{\phi, \mc L} : L_\ad / L \to \Hom ( Z_{L^\vee}(\phi (\mb W_F)), \C^\times ) .
\]
Hence there is a canonical homomorphism
\begin{equation}\label{eq:2.18}
G_\ad / G \cong L_{AD} \big/ \big( \mc L (F) / Z(\mc G)(F) \big) 
\longrightarrow L_\ad \big/ \big( \mc L (F) / Z(\mc L)(F) \big) = L_\ad / L .
\end{equation}

\begin{lem}\label{lem:2.2}
Let $g \in G_\ad / G$ and let $g_{\mc L} \in L_\ad / L$ be its image under \eqref{eq:2.18}.
Then
\[
\tau_{\phi,\mc L}(g_{\mc L}) = \tau_{\phi,\mc G}(g) \big|_{Z_{L^\vee}(\phi (\mb W_F))} .
\]
\end{lem}
\begin{proof}
We write
\begin{equation}
L_c^\vee = \text{ inverse image of } L^\vee \text{ in } {G^\vee}_\Sc .
\end{equation}
Consider the short exact sequence of $\mb W_F$-modules
\[
1 \to Z({G^\vee}_\Sc) \to L^\vee_c \to L^\vee / Z(G^\vee) = L^\vee_c / Z({G^\vee}_\Sc) \to 1 .
\]
That and \eqref{eq:2.7} give rise to a diagram
\begin{equation}\label{eq:2.19}
\xymatrix{
H^0 (\mb W_F, L^\vee) \ar[d] \ar[dr] & H^0 (\mb W_F, L^\vee_c) \ar[d] \ar[r] &
H^0 (\mb W_F, {L^\vee}_\Sc) \ar[d] \\
H^0 (\mb W_F, G^\vee) \ar[d] & H^0 (\mb W_F, L^\vee / Z(G^\vee)) \ar[r] \ar[dl] \ar[d] &
H^0 (\mb W_F, {L^\vee}_\ad) \ar[d] \\
H^0 (\mb W_F, {G^\vee}_\ad) \ar[r] & H^1 (\mb W_F, Z({G^\vee}_\Sc)) \ar[r] & 
H^1 (\mb W_F, Z({L^\vee}_\Sc))
}
\end{equation}
The middle column of \eqref{eq:2.19} is a subobject of \eqref{eq:2.7}, so the left and middle
column form a commutative diagram. By the functoriality of Galois cohomology with respect to
morphisms of short exact sequences, the middle and right columns of the diagram
also commute, so the entire diagram is commutative.

The character $\tau_{\phi,\mc L}(l)$ corresponds to first the path 
$(1,1) \to (2,2) \to (2,3) \to (3,3)$ in the diagram, and then pairing with $l \in L_\ad / L$,
using \eqref{eq:2.13}. Since $g_{\mc L}$ comes form $g \in G_\ad / G$, for $l = g_{\mc L}$ the
pairing with $H^1 (\mb W_F, Z({L^\vee}_\Sc))$ can also be realized as retraction along
$H^1 (\mb W_F, Z({G^\vee}_\Sc)) \to H^1 (\mb W_F, Z({L^\vee}_\Sc))$ and then pairing with $g$.
In effect, this means that $\tau_{\phi,\mc L}(g_{\mc L})$ can be constructed as following the
path $(1,1) \to (2,2) \to (3,2)$ in \eqref{eq:2.19}, and then applying \eqref{eq:2.13} with 
input $g$. 

On the other hand, $\tau_{\phi,\mc G}(g)$ can be seen in the diagram as the path 
$(2,1) \to (3,1) \to (3,2)$ and then pairing with $g$. Restricting $\tau_{\phi,\mc G}(g)$
to $Z_{L^\vee}(\phi (\mb W_F))$ means that we should start at position (1,1), map to position
(2,1) and then perform $\tau_{\phi, \mc G}(g)$. By the commutativity of the diagram 
\eqref{eq:2.19}, that is the same as the above procedure for $\tau_{\phi,\mc L}(g_{\mc L})$. 
\end{proof}

In view of Lemma \ref{lem:2.2} we may abbreviate $\tau_{\mc \phi,\mc G}(g),
\tau_{\phi,\mc L}(g_{\mc L})$ and $\tau_{\mc S_\phi}(g)$ to $\tau_\phi (g)$.

Suppose that $(\phi,\rho) \in \Phi_e (G)$ has cuspidal support
\[
\mb{Sc}(\phi,\rho) = (\mc L (F), \phi_v, q\epsilon) .
\]
For $g \in G_\ad / G$ we have the characters $\tau_{\mc S_\phi}(g) = \tau_\phi (g)$ and 
$\tau_{\mc S_{\phi_v}}(g_{\mc L}) = \tau_{\mc S_{\phi_v}}(g) = \tau_{\phi_v}(g)$.

\begin{lem}\label{lem:2.3}
$\mb{Sc}(\phi, \rho \otimes \tau_{\phi} (g)) = 
(\mc L (F), \phi_v, q\epsilon \otimes \tau_{\phi_v}(g))$.
\end{lem}
\begin{proof}
Let us recall some aspects of the construction of the cuspidal support map
from \cite[\S 7]{AMS1}. Firstly, $(\phi,\rho)$ is determined up to $G^\vee$-conjugacy
by $\phi |_{\mb W_F}, u_\phi$ and $\rho$. Here $u_\phi = \phi \big( 1, 
\big( \begin{smallmatrix} 1 & 1 \\ 0 & 1 \end{smallmatrix} \big)\big) \in G^\vee_\phi$
and $\rho \in \Irr (\mc S_\phi) = \Irr \big( \pi_0 (Z_{G^\vee_\phi} (u_\phi)) \big)$.
The cuspidal quasi-support of $(u_\phi,\rho)$, as defined in \cite[\S 5]{AMS1} with respect 
to the complex reductive group $G^\vee_\phi$, is of the form 
$(M^\vee,v,q\epsilon)$, where $M^\vee$ is a quasi-Levi subgroup of 
$G^\vee_\phi = Z^1_{{G^\vee}_\Sc}(\phi |_{\mb W_F})$ and $(v,q\epsilon)$ is a cuspidal
unipotent pair for $M^\vee$. Then $L^\vee \rtimes \mb W_F = Z_{G^\vee \rtimes \mb W_F}(
Z(M^\vee)^\circ)$. 

Recall that $\tau_{\mc S_\phi}(g)$ extends to the character $\tau_{\phi}(g)$ defined on the
whole of $G^\vee_\phi$. Let $\rho^\circ$ be an irreducible constituent of the restriction of
$\rho$ to
\begin{equation}\label{eq:2.30}
\mc S^\circ_\phi := \pi_0 ( Z_{(G^\vee_\phi )^\circ} (u_\phi)) . 
\end{equation} 
Then $(v,\epsilon)$ is a refinement of the cuspidal support $\mb{Sc}(u_\phi,\rho^\circ) = 
(v,\epsilon)$, with respect to the connected reductive group $(G^\vee_\phi)^\circ$ and as 
defined in \cite{LusCusp}. The characterization of the cuspidal support, as in 
\cite[\S 6.2]{LusCusp} and \cite[\S 0.4]{LuSp}, shows immediately that 
\[
\mb{Sc}(u_\phi, \rho^\circ \otimes \tau_\phi (g) ) = (v, \epsilon \otimes \tau_\phi (g)) ,
\]
where $\tau_\phi (g)$ is restricted, respectively, to the domain of $\rho^\circ$ and the
domain of $\epsilon$. This implies in particular that the unipotent element (defined up to
conjugacy) in $\mb{Sc}(u_\phi, \tau_\phi (g) \otimes \rho)$ is $v$. 
Then the characterization of the cuspidal quasi-support in \cite[(64)]{AMS1} shows that
\begin{equation}\label{eq:2.21}
\mb{Sc}(u_\phi, \rho \otimes \tau_\phi (g)) = (v, q\epsilon \otimes \tau_\phi (g)) .
\end{equation}
By Lemma \ref{lem:2.2} the right hand side can be written as 
$(v, q\epsilon \otimes \tau_{\phi,\mc L} (g_{\mc L}) )$. The cuspidal support of $(\phi,\rho)$
is built from $(\mc L (F), \phi |_{\mb W_F}, v, q\epsilon)$ by adjusting $\phi |_{\mb W_F}$
with a certain representation of $\mb W_F / \mb I_F$ \cite[Definition 7.7]{AMS1}. This
representation arises as the composition of the norm on $\mb W_F$ and a cocharacter of 
$Z (M^\vee)^\circ$ \cite[Lemma 7.6]{AMS1}. That cocharacter depends only on $\phi |_{\mb W_F}$
and $v$, so it is the same for $(\phi,\rho)$ as for $(\phi, \rho \otimes \tau_{\phi}(g))$.
Combine this with \eqref{eq:2.21}.
\end{proof}

\subsection{Action on graded Hecke algebras} \
\label{par:2.2}

\begin{prop}\label{prop:2.4}
Let $g \in G_\ad / G$ and let $(\phi_b,q\epsilon) \in \Phi_\cusp (L)$ be bounded. 
The character $\tau_{\phi_b}(g)$ induces an algebra isomorphism
\[
\alpha_g : \mh H (\phi_b ,q\epsilon, \vec{\mb r}) \to 
\mh H (\phi_b ,q\epsilon \otimes \tau_{\phi_b} (g), \vec{\mb r})
\]
such that, in the notation of \eqref{eq:2.22}:
\begin{itemize}
\item $\alpha_g$ is the identity on $\C [W^\circ_{q\cE}] \otimes S (\mf t^*) \otimes 
S (\mr{Lie}^* (X_\nr (G)))$;
\item there exists a set of representatives $\dot{\mf R_{q\cE}} \subset N_{G^\vee_{\phi_b}}(M^\vee,
q\cE) / (M^\vee)^\circ$ for $\mf R_{q\cE}$, such that 
$\alpha_g (T_r) = \tau_{\phi_b} (g) (\dot r) T_r$ for all $r \in \mf R_{q\cE}$.
\end{itemize}
\end{prop}
\begin{proof}
Notice that $N_{G^\vee_{\phi_b}}(M^\vee,q\cE) = N_{G^\vee_{\phi_b}}(M^\vee,q\cE \otimes \tau_{\phi_b} (g))$
because $\tau_{\phi_b} (g)$ extends to a character of $G^\vee_{\phi_b}$. Hence $W_{q\cE} = 
W_{q\cE \otimes \tau_{\phi_b} (g)^{-1}}$ and our asserted maps are well-defined linear bijections.

Regarding the first bullet as given, we focus on the second bullet. If we can establish that one, the 
multiplication rules \eqref{eq:2.26} and \eqref{eq:2.27} immediately show that $\alpha_g$ is an
algebra homomorphisms. 

The algebra $\C [W_{q\cE},\natural_{q\cE}]$ is realized in $\mh H (\phi_b ,q\epsilon, \vec{\mb r})$
as the endomorphism algebra of a $G^\vee_{\phi_b}$-equivariant perverse sheaf $q\pi_* (\widetilde{q\cE})$
on a complex algebraic variety \cite[Lemma 5.4]{AMS1}. To construct the action on 
$q\pi_* (\widetilde{q\cE})$, several steps are needed:
\begin{itemize}
\item Choose an irreducible constituent $\epsilon$ of $q\epsilon |_{Z_{(G^\vee_{\phi_b})^\circ}(u_\phi)}$.
(It is unique up to conjugacy.)
\item Let $\cE$ be the $(M^\vee)^\circ$-equivariant local system associated to $\epsilon$ on the 
unipotent orbit $\mc C_{u_\phi}^{(M^\vee)^\circ}$, and let $M^\vee_\cE$ be its stabilizer $M^\vee$. 
\item The generalized Springer correspondence for $(M^\vee,\cE)$ matches
$(\mc C_{u_\phi}^{(M^\vee)^\circ},q\cE)$ with an irreducible representation $\rho_{M^\vee}$ of
$\C [M^\vee_\cE / (M^\vee )^\circ, \natural_\cE]$.
\item By the equivariance of the generalized Springer correspondence $N_{G^\vee_{\phi_b}}(M^\vee,q\cE)$
equals the stabilizer of $(M^\vee, \rho_{M^\vee})$ in $G^\vee_{\phi_b}$. 
\item Write $\widetilde{W_{q\cE}} = N_{G^\vee_{\phi_b}}(M^\vee,q\cE) / (M^\vee)^\circ$, so that
$\rho_{M^\vee}$ can be extended to a projective representation of $\widetilde{W_{q\cE}}$.
\item The group $W_{q\cE}^\circ$ is a subgroup of $\widetilde{W_{q\cE}}$. Extend that to a set of
representatives $\dot{W_{q\cE}}$ for $W_{q\cE} = \widetilde{W_{q\cE}} / (M^\vee)_\cE$ in 
$\widetilde{W_{q\cE}}$, such that $\dot{rw} = \dot{r}w$ for $r \in \mf R_{q\cE}, w \in W_{q\cE}^\circ$.
\item For every $\dot w \in \dot{W_{q\cE}}$ we pick an intertwiner
$I^{\dot w} \in \Hom_{M^\vee_\cE / (M^\vee)^\circ}(\dot w \cdot \rho_{M^\vee}, \rho_{M^\vee})$.
For $m \in M^\vee_\cE / (M^\vee)^\circ$ we define $I^{\dot{w} m} = I^{\dot w} \circ \rho_{M^\vee}(m)$.
\item The operators $I^{\dot w m}$ determine a 2-cocycle $\natural_{q\cE}$ for 
$\widetilde{W_{q\cE}}$, which factors through $(W_{q\cE})^2$. This gives twisted group algebras
$\C [\widetilde{W_{q\cE}},\natural_{q\cE}]$ and $\C [W_{q\cE},\natural_{q\cE}]$. 
Here we use the conventions as in \cite[\S 1]{AMS1}:
\[
\natural_{q\cE}(w,w') = T_w T_{w'} T_{w w'}^{-1} = I^{w w'} (I^{w'})^{-1} (I^w)^{-1} .
\]
\item For a representation $(\tau,V_\tau)$ of $\C [W_{q\cE},\natural_{q\cE}]$, 
$V_\tau \otimes V_{\rho_{M^\vee}}$ becomes a representation of 
$\C [\widetilde{W_{q\cE}},\natural_{q\cE}]$ by defining
\[
\tilde{T}_{\dot w m} (v_1 \otimes v_2) = \tau (T_w) v_1 \otimes I^{\dot w} \rho_{M^\vee}(T_m) v_2 .
\]
\item For the remainder of the construction see the proof of \cite[Lemma 5.4]{AMS1}.
\end{itemize}
Now we investigate what happens when we tensor everything with $\tau_{\phi_b}(g)$. Recall from 
Lemma \ref{lem:2.1}.a that $\tau_{\phi_b}(g)$ is trivial on $(M^\vee)^\circ$ (but not necessarily
on $M^\vee_\cE$). Thus $\cE$ is stable under $\tau_{\phi_b}(g)$, and $q\cE$ is replaced by 
$q\cE \otimes \tau_{\phi_b}(g)$. Then $\rho_{M^\vee}$ is replaced by 
$\rho_{M^\vee} \otimes \tau_{\phi_b}(g)^{-1}$, see \cite[Theorem 4.7]{AMS1}. We can keep the 
$I^{\dot w}$ for $\dot w \in \dot{W_{q\cE}}$, for $\tau_{\phi_b}(g)$ extends to a character of
$\widetilde{W_{q \cE}}$, so every $\dot w$ stabilizes $\tau_{\phi_b}(g)$. However, we must 
replace $I^{\dot{w} m}$ by 
\[
\tilde{I}^{\dot{w} m} = I^{\dot w} \circ \rho_{M^\vee}(m) \tau_{\phi_b}(g)^{-1}(m) .
\]
This determines a new 2-cocycle $\natural_{q\cE \otimes \tau_{\phi_b}(g)}$ of $W_{q\cE}$, and an
algebra isomorphism
\begin{equation}\label{eq:2.28}
\begin{array}{cccc}
\alpha_g : & \C [W_{q\cE},\natural_{q\cE}] & \to & 
\C [W_{q\cE},\natural_{q\cE \otimes \tau_{\phi_b}(g)}] \\
& T_w & \mapsto & \tau_{\phi_b}(g)(\dot w) T_w 
\end{array} .
\end{equation}
For $w \in W_{q\cE}^\circ$ we chose $\dot w = w \in N_{(G^\vee_{\phi_b})^\circ} (M^\vee) / 
(M^\vee)^\circ$, so $\tau_{\phi_b}(\dot w) = 1$ (by Lemma \ref{lem:2.1}.a) and $\alpha_g (T_w) = T_w$.
\end{proof}

We note that $w \mapsto \tau_{\phi_b}(g)(\dot w)$ is not necessarily a character of $W_{q\cE}$, 
because it involves the choice of $\dot{W_{q\cE}}$. Nevertheless, $\natural_{q\cE}$ and
$\natural_{q\cE \otimes \tau_{\phi_b}(g)}$ are equivalent in $H^2 (W_{q\cE},\C^\times)$.

Consider $(\phi,\rho) \in \Phi_e (G)$ with cuspidal support in $(\mc L (F), X_\nr (L)_{rs} \phi_b 
,q\epsilon)$. As explained in the proof of \cite[Theorem 3.8]{AMS3}, upon replacing $(\phi,\rho)$ by
an equivalent parameter, we may assume that 
\[
\phi |_{\mb I_F} = \phi_b |_{\mb I_F}, \textup{d}\phi |_{SL_2 (\C)} \matje{1}{0}{0}{-1} \in \mf t 
\text{ and } \phi (\Fr) \phi_b (\Fr )^{-1} \in X_\nr (L)_{rs} .
\]
We abbreviate $\sigma = \phi (\Fr) \phi_b (\Fr )^{-1}$ and regard it both as an element of\\
$\big( Z(L^\vee)^{\mb I_F} \big)^\circ_\Fr$ and an element of $X_\nr (L)_{rs}$. As $\sigma$
lies in the real split part of a torus, $\sigma^t \in X_\nr (L)_{rs}$ is well-defined for every
$t \in \R$. Then
\[
[0,1] \to \Hom (\mb W_F, L^\vee \rtimes \mb W_F) : t \mapsto \sigma^t \phi_b |_{\mb W_F}
\]
is a path between $\phi_b |_{\mb W_F}$ and $\phi |_{\mb W_F}$, We note that $\tau_\phi (g)$ 
depends only on $\phi |_{\mb W_F}$ (and similarly for $\phi_b$), and that
\[
Z_{G^\vee}(\phi (\mb W_F)) = Z_{G^\vee}(\sigma^t \phi (\mb W_F)) \subset
Z_{G^\vee}(\phi_b (\mb W_F)) \qquad \text{ for all } t \in (0,1] .
\]
Now Lemmas \ref{lem:2.1}.b and \ref{lem:2.2} say that
\begin{equation}\label{eq:2.29}
\tau_\phi = \tau_{\phi_b} \big|_{Z_{G^\vee}(\phi (\mb W_F))} .
\end{equation}
This allows us replace $\tau_\phi$ and $\tau_{\mc S_\phi}$ by $\tau_{\phi_b}$ whenever we please.

The algebra isomorphism $\alpha_g$ from Proposition \ref{prop:2.4} gives an equivalence of
categories
\[
\begin{array}{cccc}
\alpha_g^* : & \mr{Mod} \big( \mh H (\phi_b, q\epsilon \otimes \tau_{\phi_b}(g), \vec{\mb r}) \big) & 
\to & \mr{Mod} \big( \mh H (\phi_b, q\epsilon, \vec{\mb r}) \big)  \\
& (\pi,V) & \mapsto & (\pi \circ \alpha_g,V)
\end{array} .
\]

\begin{lem}\label{lem:2.5}
Suppose that $(\phi,\rho) \in \Phi_e (G)$ with cuspidal support in \\
$(L, X_\nr (L)_{rs} \phi_b ,q\epsilon \otimes \tau_{\phi_b}(g))$. Fix $\vec{r} \in \C^d$ 
and recall the $\mh H (\phi_b, q\epsilon \otimes \tau_{\phi_b}(g), \vec{\mb r})$-modules
$M(\phi,\rho,\vec{r})$ and $E(\phi,\rho,\vec{r})$ from Theorem \ref{thm:1.2}. Then
\[
\alpha_g^* E(\phi,\rho,\vec{r}) = E(\phi,\rho \otimes \tau_{\phi}(g)^{-1},\vec{r}) 
\quad \text{and} \quad 
\alpha_g^* M(\phi,\rho,\vec{r}) = M(\phi,\rho \otimes \tau_{\phi}(g)^{-1},\vec{r}) .
\]
\end{lem}
\begin{proof}
Let IM denote the Iwahori--Matsumoto involution of a (twisted) graded Hecke algebra.
(By definition, it restricts to the identity on $\C [\mf R_{q\cE}, \natural_{q\cE}]$.)
By \cite[Theorem 3.8]{AMS3}, for a certain $\sigma_L \in \mr{Lie}(X_\nr (L)_{rs})$:
\begin{equation}\label{eq:2.55}
E(\phi,\rho,\vec{r}) = \mr{IM}^* E_{\log (u_\phi), \sigma_L, \vec{r}, \rho} .
\end{equation}
Let $\rho^\circ$ be an irreducible constituent of $\rho |_{\mc S_\phi^\circ}$, as in \eqref{eq:2.30}.
By \cite[Lemma 3.13]{AMS2} there is a unique 
\[
\tau \in \Irr (\C [\mf R_{\cE,u_{\phi_b},\sigma},
\natural_{q\cE \otimes \tau_{\phi_b}(g))}^{-1}]) \cong
\Irr (\C [(\mc S_\phi)_{\rho^\circ} / \mc S_\phi^\circ, 
\natural_{q\cE \otimes \tau_{\phi_b}(g))}^{-1} ]) 
\]
such that $\rho = \rho^\circ \rtimes \tau$. Let $\tau^* \in \Irr (\C [\mf R_{\cE,u_{\phi_b},
\sigma},\natural_{q\cE \otimes \tau_{\phi_b}(g))}])$ be the contragredient of $\tau$.
By \cite[Lemma 3.18]{AMS2}, which is applicable in this generality by \cite[Theorem 1.4]{AMS3},
\begin{equation}\label{eq:2.31}
E_{\log(u_\phi),\sigma_L,\vec{r}, \rho^\circ \rtimes \tau} =
\tau^* \ltimes E^\circ_{\log (u_\phi),\sigma_L,\vec{r}, \rho^\circ} ,
\end{equation}
where $E^\circ_{\log(u_\phi),\sigma_L,\vec{r}, \rho^\circ}$ is a standard module for
\[
\mh H ((G^\vee_{\phi_b})^\circ, (M^\vee)^\circ, \cE,\vec{\mb r}) 
\otimes \mc O \big( \mr{Lie}(X_\nr (G)) \big) .
\]
Applying $\alpha_g$ in the form \eqref{eq:2.28} to \eqref{eq:2.55} and \eqref{eq:2.31}, we obtain 
\[
\alpha_g^* E(\phi,\rho,\vec{r}) = \mr{IM}^* (\tau^* \otimes \tau_{\phi_b}(g)) 
\ltimes E^\circ_{\log (u_\phi),\sigma_L,\vec{r}, \rho^\circ} .
\]
Again by \eqref{eq:2.31} this equals
\[
\mr{IM}^* E_{\log(u_\phi),\sigma_L,\vec{r}, \rho^\circ \rtimes (\tau \otimes \tau_{\phi_b}(g)^{-1})} 
= \mr{IM}^* E_{\log(u_\phi),\sigma_L,\vec{r}, \rho \otimes \tau_{\phi_b}(g)^{-1}} = 
E(\phi,\rho \otimes \tau_{\phi}(g)^{-1},\vec{r}) .
\]
By \cite[(69) and (70)]{AMS2} the analogue of \eqref{eq:2.31} for $M(\phi,\rho,\vec{r})$ also holds.
Knowing that, the above argument for $E(\phi,\rho,\vec{r})$ applies to $M(\phi,\rho,\vec{r})$. 
\end{proof}

We would like to show that, for every bounded $\phi_b$, the isomorphism
$\alpha_g$ from Proposition \ref{prop:2.4} is induced by an isomorphism
\[
\mc H (\mf s^\vee, \vec{\mb z}) \to \mc H (\mf s^\vee \otimes \tau_{\phi_b}(g),\vec{\mb z}) . 
\]
However, it seems that this cannot be realized with (twisted) graded Hecke algebras.
To approach the desired situation, we replace $\mh H (\phi_b,q\epsilon,\vec{\mb r})$ by a larger 
algebra, which has the same irreducible representations but admits more inner automorphisms. 

Recall from \cite[Lemma 2.3]{AMS2} that 
\begin{equation}\label{eq:2.32}
S \big( \mf t^* \oplus \mr{Lie}^* (X_\nr (G)) \oplus (\C^*)^d \big)^{W_{q\cE}} = 
\mc O \big( \mf t \times \mr{Lie}(X_\nr (G)) \big)^{W_{q\cE}} \oplus \C [\vec{\mb r}]
\end{equation}
is a central subalgebra of $\mh H (\phi_b,q\epsilon,\vec{\mb r})$. More precisely, since 
$W_{q\cE}$ acts faithfully on $\mf t$ (the pointwise stabilizer of $\mf t$ in $G^\vee_{\phi_b}$ 
is $M^\vee$, but the image of $M^\vee$ in $W_{q\cE}$ is 1), \eqref{eq:2.32} is the full centre of 
$\mh H (\phi_b,q\epsilon,\vec{\mb r})$. 

Let $C^{an}(U)$ be the algebra of complex analytic functions on a complex variety $U$. 
We define the algebras
\[
\mh H^{an} (\phi_b,q\epsilon,\vec{\mb r}) = C^{an} \big( \mf t \times \mr{Lie}(X_\nr (G))
\oplus \C^d \big)^{W_{q\cE}} \underset{\mc O \big( \mf t \times \mr{Lie}(X_\nr (G)) \times \C^d 
\big)^{W_{q\cE}}}{\otimes} \mh H (\phi_b,q\epsilon,\vec{\mb r}) .
\]
As observed in \cite[\S 1.5]{SolAHA}, based on \cite[Proposition 4.3]{Opd-Sp}, pullback along
the inclusion
\begin{equation}\label{eq:2.41}
\mh H (\phi_b,q\epsilon,\vec{\mb r}) \to \mh H^{an} (\phi_b,q\epsilon,\vec{\mb r}) 
\end{equation}
provides an equivalence between the respective categories of finite dimensional modules. 
We note that 
\begin{multline}
C^{an} \big( \mf t \times \mr{Lie}(X_\nr (G)) \times \C^d \big) = \\ 
C^{an} \big( \mf t \times \mr{Lie}(X_\nr (G)) \times \C^d \big)^{W_{q\cE}} 
\underset{\mc O \big( \mf t \times \mr{Lie}(X_\nr (G)) \times \C^d \big)^{W_{q\cE}}}{\otimes} 
\mc O \big( \mf t \times \mr{Lie}(X_\nr (G)) \big)
\end{multline}
is a (commutative, but usually not central) subalgebra of $\mh H^{an} (\phi_b,q\epsilon,\vec{\mb r})$.
As vector spaces
\begin{equation}\label{eq:2.45}
\mh H^{an} (\phi_b,q\epsilon,\vec{\mb r}) = C^{an} \big( \mf t \times \mr{Lie}(X_\nr (G)) \times \C^d \big) 
\otimes_\C \C [W_{q \cE},\natural_{q\cE}] .
\end{equation}
We extend $\alpha_g$ to an isomorphism
\[
\mh H^{an} (\phi_b,q\epsilon,\vec{\mb r}) \to 
\mh H^{an} (\phi_b,q\epsilon \otimes \tau_{\phi_b}(g),\vec{\mb r})  
\]
by letting it act trivially on $C^{an} \big( \mf t \times \mr{Lie}(X_\nr (G)) \big)$.

In view of \eqref{eq:2.33} we may assume that $g$ is represented in $T_{AD} = (\mc T / Z(\mc G))(F)$.
Evaluating unramified characters at $g$ determines an element 
\[
x_g \in X^* (X_\nr (T_{AD}))
\]
which can also be regarded as an algebraic character of the subtorus $X_\nr (L_{AD}) \cong 
\big( Z(L_c^\vee)^{\mb I_F} \big)^\circ_\Fr$ of $X_\nr (T_{AD})$. Via the covering
\[
Z(M^\vee)^\circ = Z(L_c^\vee)^{\mb W_F,\circ} \to X_\nr (L_{AD}),
\] 
$x_g$ lifts to a character of $Z(M^\vee)^\circ$.
We note that $-x_g = x_{g^{-1}}$ (in all the aforementioned interpretations).

Composition with the exponential map $\mf t = \mr{Lie}(X_\nr (L_{AD})) \to X_\nr (L_{AD})$ 
yields an analytic character $e^{x_g}$ of $\mf t$. It can be regarded as an invertible element of 
\[
C^{an} \big( \mf t \times \mr{Lie}(X_\nr (G)) \times \C^d \big) \subset 
\mh H^{an} (\phi_b,q\epsilon,\vec{\mb r}).
\]
The automorphism Ad$(e^{x_g})$ of $\mh H^{an} (\phi_b,q\epsilon,\vec{\mb r})$ is inner, so it
acts trivially on representations up to equivalence. Obviously Ad$(e^{x_g})$ can also be considered 
as an inner automorphism of $\mh H^{an} (\phi_b,q\epsilon \otimes \tau_{\phi_b}(g),\vec{\mb r})$.

\begin{lem}\label{lem:2.6}
The algebra isomorphism 
\[
\alpha_g \circ \mr{Ad}(e^{x_g}) : \mh H^{an} (\phi_b,q\epsilon,\vec{\mb r}) \to 
\mh H^{an} (\phi_b,q\epsilon \otimes \tau_{\phi_b}(g),\vec{\mb r})   
\]
equals $\mr{Ad}(e^{x_g}) \circ \alpha_g$. Up to equivalence, it has the same effect on modules
as $\alpha_g$ (as in Lemma \ref{lem:2.5}).
\end{lem}
\begin{proof}
The statement about modules is clear from the above. Let $A$ be the subalgebra of 
$\mh H^{an} (\phi_b,q\epsilon,\vec{\mb r})$ generated by 
$C^{an} \big( \mf t \times \mr{Lie}(X_\nr (G)) \big)$ and $\C [W_{q\cE}^\circ]$. Then 
\begin{equation}\label{eq:2.34}
\mh H^{an} (\phi_b,q\epsilon,\vec{\mb r}) = A \otimes \C [\mf R_{q\cE},\natural_{q\cE}] 
\end{equation}
as vector spaces, and $A$ can also be regarded as a subalgebra of 
$\mh H^{an} (\phi_b,q\epsilon \otimes \tau_{\phi_b}(g),\vec{\mb r})$.
With that identification $\alpha_g$ is the identity on $A$, and in particular it commutes with
Ad$(e^{x_g})$ on $A$. For $r \in \mf R_{q\cE}$, represented by $\dot r \in \widetilde{W_{q\cE}}$ 
as in \eqref{eq:2.28}:
\begin{align*}
\alpha_g \circ \mr{Ad}(e^{x_g}) (T_r) & = \alpha_g (e^{x_g} T_r e^{-x_g}) = 
\alpha_g (T_r r^{-1}(e^{x_g}) e^{-x_g}) \\
& = \alpha_g (T_r) r^{-1}(e^{x_g}) e^{-x_g} =
\tau_{\phi_b}(g)(\dot r) T_r r^{-1}(e^{x_g}) e^{-x_g} \\
& = e^{x_g} \tau_{\phi_b}(g)(\dot r) T_r e^{-x_g} =
e^{x_g} \tau_{\phi_b}(g)(\dot r) T_r e^{-x_g} \\
& = \mr{Ad}(e^{x_g}) \circ \alpha_g (T_r) .
\end{align*}
In view of \eqref{eq:2.34}, this shows that $\alpha_g \circ \mr{Ad}(e^{x_g}) = \mr{Ad}(e^{x_g})
\circ \alpha_g$ on the whole of $\mh H^{an} (\phi_b,q\epsilon,\vec{\mb r})$.
\end{proof}

\subsection{Action on affine Hecke algebras} \
\label{par:2.3}

Recall the twisted affine Hecke algebra $\mc H (\mf s^\vee, \vec{\mb z})$ from \eqref{eq:2.4}.
For every $\phi_b$ with $(L,\phi_b,q\epsilon) \in \mf s^\vee_L$, we can consider the subalgebra
$\mc H (\mf s^\vee,\phi_b,\vec{\mb z})$ with as data the torus $\mf s^\vee_L$, roots
$\{ \alpha \in \Phi_{\mf s^\vee} : s_\alpha (\phi_b) = \phi_b\}$, the finite group
$W_{q\cE} = W_{\mf s^\vee, \phi_b, q\epsilon}$, parameters $\lambda, \lambda^*$ and the 2-cocycle 
$\natural_{q\cE}$. As explained in the proofs of \cite[Theorems 2.5.a and 3.18.a]{AMS3}, 
there is a natural equivalence between the following categories:
\begin{itemize}
\item[(i)] finite dimensional $\mc H (\mf s^\vee,\vec{\mb z})$-modules with weights in 
$W_{q\cE} X_\nr (L)_{rs} \phi_b \times \R_{>0}^d$;
\item[(ii)] finite dimensional $\mc H (\mf s^\vee,\phi_b,\vec{\mb z})$-modules with weights in 
$X_\nr (L)_{rs} \phi_b \times \R_{>0}^d$.
\end{itemize}
(The weights are meant with respect to the commutative subalgebras from the 
Bernstein presentation.) The map
\[
\begin{array}{ccccc}
\exp_{\phi_b} : & \mf t \times \mr{Lie}(X_\nr (G)_{rs}) \times \C^d & \to & 
X_\nr (L) \phi_b \times (\C^\times )^d & = \mf s^\vee_L \times (\C^\times )^d \\
& (\lambda, \vec{r}) & \mapsto & (\exp (\lambda) \phi_b, \exp (\vec r)) 
\end{array} 
\]
induces a $W_{q\cE}$-equivariant homomorphism
\[
\begin{array}{cccc}
\exp_{\phi_b}^* : & \mc O \big( \mf s_L^\vee \times (\C^\times)^d \big) & \to & 
C^{an} \big( \mf t \times \mr{Lie}(X_\nr (G)) \times \C^d \big) \\
& f & \mapsto & f \circ \exp_{\phi_b} 
\end{array}.
\]
By \cite[Theorem 2.1.4]{SolAHA} this extends to an injective algebra homomorphism
\begin{equation}\label{eq:2.35}
\exp_{\phi_b}^* : \mc H (\mf s^\vee, \phi_b, \vec{\mb z}) \to \mh H^{an}(\phi_b, q\epsilon, \vec{\mb r}) . 
\end{equation}
By \cite[Corollary 2.15]{SolAHA} (see the explanation in the proof of \cite[Theorem 3.18.a]{AMS3}),
\eqref{eq:2.35} induces an equivalence between (ii) and the category of
\begin{itemize}
\item[(iii)] \begin{tabular}{l}            
finite dimensional $\mh H^{an}(\phi_b, q\epsilon, \vec{\mb r})$-modules (or, equivalently,\\
$\mh H (\phi_b, q\epsilon, \vec{\mb r})$-modules)
with weights in $\mf t_\R \times \mr{Lie}(X_\nr (G)_{rs}) \times \R_{>0}^d$.
\end{tabular}
\end{itemize}
Next we would like to define an analogue of conjugation by $x_g$ on $\mc H (\mf s^\vee,\vec{\mb z})$.
However, $x_g \in X^* (X_\nr (T_{AD}))$ does in general not define a character of $T_{\mf s^\vee}$ 
or $\mf s_L^\vee$. The best approximation we found is:

\begin{lem}\label{lem:2.7}
Let $x \in X^* (X_\nr (L_{AD}))$ and $w \in N_{{G^\vee}_\Sc}(L^\vee_c \rtimes \mb W_F) / 
L^\vee_c \cong N_{G^\vee}(L^\vee \rtimes \mb W_F) / L^\vee$. Then $w(x) - x$ 
naturally defines an algebraic character of $X_\nr (L)$, trivial on $X_\nr (G)$.
\end{lem}
\begin{proof}
The natural map $Z(L_c^\vee) \times Z(G^\vee) \to Z(L^\vee)$ induces a finite covering of tori
\begin{equation}\label{eq:2.36}
X_\nr (L_{AD}) \times X_\nr (G) \to X_\nr (L) ,
\end{equation}
compare with \cite[Lemma 3.7]{AMS3}. Its kernel is contained in the image of
\begin{equation}\label{eq:2.56}
\big( Z({G^\vee}_\Sc) \cap Z(L_c^\vee)^\circ \big) \times Z(G^\vee)^\circ \longrightarrow
X_\nr (L_{AD}) \times X_\nr (G) .
\end{equation}
For any $z \in Z({G^\vee}_\Sc ) \cap Z(L_c^\vee)^\circ$:
\[
(w(x) - x) (z) = x (w^{-1} z w z^{-1}) = x (1) = 1 .
\]
So we may regard $w(x) - x$ as a character of $X_\nr (L_{AD}) \times X_\nr (G)$ which is
trivial on $\big( Z({G^\vee}_\Sc) \cap Z(L_c^\vee)^\circ \big) \times Z(G^\vee)^\circ$. 
In view of \eqref{eq:2.56}, $w(x)-x$ factors through \eqref{eq:2.36}.
\end{proof}

To proceed, we impose the following conditions:
\begin{cond}\label{cond:2}
\begin{itemize}
\item[(i)] The group $\mf R_{\mf s^\vee}$ fixes the basepoint $(\phi_L,\rho_L)$ of $\mf s_L^\vee$.
\item[(ii)] $w (x) - x \in X^* (T_{\mf s^\vee})$ for all $w \in W_{\mf s^\vee} ,\; x \in
X^* (X_\nr (T_{AD}))$.
\end{itemize}
\end{cond}
Assuming Condition \ref{cond:2}.i, the action of $\mf R_{\mf s^\vee}$ on $\mf s_L$ 
lifts to an action on $T_{\mf s^\vee}$, by algebraic group automorphisms.

Recall from \cite[\S 3]{Lus-Gr} that $\mc H (T_{\mf s^\vee}, W_{\mf s^\vee}^\circ, 
\lambda, \lambda^*, \vec{\mb z})$ has an Iwahori--Matsumoto presentation, with a basis 
$\{ N_w : w \in X^* (T_{\mf s^\vee}) \rtimes W_{\mf s^\vee}^\circ \}$. 
In terms of the length function $\ell$ and a simple (affine) reflection
$s$, the multiplication relations are determined by
\begin{equation}\label{eq:2.49}
N_w N_s = \left\{ 
\begin{array}{ll}
N_{ws} & \text{ if } \ell (ws) = \ell (w) + 1\\
N_{ws} + (\mb z (s) - \mb z (s)^{-1}) N_w & \text{ if } \ell (ws) = \ell (w) - 1
\end{array} \right. .
\end{equation}
Here $\mb z$ means the function on $X^* (T_{\mf s^\vee}) \rtimes W_{\mf s^\vee}^\circ$ 
derived from $\vec{\mb z}, \lambda, \lambda^*$ as in \cite[\S 3.1]{Lus-Gr}. When Condition 
\ref{cond:2}.i holds, $\mc H (\mf s^\vee, \vec{\mb z})$ also admits an Iwahori--Matsumoto 
presentation, such that $N_w N_r = N_{w r}$ for all $r \in \mf R_{\mf s^\vee}, w \in 
X^* (T_{\mf s^\vee}) \rtimes W_{\mf s^\vee}^\circ$.

\begin{prop}\label{prop:2.8}
Assume that Condition \ref{cond:2} holds and let $g \in T_{AD}$. There exists a
$\C [\vec{\mb z},\vec{\mb z}^{-1}]$-algebra 
automorphism $\mr{Ad}(x_g)$ of $\mc H (\mf s^\vee, \vec{\mb z})$, such that 
\[
\mr{Ad}(x_g)(N_w)  = N_{w (w^{-1}(x_g) - x_g)} \qquad w \in X^* (T_{\mf s^\vee}) 
\rtimes W_{\mf s^\vee} .
\]
It fixes $\mc O (T_{\mf s^\vee}) \otimes \C [\vec{\mb z},\vec{\mb z}^{-1}]$ pointwise.
\end{prop}
\begin{proof}
Notice that Condition \ref{cond:2}.ii is needed to make the formula for Ad$(x_g)$ 
well-defined. The automorphism $w \mapsto x_g w x_g^{-1} = w (w^{-1}(x_g) - x_g)$ of 
$X^* (T_{\mf s^\vee}) \rtimes W_{\mf s^\vee}$ need not preserve the length function $\ell$. 
However, we can realize $\mc H (T_{\mf s^\vee}, W_{\mf s^\vee}^\circ, 1, 1, \vec{\mb z})$
as the Iwahori--Hecke algebra $\mc H (G',I)$ of a suitable reductive $p$-adic group $G'$. Then
conjugation by $x_g$ becomes conjugation by an element of the adjoint group $G'_\ad$. 
Ad$(x_g)$ defines an automorphism of $\mc H (G')$ which restricts to an algebra isomorphism
\[
\mc H (G',I) \to \mc H (G',\mr{Ad}(x_g)(I)) : N_w \mapsto N_{x_g w x_g^{-1}} . 
\]
Comparing with the multiplication relations \eqref{eq:2.49} for these Iwahori--Hecke algebras,
we deduce that, for $w, w' \in X^* (T_{\mf s^\vee}) \rtimes W_{\mf s^\vee}^\circ$:
\[
\ell (x_g w w' x_g^{-1}) = \ell (x_g w x_g^{-1}) + \ell (x_g w' x_g^{-1})
\; \Longleftrightarrow \; \ell (w w') = \ell (w) + \ell (w') .
\]
Further, conjugation by any $r \in \mf R_{\mf s^\vee}$ defines an automorphism of 
$X^* (T_{\mf s^\vee}) \rtimes W_{\mf s^\vee}^\circ$ which preserves $\ell$, so conjugation by
$x_g r x_g^{-1}$ defines an automorphism which preserves $\ell \circ \mr{Ad}(x_g)$.
This implies that $N_w \mapsto N_{x_g w x_g^{-1}} = N_{w (w^{-1}(x_g) - x_g)}$ defines an 
automorphism of $\mc H (T_{\mf s^\vee}, W_{\mf s^\vee}, \lambda, \lambda^*, \vec{\mb z})$.

By Lemma \ref{lem:2.7} $X^* (T_{\mf s^\vee})$ and $x_g$ are contained in the commutative group
$X^* (X_\nr (L))$. Hence Ad$(x_g)$ fixes $N_w$ for every $w \in X^* (T_{\mf s^\vee})$.
Starting with the Iwahori--Matsumoto presentation, the subalgebra $\mc O (T_{\mf s^\vee})$
is constructed from the $N_w$ with $w$ in the positive part of $X^* (T_{\mf s^\vee})$, see
\cite[Lemmas 2.7 and 3.4]{Lus-Gr}. As Ad$(x_g)$ fixes $\C [\vec{\mb z},\vec{\mb z}^{-1}]$
by definition, it follows that it fixes 
$\mc O (T_{\mf s^\vee}) \otimes \C [\vec{\mb z},\vec{\mb z}^{-1}]$ pointwise.
\end{proof}

In the remainder of this paragraph we assume that Condition \eqref{cond:2} holds.
Because Ad$(x_g)$ is the identity on $\mc O (T_{\mf s^\vee})$, it stabilizes the 
subalgebra $\mc H (\mf s^\vee, \phi_b, \vec{\mb z})$ of $\mc H (\mf s^\vee, \vec{\mb z})$.
Via the injection \eqref{eq:2.35}, Ad$(x_g)$ can be extended uniquely to an 
automorphism of $\mh H^{an} (\phi_b,q\epsilon,\vec{\mb r})$ which is the identity on
$C^{an}(\mf t \times \mr{Lie}(X_\nr (G)))$. For $\phi_b = \phi_L$ it is none other than
Ad$(e^{x_g})$, but for $\phi_b \neq \phi_L$ the map $\exp_{\phi_b}$ creates more 
complications. From the proof of Proposition \ref{prop:2.8} we see that, for $w \in W_{q\cE}$:
\begin{equation}\label{eq:2.42}
\exp_{\phi_b}^* \mr{Ad}(x_g) (N_w) = (w^{-1}(x_g) - x_g) (\sigma) 
\mr{Ad}(e^{x_g}) (\exp_{\phi_b}^* N_w) \in \mh H^{an}(\phi_b,q\epsilon,\vec{\mb r}) .
\end{equation}
where $\sigma = \phi_b (\Fr) \phi_L (\Fr)^{-1} \in X_\nr (L)$. For $w \in W_{q\cE}^\circ$,
the calculations for Proposition \ref{prop:2.8}.a entail that $w^{-1}(x_g) - x_g$ is $\Z$-linear
combination of roots $\alpha$ with $\alpha (\sigma) = 1$. 
Hence $(w^{-1}(x_g) - x_g)(\sigma) = 1$ and 
\begin{equation}\label{eq:2.43}
\mr{Ad}(x_g) = \mr{Ad}(e^{x_g}) \quad \text{on} \quad
\mh H^{an} \big( (G^\vee_{\phi_b})^\circ ,(M^\vee)^\circ, \cE,\vec{\mb r} \big) ,
\end{equation}
the subalgebra of $\mh H^{an}(\phi_b,q\epsilon,\vec{\mb r})$ generated by 
$C^{an}(\mf t \times \mr{Lie}(X_\nr (G)))$ and $\C [W_{q\cE}^\circ]$.

Fix a lift $\sigma_\Sc \in Z(M^\vee)^\circ$ of $\sigma \in X_\nr (L)$.
For $w \in \mf R_{q\cE}$ we can relate the character
\begin{equation}\label{eq:2.44}
w \mapsto (w^{-1}(x_g) - x_g)(\sigma) = x_g (w \sigma_\Sc w^{-1} \sigma_\Sc^{-1} ) 
\end{equation}
to $\tau_{\phi_b}(g)(w)$. Write $g = g_c g_x$ as in \eqref{eq:2.51}, with $g_c$ compact 
and $g_x$ in the image of $X_* (T_{AD})$. Then 
\[
\tau_{\phi_b}(g) = \tau_{\phi_b}(g_c) \tau_{\phi_b}(g_x) . 
\]
As every unramified character of $T$ or $T_{AD}$ is trivial on $g_c$, 
$\tau_{\phi_b}(g_c) |_{T^\vee}$ is insensitive to twisting $\phi_L$ by
unramified characters, and $x_g$ depends only on $g_x$. 

Suppose that $\tau_{\phi_L}(g_x) = 1$ (which can be achieved for instance when 
$\phi_b$ or $\phi_L$ factors through ${}^L T$). The expression 
$\tau_{\phi_b}(g_x) (w) = \langle g_x , c_w \rangle$
from \eqref{eq:2.16} and \eqref{eq:2.13} works out as  
\[
\tau_{\phi_b}(g_x) (w) = \langle g_x, w \sigma_\Sc w^{-1} \sigma_\Sc^{-1} \rangle 
= x_g (w \sigma_\Sc w^{-1} \sigma_\Sc^{-1}) .
\]
Thus $\tau_{\phi_b}(g_x)$ equals \eqref{eq:2.44}, and we can regard that as the 
unramified part of $\tau_{\phi_b}(g)$.

By Condition \eqref{cond:2} $W_{\mf s^\vee}^\circ = W_{q\cE}^\circ$ and 
$\mf R_{\mf s^\vee} = \mf R_{q\cE}$, where $q\epsilon = q\cE |_{u_\phi}$ is regarded as 
an enhancement of $\phi_L$. Via the choice of representatives $\dot{W_{q\cE}}$ (as in 
the proof of Proposition \ref{prop:2.4}), $\tau_{\phi_L}(g)$ defines an isomorphism
\[
\begin{array}{ccc}
\C [W_{\mf s^\vee}, \kappa_{\mf s^\vee}] = \C [W_{q\cE},\natural_{q\cE}] & \to &
\C [W_{\mf s^\vee}, \kappa_{\mf s^\vee \otimes \tau_{\phi_L} (g)}] = 
\C [W_{q\cE},\natural_{q\cE \otimes \tau_{\phi_L}(g)}] \\
N_w & \mapsto & \tau_{\phi_L}(g)(\dot w) N_w 
\end{array},
\]
compare with \eqref{eq:2.28}. It extends to an algebra isomorphism
\begin{equation}\label{eq:2.54}
\alpha_g : \mc H (\mf s^\vee,\vec{\mb z}) \to 
\mc H (\mf s^\vee \otimes \tau_{\phi_L}(g),\vec{\mb z}) 
\end{equation}
which is the identity on $\mc H (T_{\mf s^\vee}, W_{\mf s^\vee}^\circ, 
\lambda, \lambda^*, \vec{\mb z})$. Consider the composition 
\begin{equation}\label{eq:2.37}
\alpha_g \circ \mr{Ad}(x_g) : \mc H (\mf s^\vee,\vec{\mb z}) \to 
\mc H (\mf s^\vee \otimes \tau_{\phi_L}(g),\vec{\mb z}) .
\end{equation}
We note that, when $g = g_c g_x$ is as in \eqref{eq:2.51} and $\tau_{\phi_L}(g_x) = 1$,
we can decompose \eqref{eq:2.37} in $\alpha_g = \alpha_{g_c}$ and 
$\mr{Ad}(x_g) = \mr{Ad}(x_{g_x})$.

\begin{thm}\label{thm:2.9}
Recall that Condition \eqref{cond:2} is in force. Let $\vec z \in \R_{>0}^d$ and let 
$(\phi,\rho) \in \Phi_e (G)^{\mf s^\vee \otimes \tau_{\phi_L} (g)}$.
The equivalence of categories
\[
\begin{array}{cccc}
(\alpha_g \circ \mr{Ad}(x_g))^* : &
\mr{Mod} \big( \mc H (\mf s^\vee \otimes \tau_{\phi_L} (g),\vec{\mb z}) \big) & 
\to & \mr{Mod} \big( \mc H (\mf s^\vee,\vec{\mb z}) \big),\\
& \pi & \mapsto & \pi \circ \alpha_g \circ \mr{Ad}(x_g)
\end{array}
\]
sends $\bar E (\phi,\rho,\vec z)$ to $\bar E (\phi,\rho \otimes \tau_\phi (g)^{-1},\vec z)$ 
and $\bar M (\phi,\rho,\vec z)$ to $\bar M (\phi,\rho \otimes \tau_\phi (g)^{-1},\vec z)$.
\end{thm}
\begin{proof}
The isomorphism $\alpha_g \circ \mr{Ad}(x_g)$ fixes $\mc O (T_{\mf s^\vee})$ 
pointwise. Hence it restricts to an isomorphism
\[
\mc H (\mf s^\vee,\phi_b,\vec{\mb z}) \to 
\mc H (\mf s^\vee \otimes \tau_{\phi_L}(g),\phi_b,\vec{\mb z}),
\]
where we note that $\tau_{\phi_L}(g) = \tau_{\phi_b}(g)$ on $L^\vee$.
Via the inclusion \eqref{eq:2.41}, this extends canonically to an isomorphism 
\begin{equation}\label{eq:2.52}
\mh H^{an} (\phi_b,q\epsilon,\vec{\mb r}) \to 
\mh H^{an} (\phi_b,q\epsilon \otimes \tau_{\phi_b}(g),\vec{\mb r})
\end{equation}
Here $\phi_b = \sigma \phi_L$ and $\tau_{\phi_b} = \tau_\sigma \tau_{\phi_L}$.
The part $\tau_{\phi_L}$ is accounted for by $\alpha_g$, whereas $\tau_\sigma$
is incorporated in \eqref{eq:2.42}. A calculation analogous to \eqref{eq:2.42}--\eqref{eq:2.44} 
shows that \eqref{eq:2.52} is none other than $\alpha_g \circ \mr{Ad}(e^{x_g})$. 
Finally we apply Lemmas \ref{lem:2.5} and \ref{lem:2.6}.
\end{proof}

\section{Isomorphisms of reductive groups}
\label{sec:3}

In this section $F$ is any local field.
Consider an isomorphism of connected reductive $F$-groups
\[
\eta : \tilde{\mc G} \to \mc G .
\]
As observed in \cite[\S 1.2]{Bor}, $\eta$ and the choice of a pinnning of 
$G^\vee$ induce isomorphisms $\eta^\vee : G^\vee \to \tilde{G}^\vee$ and
\begin{equation}\label{eq:3.8}
{}^L \eta = \eta^\vee \rtimes \mr{id} : G^\vee \rtimes \mb W_F \to \tilde{G}^\vee \rtimes \mb W_F .
\end{equation}
Given $\eta$, these maps are unique up to conjugation by elements of $G^\vee$ (and adjusting the 
$\mb W_F$-action on $G^\vee$ accordingly). By the naturality of the Kottwitz isomorphism
\begin{equation}\label{eq:3.3}
\zeta_{\tilde{\mc G}} \circ \eta^\vee = \zeta_{\mc G} \in \Irr (Z({G^\vee}_\Sc)^{\mb W_F})
\end{equation}
Thus $\eta^\vee$ remembers to which inner twist of $\mc G^*$ the group $\tilde{\mc G}$ is mapped,
and ${}^L \eta$ induces a bijection $\Phi ({}^L \eta) : \Phi (G) \to \Phi (\tilde G )$.
We agree that the extensions $\zeta_{\mc G}^+, \zeta_{\tilde{\mc G}}^+$ from \eqref{eq:1.4} 
are chosen such that $\zeta_{\tilde{\mc G}}^+ \circ \eta^\vee = \zeta_{\mc G}^+$. 
Under that assumption ${}^L \eta$ also induces a bijection
\begin{equation}\label{eq:3.1}
\begin{array}{cccc}
\Phi_e ({}^L \eta) : & \Phi_e (G) & \to & \Phi_e (\tilde G ) ,\\
& (\phi,\rho) & \mapsto & ({}^L \eta \circ \phi, \rho \circ (\eta^\vee)^{-1}) .
\end{array}
\end{equation}
Composing $\eta^\vee$ by an inner automorphism does not change \eqref{eq:3.1}, so
$\Phi ({}^L \eta)$ and $\Phi_e ({}^L \eta)$ are determined uniquely by $\eta$. 

Assume for the moment that $F$ is non-archimedean. Then ${}^L \eta$ induces isomorphisms from 
all the objects associated in 
Section \ref{sec:1} to $G^\vee$ to the analogous objects for $\tilde{G^\vee}$. All the 
constructions underlying the results of \cite{AMS1,AMS2,AMS3} are functorial for algebraic
isomorphisms. For instance, ${}^L \eta$ naturally gives rise to algebra isomorphisms 
\[
\begin{array}{ccc@{\;:\;}ccc}
\mc O (\eta^\vee (\mf s_L^\vee)) & \to & \mc O (\mf s_L^\vee) & 
f & \mapsto & f \circ \Phi_e ({}^L \eta) ,\\
\mc O (\eta^\vee (\mf t) \times \mr{Lie}(X_\nr (\tilde G)) ) & \to &
\mc O (\mf t \times \mr{Lie}(X_\nr (G)) ) & f & \mapsto & f \circ \eta^\vee ,\\
\mc H \big( W_{{}^L \eta (\mf s^\vee )}^\circ, \vec{\mb z}^{2 \tilde{\lambda}} \big) & \to &
\mc H \big( W_{\mf s^\vee}^\circ, \vec{\mb z}^{2 \lambda} \big) &
N_{s_{\eta^\vee (\alpha)}} & \mapsto & N_{s_\alpha} .
\end{array}
\]
Using the notations from the proof of \cite[Proposition 3.15.b]{AMS3}, in particular
$J = Z^1_{{G^\vee}_\Sc (\phi |_{\mb I_F})}$, we get an isomorphism
\[
\begin{array}{ccccc}
\C [W_{{}^L \eta (\mf s^\vee)}, \natural_{{}^L \eta (\mf s^\vee)}] \cong &
\mr{End}_{\tilde J} \big( \widetilde{q \pi}_* (\eta^\vee_* (\widetilde{q \cE})) \big) & \to &
\mr{End}_{J} \big( q \pi_* (\widetilde{q \cE}) \big) & 
\cong \C [W_{\mf s^\vee}, \natural_{\mf s^\vee}] \\
& T & \mapsto & (\eta^\vee_*)^{-1} \circ T \circ \eta^\vee_* 
\end{array}.
\]
The above maps combine to algebra isomorphisms 
\begin{equation}\label{eq:3.2}
\begin{array}{cccc}
\mh H ({}^L \eta) : & \mh H (\Phi_e ({}^L \eta)(\phi_b,q\epsilon), \vec{\mb r}) & \to &
\mh H (\phi_b,q\epsilon, \vec{\mb r}) , \\
\mc H ({}^L \eta) : & \mc H (\Phi_e ({}^L \eta) (\mf s^\vee), \vec{\mb z}) & \to &
\mc H (\mf s^\vee, \vec{\mb z}) .
\end{array}
\end{equation}
The canonicity of the bijections in Theorems \ref{thm:1.1} and \ref{thm:1.2} allows us to conclude:

\begin{cor}\label{cor:3.1}
Let $F$ be a non-archimedean local field and assume the notations of Theorems \ref{thm:1.1} and 
\ref{thm:1.2}, for $G$ and for $\tilde G$. 
Then, for any $\vec{r} \in \C^d, \vec{z} \in (\C^\times)^d$:
\[
\begin{array}{ccc}
M(\phi,\rho,\vec{r}) \circ \mh H ({}^L \eta) & = & M (\Phi_e ({}^L \eta)(\phi,\rho), \vec{r}) ,\\
\overline{M}(\phi,\rho,\vec{z}) \circ \mc H ({}^L \eta) & = & 
\overline{M} (\Phi_e ({}^L \eta)(\phi,\rho), \vec{z}) ,
\end{array}
\]
and similarly for the standard modules $E(\phi,\rho,\vec{r}), \overline{E}(\phi,\rho,\vec{z})$.
\end{cor}

Now $F$ may again be any local field. Upon adjusting $\eta^\vee$ by an inner automorphism of 
$G^\vee$, we may assume that it sends a chosen $\mb W_F$-stable pinning of $G^\vee$ to a 
chosen $\mb W_F$-stable pinning of $\tilde{G}^\vee$. (Actually the Weil group acts via the 
absolute Galois group of $F$, but we find it notationally more convenient to stick to $\mb W_F$.)
Say these pinnings involve the maximal tori $T^\vee, \tilde{T}^\vee$ and the Borel subgroups 
$B^\vee, \tilde{B}^\vee$. Then $\eta^\vee$ determines a $\mb W_F$-equivariant isomorphism 
$\mc R(\eta^\vee)$ from the based root datum
\begin{align*}
\mc R (G^\vee,T^\vee) & = (X^* (T^\vee), \Phi (G^\vee,T^\vee), X_* (T^\vee), 
\Phi^\vee (G^\vee,T^\vee), \Delta^\vee) \\
& = (X_* (\mc T), \Phi^\vee (\mc G,\mc T), X^* (\mc T), \Phi (\mc G,\mc T), \Delta^\vee) .
\end{align*}
to $\mc R (\tilde{G}^\vee, \tilde{T}^\vee)$. 
Up to inner automorphisms of $G^\vee$, $\eta^\vee$ is determined by $\mc R (\eta^\vee)$.

Similarly, up to inner automorphisms of $F_s$-groups, $\eta$ is determined by the isomorphism
of based root data $\mc R (\eta) = \mc R (\eta^\vee)^\vee$. However, as we saw in Section
\ref{sec:2}, some inner automorphisms of $\mc G (F)$ come from elements of $\mc G (F_s)$
that do not lie in (the image of) $\mc G (F)$. As a consequence, there can exist several
isomorphisms $\tilde{\mc G}(F) \to \mc G (F)$, not equivalent up to $\mc G (F)$-conjugation,
that give rise to $G^\vee$-conjugate isomorphisms $\eta^\vee : G^\vee \to \tilde{G^\vee}$.
The remainder of this section is dedicated to making these remarks precise.

In \cite[\S 16.3--16.4]{Spr} the group of $F$-algebraic automorphisms of $\mc G$ which are the 
identity on $Z(\mc G)^\circ$ is analysed. It turns out to be a linear algebraic $F$-group with
finitely many components. We will extend Springer's analysis to the group $\Aut_F (\mc G)$ 
of all $F$-algebraic automorphisms of $\mc G$. This need not be an algebraic group, for instance
because the automorphism group of a split torus of dimension $\geq 2$ is infinite and discrete.

Let $\mc G^*$ be a quasi-split reductive $F$-group. We fix a $\mb W_F$-stable pinning
of $\mc G^*$, consisting of a maximally split maximal $F$-torus $\mc T^*$, a Borel
$F$-subgroup $\mc B^*$ containing $\mc T^*$ and for every simple root $\alpha \in 
\Phi (\mc G^*,\mc T^*)$ an element $x_\alpha \in \mc U_\alpha (F_\alpha)$, where 
$F_\alpha$ is the minimal field extension of $F$ over which $\mc U_\alpha$ is defined. Then
$\prod_{\alpha' \in \mb W_{F \alpha}} x_{\alpha'} \in \mc G^* (F)$ for every $\alpha \in \Delta$.

\begin{thm}\label{thm:3.2}
Let $\tau$ be a $\mb W_F$-equivariant automorphism of $\mc R (\mc G^*,\mc T^*)$.
\enuma{
\item There exists a unique $\eta_\tau \in \Aut_F (\mc G^*)$ which stabilizes the pinning,
commutes with $\mb W_F$ and induces $\tau$ on $\mc R (\mc G^*,\mc T^*)$.
\item $\Aut_F (\mc G^*)$ is isomorphic to the semidirect product of $\mc G^*_\ad (F)$ and the 
group of $\mb W_F$-automorphisms of $\mc R (\mc G^*,\mc T^*)$ (acting via part a).
\item Let $\mc G^*_u$ be the inner twist of $\mc G^*$ parametrized by 
$u \in Z^1 (F,\mc G^*_\ad)$. Then the automorphism $\eta_\tau$ of $\mc G^* (F_s)$ restricts 
to an automorphism of $F$-groups $\mc G^*_u \to \mc G^*_{\eta_\tau (u)}$.
}
\end{thm}
\begin{proof}
(a) By \cite[Lemma 16.3.8]{Spr} $\tau$ lifts uniquely to an $F$-automorphism of
$\mc G^*_\ad$ stabilizing the pinning derived from the pinning of $\mc G^*$.
The same argument works for $\mc G^* (F)$, provided we omit the condition that the
connected centre must be fixed, and instead use the automorphism of $Z(\mc G^*)^\circ$
determined by $\tau$.\\
(b) By part (a) the indicated semidirect product embeds in $\Aut_F (\mc G^*)$.

Every $F$-automorphism $\eta^*$ of $\mc G^*$ induces an $F$-automorphism of the 
quasi-split adjoint group $\mc G^*_\ad$. From \cite[Lemma 16.4.6]{Spr} we see that the 
induced automorphism of $\mc R (\mc G^*_\ad, \mc T^*_{AD})$ must commute with the
$\mb W_F$-action. When $\eta^*$ stabilizes $(\mc B^*, \mc T^*)$ (which can always be
achieved by composing $\eta^*$ with an inner automorphism $c$), it restricts to an
$F$-automorphism of $\mc T^*$ and of $\Phi (\mc G^*, \mc T^*)$. That it is defined over
$F$ implies that its action on $X_* (\mc T^*)$ must be $\mb W_F$-equivariant, and
that it stabilizes $\mc B^*$ means that it maps $\Delta^\vee \subset 
\Phi (\mc G^*,\mc T^*)$ to itself. Combining these two observations, we deduce that the
automorphism $\mc R(\eta^*)$ of $\mc R (\mc G^*,\mc T^*)$ induced by $\eta^*$ is
$\mb W_F$-equivariant. \\
(c) By definition
\begin{equation}\label{eq:3.4}
\mc G^*_u (F) = \{ g \in \mc G^* (F_s) : \mr{Ad}(u(\gamma)) \circ \gamma (g) = g \;
\forall \gamma \in \mb W_F \} .
\end{equation}
For $\gamma \in \mb W_F, g \in \mc G_u (F)$, by part (a) and \eqref{eq:3.4}:
\[
\mr{Ad} \big( \eta_\tau (u(\gamma)) \big) \circ \gamma \circ \eta_\tau (g) = 
\eta_\tau \circ \mr{Ad}(u(\gamma)) \circ \gamma (g) = \eta_\tau (g) .
\]
This shows that $\eta_\tau (g) \in \mc G^*_{\eta_\tau (u)}(F)$. Hence $\eta_\tau$ 
determines an algebraic automorphism $\mc G^*_u \to \mc G^*_{\eta_\tau (u)}$, which
is defined over $F$ because it maps $F$-rational points to $F$-rational points.
\end{proof}

From Theorem \ref{thm:3.2} we see that the identity component of $\Aut_F (\mc G^*)$ 
is the group of inner automorphisms, which is an algebraic $F$-group
naturally isomorphic to $\mc G_\ad$ \cite[Lemma 16.3.7]{Spr}. The component group of  
$\Aut_F (\mc G^*)$ is canonically isomorphic to the group of 
$\mb W_F$-equivariant automorphisms of the based root datum of $\mc G^*$. When $\mc G^*$
is semisimple, this component group is finite, but for reductive groups it can be
infinite. That happens if and only if there is an irreducible representation of
$\mb W_F$ which appears with multiplicity $>1$ in $X_* (Z(\mc G^*)^\circ)$.

For $u \in Z^1 (F,\mc G^*_\ad)$ Springer defines the inner twist Inn$(\mc G^*)_u$ of
Inn$(\mc G^*) \cong \mc G^*_\ad$. It comes from the Galois action
\begin{equation}\label{eq:3.5}
\gamma * \eta = u(\gamma) (\gamma \cdot \eta) u (\gamma)^{-1} 
\qquad \gamma \in \mr{Gal}(F_s / F), \eta \in \mr{Inn}(\mc G^*) ,
\end{equation}
and Inn$(\mc G^*)_u (F)$ consists precisely of the inner automorphisms of $\mc G^*_u$
that are defined over $F$ \cite[Lemma 16.4.6]{Spr}. As in \cite[Lemma 16.3.7]{Spr},
one can show that there is a canonical isomorphism of $F$-groups
\begin{equation}\label{eq:3.6}
(\mc G^*_u)_\ad = \mc G^*_{\ad,u} \to \mr{Inn}(\mc G^*)_u . 
\end{equation}
Since $\mc G$ will typically be an inner twist of $\mc G^*$, we write $(G^*)^\vee = G^\vee$.
Let $\zeta \in \Irr (Z({G^\vee}_\Sc)^{\mb W_F})$ be the image of $u \in H^1 (F,\mc G^*_\ad)$
under the Kottwitz homomorphism. For $\tau$ as in Theorem \ref{thm:3.2}.c, the induced
automorphism $\eta_{\tau^\vee}$ of $G^\vee$ (unique up to conjugacy) enables one to define 
\[
\tau (\zeta) = \zeta \circ \eta_{\tau^\vee} \in \Irr (Z({G^\vee}_\Sc)^{\mb W_F}) 
\]
unambigously. By the naturality of the Kottwitz homomorphism, $\tau (\zeta)$ is the image
of $\eta_\tau (u)$.

\begin{cor}\label{cor:3.3}
Let $u,u' \in Z^1 (F,\mc G^*_\ad)$ with images $\zeta, \zeta' \in \Irr (Z({G^\vee}_\Sc)^{\mb W_F})$.
Let $\tau$ be a $\mb W_F$-automorphism of $\mc R (\mc G^*,\mc T^*)$. 
\enuma{
\item The following are equivalent:
\begin{itemize}
\item[(i)] $\tau (\zeta') = \zeta$;
\item[(ii)] there exists an $\eta \in \Aut_{F_s}(\mc G^*)$ with $\mc R (\eta) = \tau$, 
which restricts to an isomorphism of $F$-groups $\mc G^*_{u'} \to \mc G^*_u$.
\end{itemize}
\item When (i) and (ii) hold, the group $\mc G_{\ad,u}(F)$ acts simply transitively on the set of
such $\eta$ (by composition).
\item $\Aut_F (\mc G^*_u)$ has:
\begin{itemize}
\item identity component $\mc G^*_{u,\ad}(F)$; 
\item component group canonically isomorphic to the group of $\mb W_F$-automorphisms of 
$\mc R (\mc G^*,\mc T^*)$ that preserve $\zeta$.
\end{itemize}
}
\end{cor}
\begin{proof}
(a) By Theorem \ref{thm:3.2}.c, (i) necessary for (ii). 

Suppose that (i) holds. Then $\eta_\tau (u')$ is equivalent to $u$ in $H^1 (F,\mc G^*_\ad)$,
so there exists $c \in \mc G^*_\ad (F_s)$ with $c \eta_\tau (u') c^{-1} = u$ in $Z^1 (F,\mc G^*_\ad)$.
Then Ad$(c) \circ \eta_\tau$ is an $F_s$-automorphism $\mc G^*_{u'} \to \mc G^*_u$, which induces
$\tau$ on $\mc R (\mc G^*, \mc T^*)$. By construction
\[
\mr{Ad}(c) \circ \eta_\tau (\mc G^*_{u}(F)) = \mr{Ad}(c) (\mc G^*_{\eta_\tau (u)}(F)) = \mc G^*_u (F), 
\]
so (ii) holds. \\
(b) Theorem \ref{thm:3.2}.b tells us that any $\eta$ as in (ii) is determined up to postcomposition
with $c' \in \mc G^*_\ad (F_s)$ which fixes $u \in Z^1 (F,\mc G^*_\ad)$. By \cite[Lemma 16.4.6]{Spr}
and \eqref{eq:3.6} such $c'$ form the group $\mc G^*_{\ad,u}(F)$.\\
(c) This is a direct consequence of parts (a) and (b) in the case $u = u'$.
\end{proof}

We will often denote the $F$-group $\mc G^*_u$ by $\mc G^*_\zeta$. Of course this is only correct
up to isomorphism. Technically, we can make it precise by choosing a splitting of
$Z^1 (F,\mc G^*_\ad) \to \Irr (Z({G^\vee}_\Sc)^{\mb W_F})$. 

\begin{prop}\label{prop:3.4}
Consider a connected reductive $F$-group $\mc G = \mc G^*_\zeta$ where $\mc G^*$ is a quasi-split 
$F$-group and $\zeta \in \Irr (Z({G^\vee}_\Sc)^{\mb W_F})$. Let $\mc S$ be a maximal $F$-split torus
of $\mc G$, $\mc T$ a maximal $F$-torus containing $\mc S$ and $\mc P_\emptyset$ a minimal parabolic
$F$-subgroup of $\mc G$ containing $\mc S$. Let $\tilde{\mc G}$ be another such group, with analogous
objects (endowed with tildes). Let 
\[
\tau : \mc R (\tilde{\mc G}^*,\tilde{\mc T}^*) \to \mc R (\mc G^*, \mc T^*)
\]
be a $\mb W_F$-equivariant isomorphism of based root data. The following are equivalent:
\begin{itemize}
\item[(i)] $\tau (\tilde \zeta) := \tilde \zeta \circ \eta_\tau^\vee$ equals $\zeta$;
\item[(ii)] there exists an isomorphism of $F$-groups $\eta : \tilde{\mc G} \to \mc G$ with 
$\mc R (\eta) = \tau$.
\end{itemize}
Suppose now that (ii) holds. Then
\enuma{
\item The group $G_\ad = \mc G_\ad (F)$ acts simply transitively on the set of $\eta$ as in (ii).
\item The group $G = \mc G (F)$ acts naturally on the set of $\eta$ as in (ii), with 
$|G_\ad / G|$ orbits. 
\item When $F$ is non-archimedean, let $K_0$ be the parahoric subgroup of $G$ attached to the 
origin of the standard apartment of the Bruhat--Tits building. When $F$ is archimedean, let
$K_0$ be a maximal compact subgroup of $G$ such that the Iwasawa decomposition holds with
respect to $K_0, \mc S (\R_{>0})$ and the unipotent radical of $\mc P_\emptyset$.

There exists an isomorphism of $F$-groups $\eta_{\tau,\mc G} : \tilde{\mc G} \to \mc G$
such that 
\begin{itemize}
\item $\eta_{\tau,\mc G} (\tilde{\mc S}) = \mc S$; 
\item $\eta_{\tau,\mc G}(\tilde{\mc P_\emptyset}) = \mc P_\emptyset$;
\item $\eta_{\tau,\mc G}(\tilde{K}_0) = K_0$.
\end{itemize}
\item The isomorphism $\eta_{\tau,\mc G}$ is unique up to $Z_{G_\ad} (S)_\cpt$, the maximal 
compact subgroup of $(Z_{\mc G}(\mc S) / Z(\mc G))(F)$. The maximal compact subgroup 
$Z_G (S)_\cpt$ of $Z_{\mc G}(\mc S)(F)$ acts on the set of $\eta_{\tau,\mc G}$ as in (c) with
$[Z_{G_\ad} (S)_\cpt : Z_G (S)_\cpt / Z(G)_\cpt]$ orbits.
}
\end{prop}
\begin{proof}
With the same argument as for Theorem \ref{thm:3.2}.a, we can find an isomorphims of 
$F_s$-groups $\eta_s : \tilde{\mc G}^* \to \mc G^*$ with $\mc R (\eta_s) = \tau$.
In Theorem \ref{thm:3.2}.c we showed that 
\[
\eta_s (\tilde{\mc G}) = \eta_s (\tilde{\mc G}^*_{\tilde \zeta}) = \mc G^*_{u'}
\]
for a $u' \in Z^1 (F,\mc G^*_\ad)$ with image $\tau (\tilde \zeta)$ in 
$\Irr (Z({G^\vee}_\Sc)^{\mb W_F})$. Now the equivalence of (i) and (ii) follows from Corollary
\ref{cor:3.3}.\\
(a) This follows from \eqref{eq:3.6} and Corollary \ref{cor:3.3}.\\
(b) Obvious from part (a) and $G_\ad / G := \mc G_\ad (F) / \mr{im}(\mc G (F) \to \mc G_\ad (F))$.\\
(c) Notice that $\mc S$ and $\eta (\tilde{\mc S})$ are both maximal $F$-split tori of $\mc G$,
and that $\mc P_\emptyset , \eta (\tilde{\mc P}_\emptyset)$ are minimal parabolic $F$-subgroups
of $\mc G$. By \cite[Theorem 15.4.6]{Spr} all minimal parabolic $F$-subgroups of $\mc G$ are
$\mc G (F)$-conjugate, while \cite[Theorem 15.2.6]{Spr} tells us that all maximal $F$-split
tori of $\mc P_\emptyset$ are $\mc P_\emptyset (F)$-conjugate. Hence, replacing $\eta$ by
$\tilde \eta = \mr{Ad}(\tilde c) \circ \eta$ for a suitable $\tilde c \in \mc G(F)$, we can
achieve that $\tilde \eta (\tilde{\mc S}) = \mc S$ and $\tilde \eta (\tilde{\mc P}_\emptyset) =
\mc P_\emptyset$. 

Assume that $F$ is non-archimedean. Then $\tilde \eta$ sends the root subgroup 
$\mc U_{\tilde \alpha}$ for $\tilde \alpha \in \Phi (\tilde{\mc G},\tilde{\mc S})$ to 
$\mc U_{\tau (\tilde \alpha)}$. The special parahoric subgroup $K_0$ of $G$ is described in terms 
of root subgroups in \cite[\S 6]{BrTi1}. Comparing 
with the analogous description of $\tilde{K}_0$, one sees that $\tilde \eta (\tilde K_0)$ is
a special parahoric subgroup of $G$ associated to a vertex of the apartment $\mh A_{\mc S}$
of $\mc{BT}(\mc G,F)$. Looking at the same situation over field extensions of $F$, we deduce
that $\tilde \eta (\tilde{K}_0)$ is hyperspecial if and only if $K_0$ is hyperspecial. From
the classification of root data one checks that $X_* (\mc S / Z_{\mc S}(\mc G))$ acts
transitively on the hyperspecial vertices in $\mh A_{\mc S}$ and on the special non-hyperspecial
vertices in $\mh A_{\mc S}$. As $\mc S / Z_{\mc S}(\mc G)$ embeds in $\mc G_\ad$, there exists a 
\[
s \in (\mc S / Z_{\mc S}(\mc G))(F) \subset \mc G_\ad (F) \text{ such that } 
\mr{Ad}(s) \circ \tilde \eta (\tilde{K}_0) = K_0.
\]
Define $\eta_{\tau,\mc G} = \mr{Ad}(s) \circ \tilde \eta = \mr{Ad}(s \tilde c) \circ \eta$.

Suppose now that $F$ archimedean. Since $\tilde \eta$ is an isomorphism and the Iwasawa 
decomposition holds for $\tilde G$ (as in the statement), it also holds for $G$ with respect
to $\tilde \eta (\tilde K_0) ,\; \mc S (\R_{>0})$ and the unipotent radical of $\mc P_\emptyset$.
All maximal compact subgroups of $G$ are conjugate, so this implies that $\tilde \eta (\tilde K_0)
= g^{-1} K_0 g$ for some $g$ in the joint normalizer of $\mc S$ and $\mc P_\emptyset$. 
Composing $\tilde \eta$ with Ad$(g)$ yields the desired $\eta_{\tau,\mc G}$.\\
(d) In view of part (b), $\eta_{\tau,\mc G}$ is unique up to the subgroup of $G_\ad$ which
normalizes $\mc S ,\mc P_\emptyset$ and $K_0$. The joint normalizer of $\mc S$ and 
$\mc P_\emptyset$ is $Z_{\mc G}(\mc S)$, that follows for instance from the Bruhat decomposition
\cite[Theorem 16.1]{Spr}. 

When $F$ is non-archimedean, the group $Z_{G_\ad}(S) = Z_{\mc G_\ad}(\mc S)(F)$ acts on the
apartment $\mh A_{\mc S}$, and $N_{Z_{G_\ad}(S)}(K_0)$ equals the isotropy group of $0 \in 
\mh A_{\mc S}$. Since $Z_{G_\ad}(S)$ acts on $\mh A_{\mc S}$ by translations, the isotropy group
of a point of $\mh A_{\mc S}$ equals the maximal compact subgroup $Z_{G_\ad}(S)_\cpt$.

When $F$ is archimedean, the Iwasawa decomposition shows that the normalizer of $K_0$ in
$G_\ad$ equals the preimage of $K_0$ under the canonical map $G_\ad \to G$. That is a
maximal compact subgroup of $G_\ad$, which implies  that $N_{G_\ad} (K_0) \cap Z_{G_\ad}(S)$ 
equals the maximal compact subgroup of $Z_{G_\ad}(S)$. This proves the unicity part. 

The conjugation action of $Z_G (S)$ (on the collection of subgroups of $G$) factors through
$Z_{G_\ad}(S)$, and ker$(Z_{\mc G} (\mc S) \to Z_{\mc G_\ad}(\mc S)) = Z(\mc G)$. Hence
ker$(Z_G (S)_\cpt \to Z_{G_\ad}(S)_\cpt) = Z (G)_\cpt$, and the final claim follows.
\end{proof}

Proposition \ref{prop:3.4} narrows down the choice of isomorphisms of $F$-groups which give
rise to one particular isomorphism of based root data. But it does not provide a unique
unambigous choice, only up to $Z_{G_\ad} (S)_\cpt \big (Z_G (S)_\cpt / Z(G)_\cpt)$, a subgroup
of $Z_{G_\ad} (S) / (Z_G (S) / Z(G)) \cong G_\ad / G$.

Let $\eta : \tilde{\mc G} \to \mc G$ be any isomorphism of connected reductive $F$-groups. 
Under the LLC, $\Phi_e ({}^L \eta) : \Phi_e (G) \to \Phi_e (\tilde G)$ should correspond 
to a bijection $\Irr (G) \to \Irr (\tilde G)$. In view of the above ambiguity, 
it cannot always correspond to 
\[
\begin{array}{cccc}
\eta^* : & \Irr (G) & \to & \Irr (\tilde G) \\
 & \pi & \mapsto & \pi \circ \eta 
\end{array}.
\]
We expect that $\Phi_e ({}^L \eta)$ corresponds to $\eta^*_{\mc R (\eta),\mc G}$, 
for a suitable choice of $\eta_{\mc R (\eta),\mc G}$ as in Proposition \ref{prop:3.4}.c. 
Proposition \ref{prop:3.4}.a tells us that there exists a unique $g \in \mc G_\ad (F)$
such that 
\begin{equation}\label{eq:3.9}
\eta = \mr{Ad}(g) \circ \eta_{\mc R (\eta),\mc G}.
\end{equation}
In that case we expect that under the LLC
\begin{equation}\label{eq:3.7}
\eta^* : \Irr (G) \to \Irr (\tilde G) \quad \text{corresponds to} \quad
\Phi_e ({}^L \eta) \circ \tau_{\mc G}(g)^{-1} : \Phi_e (G) \to \Phi_e (\tilde G),
\end{equation}
where $\tau_{\mc G}(g)$ twists enhancements of $\phi$ by $\tau_{\mc S_\phi}(g) = 
\tau_{\phi,\mc G}(g) \big|_{\mc S_\phi}$, as in \eqref{eq:2.20}.

\section{Quotients by central subgroups}
\label{sec:4}

Let $\tilde{\mc G}$ be a connected reductive $F$-group and let $\mc N$ be a central 
$F$-subgroup. Then $\mc G : = \tilde{\mc G} / \mc N$ is again a connected reductive 
$F$-group \cite[Corollary 12.2.2]{Spr}. To the quotient map
\[
q : \tilde{\mc G} \to \mc G
\]
one can associate a dual homomorphism
\[
q^\vee : G^\vee \to \tilde{G}^\vee.
\]
Since $G^\vee$ and $\tilde{G}^\vee$ have the same root system, $q^\vee (G^\vee)$ is
a normal subgroup of $\tilde G^\vee$. We note that
\begin{equation}\label{eq:4.1}
({\tilde G}^\vee )_\Sc = ({\tilde G}_\ad )^\vee = {(\tilde{\mc G} / \mc N)_\ad}^\vee = 
{G_\ad}^\vee = {G^\vee}_\Sc .
\end{equation}
The maps 
\[
{}^L q : G^\vee \rtimes \mb W_F \to \tilde G^\vee \rtimes \mb W_F
\quad \text{and} \quad
\Phi ({}^L q) : \Phi (G) \to \Phi (\tilde G)
\]
can be defined as in \eqref{eq:3.8} and \eqref{eq:3.1}. The effect of $q^\vee$ on 
enhancements of L-parameters is more complicated. 
Let $\phi : \mb W_F \times SL_2 (\C) \to G^\vee \rtimes \mb W_F$ be a L-parameter and
write ${}^L q (\phi) = \tilde \phi$. 

\begin{lem}\label{lem:4.7}
In the above setting $\mc S_{\phi}$ is a normal subgroup of $\mc S_{\tilde \phi}$
and $\mc S_{\tilde \phi} / \mc S_\phi$ is abelian.
\end{lem}
\begin{proof}
In view of \eqref{eq:4.1} there is an inclusion
\begin{equation}\label{eq:4.2}
Z^1_{{G^\vee}_\Sc} (\phi) \subset Z^1_{{G^\vee}_\Sc} (\tilde \phi) .
\end{equation}
Since $q^\vee (G^\vee)$ is normal in $\tilde G^\vee$, $Z^1_{{G^\vee}_\Sc} (\phi)$
is a normal subgroup of $Z^1_{{G^\vee}_\Sc} (\tilde \phi)$.
These groups need not be equal, for $q^\vee$ need not be injective and the superscript
1 has a different meaning on both sides. Nevertheless, we do have equalities
\begin{equation}\label{eq:4.3}
Z^1_{{G^\vee}_\Sc} (\phi)^\circ = Z_{{G^\vee}_\Sc} (\phi)^\circ = 
Z_{{G^\vee}_\Sc} (\tilde \phi)^\circ = Z^1_{{G^\vee}_\Sc} (\tilde \phi)^\circ .
\end{equation}
From \eqref{eq:4.3} we see that \eqref{eq:4.2} induces an isomorphism
\begin{equation}\label{eq:4.6}
\mc S_{\tilde \phi} / \mc S_\phi \cong Z^1_{{G^\vee}_\Sc} (\tilde \phi) / Z^1_{{G^\vee}_\Sc} (\phi) .
\end{equation}
In particular \eqref{eq:4.2} descends to an embedding of
$\mc S_\phi$ as a normal subgroup of $\mc S_{\tilde \phi}$.
Unwinding the definitions, we find
\begin{align*}
Z^1_{{G^\vee}_\Sc} (\tilde \phi) & = \{ g \in G^\vee_\Sc : g \tilde \phi g^{-1} = 
z \tilde \phi z^{-1} \text{ for some } z \in Z(\tilde{G}^\vee) \}\\
& = \{ g \in G^\vee_\Sc : g \phi g^{-1} = 
c z \phi z^{-1} \text{ for some } z \in Z(\tilde{G}^\vee), c \in Z^1 (\mb W_F, \ker q^\vee) \} .
\end{align*}
In the second line $z \phi z^{-1} = a_z \phi$ with $a_z \in Z^1 (\mb W_F, Z(G^\vee))$.
Similarly
\begin{align*}
Z^1_{{G^\vee}_\Sc} (\phi) & = \{ g \in G^\vee_\Sc : g \phi g^{-1} = 
z \phi z^{-1} \text{ for some } z \in Z(G^\vee) \} \\
& = \{ g \in G^\vee_\Sc : g \phi g^{-1} = b_z \phi \text{ for some } b_z \in B^1 (\mb W_F, Z(G^\vee)) \} .
\end{align*}
Comparing these characterizations and using \eqref{eq:4.6}, we obtain an injection
\[
\mc S_{\tilde \phi} / \mc S_\phi \to Z^1 (\mb W_F , Z(G^\vee)) / B^1 (\mb W_F, Z(G^\vee))
= H^1 (\mb W_F, Z(G^\vee)) .
\]
Since the right hand side is abelian, so is $\mc S_{\tilde \phi} / \mc S_\phi$. 
\end{proof}

For $\rho \in \Irr (\mc S_\phi)$ we put 
\[
{}^L q^* (\rho) = \mr{ind}_{\mc S_\phi}^{\mc S_{\tilde \phi}}(\rho) 
\in \mr{Rep}(\mc S_{\tilde \phi}) .
\]
Then ${}^L q^* (\rho)$ may very well be reducible, but that is only natural, as 
\[
q^* : \mr{Rep}(G) \to \mr{Rep}(\tilde G)
\]
need not preserve irreducibility either. Let $\Phi_{e+}(G)$ the set of pairs $(\phi,\rho)$ 
with $\phi \in \Phi (G)$ and $\rho \in \mr{Rep} (\mc S_\phi)$, considered
modulo $G^\vee$-conjugacy. With that notion $q^\vee$ induces a map
\[
\begin{array}{cccc}
\Phi_{e+}({}^L q) : & \Phi_{e+}(G) & \to & \Phi_{e+}(\tilde G) , \\
& (\phi,\rho) & \mapsto & ({}^L q \circ \phi, {}^L q^* (\rho)) .
\end{array}
\]
First we investigate this map on the cuspidal level. Let $\tilde {\mc L}$ be a Levi 
$F$-subgroup of $\tilde{\mc G}$, and let $\mc L = \tilde{\mc L} / \mc N$ be its
image in $\mc G$. We consider $(\phi_L,\rho_L) \in \Phi_\cusp (L)$, and we let 
$\rho_i \in \Irr (\mc S_{\tilde \phi_L})$ be a constituent of ${}^L q^* (\rho_L)$. 
Cuspidality of $\rho_i$ depends only on $\rho_i |_{Z_{{G^\vee}_\Sc}(\phi_L)}$, 
which is a sum of $\mc S_{\tilde \phi_L}$-conjugates of $\rho_L$. 
Those conjugates are cuspidal because $\rho_L$ is, so $\rho_i$ is cuspidal as well.

We write $\mf s^\vee$ (resp. $\mf s^\vee_L$) for the inertial equivalence class for 
$\Phi_e (G)$ (resp. for $\Phi_e (L)$) containing $(\phi_L,\rho_L)$ and $\mf s_i^\vee$  (resp.
$\mf s_{i,L}^\vee$) for the inertial class of $(\tilde \phi_L,\rho_i)$ (for $\Phi_e (\tilde{G})$
and for $\Phi_e (\tilde L)$, respectively). Decompose 
\begin{equation}\label{eq:4.26}
{}^L q^* (\rho_L) = \bigoplus\nolimits_i \rho_i^{m_i} \text{ with } m_i \in \Z_{>0}.
\end{equation}
We define ${}^L q (\mf s_L^\vee)$ as the union $\bigcup_i \mf s_{i,L}^\vee$, where 
$\mf s_{i,L}^\vee$ is identified with $\mf s_{j,L}^\vee$ if they are $(\tilde L)^\vee$-conjugate. 
For every $i$ there is a canonical map (in general neither injective nor surjective)
\begin{equation}\label{eq:4.4}
\mf s_L^\vee \to \mf s_{i,L}^\vee : (t \phi_L ,\rho_L) \mapsto (t \tilde \phi_L, \rho_i)
\qquad t \in X_\nr (L) .
\end{equation}
It induces an algebra homomorphism 
\begin{equation}\label{eq:4.7}
\mc O ({}^L q,\mf s_{i,L}^\vee) : \mc O (\mf s_{i,L}^\vee) \to \mc O (\mf s_L^\vee) .
\end{equation}
Then $\Phi_{e+}(\phi_L,\rho_L) = (\tilde \phi_L, \bigoplus_i \rho_i^{m_i})$ is the parameter
of the representation
\begin{equation}\label{eq:4.5}
\bigoplus\nolimits_i \mc O ({}^L q,\mf s_{i,L}^\vee)^* (\phi_L,\rho_L )^{m_i} 
\quad \text{of} \quad \sum\nolimits_i \mc O (\mf s_{i,L}^\vee) . 
\end{equation}
\textbf{Remark.} In many cases the multiplicities $m_i$ are one, cf. \cite{AdPr1}. For a 
counterexample to that among supercuspidal representations of quasi-split groups, see
\cite[Theorem 13]{AdPr}.

\subsection{Intermediate Hecke algebras} \

Returning to our general setting, we consider
\[
\mc H ({}^L q (\mf s^\vee),\vec{\mb z}) := 
\bigoplus\nolimits_i \mc H (\mf s_i^\vee,\vec{\mb z}) ,
\]
where $\mf s_i^\vee$ is identified with $\mf s_j^\vee$ if they are $\tilde G^\vee$-conjugate.
One complication immediately arises: in general there is no good homomorphism between
$\mc H (\mf s^\vee,\vec{\mb z})$ and $\mc H ({}^L q (\mf s^\vee),\vec{\mb z})$.
To work around that, we will introduce an intermediate Hecke algebra which is Morita
equivalent with $\mc H (\mf s^\vee,\vec{\mb z})$.

\begin{lem}\label{lem:4.1}
\enuma{
\item There are canonical identifications $\Phi_{\mf s^\vee} = \Phi_{\mf s_i^\vee}$
and $W_{\mf s^\vee}^\circ = W_{\mf s_i^\vee}^\circ$, and with respect to these the
label functions $\lambda, \lambda^*$ for both are equal.
\item $W_{\mf s^\vee}$ is canonically embedded in $W_{\mf s_i^\vee}$.
With respect to that embedding, the 2-cocycle $\kappa_{\mf s^\vee}$ is
the restriction of $\kappa_{\mf s_i^\vee}$.
}
\end{lem}
\textbf{Remark.} In general the inclusion $W_{\mf s^\vee} \subset W_{\mf s_i^\vee}$
can be proper, and $W_{\mf s_i^\vee}$ need not coincide with $W_{\mf s_j^\vee}$
(for another $\rho_j$).
\begin{proof}
(a) The root systems $\Phi_{\mf s^\vee}$ and $\Phi_{\mf s_i^\vee}$ depend only on
constructions in the group
\[
Z_{{G^\vee}_\Sc} (\phi (\mb W_F))^\circ = Z_{{G^\vee}_\Sc} (\tilde \phi (\mb W_F))^\circ ,
\]
so they can be identified. The same holds for the label functions $\lambda,\lambda^*$ 
and for the Weyl groups $W_{\mf s^\vee}^\circ$ and $W_{\mf s_i^\vee}^\circ$.\\
(b) Every $w \in W_{\mf s^\vee}$ can be represented by an element $\dot w \in N_{G^\vee}
(L^\vee \rtimes \mb W_F)$ which stabilizes $\mf s_L^\vee = (X_\nr (L) \phi_L, \rho_L)$. 
Then $\dot w$ also stabilizes ${}^L q^* (\rho_L)$, so it permutes the various $\rho_i$. 
As the different $\mf s_i^\vee$ are inertially inequivalent, this implies that 
$\dot w$ stabilizes every $\mf s_i^\vee$, and that $w \in W_{\mf s_i^\vee}$.

The twisted group algebras $\C [W_{\mf s^\vee},\kappa_{\mf s^\vee}]$ and
$\C [W_{\mf s_i^\vee},\kappa_{\mf s_i^\vee}]$ are both defined as the endomorphism algebra
of a certain perverse sheaf $q\pi_* (\widetilde{q\cE})$ \cite[Proposition 3.15.b]{AMS3}. 
The relevant perverse sheaf
for $\mf s^\vee$ is the restriction of the perverse sheaf for $\mf s_i^\vee$, from
$Z^1_{{G^\vee}_\Sc}(\tilde \phi |_{\mb I_F})$ to the finite index subgroup 
$Z^1_{{G^\vee}_\Sc}(\phi |_{\mb I_F})$. By construction, the projective actions of
$W_{\mf s^\vee}$ on these two perverse sheaves are related by the same restriction of
base spaces. Thus $\C [W_{\mf s^\vee},\kappa_{\mf s^\vee}]$ acts on both these
perverse sheaves, and embeds in $\C [W_{\mf s_i^\vee},\kappa_{\mf s_i^\vee}]$.
\end{proof}
 
Choose a set of representatives $[\mf R_{\mf s_i^\vee} / \mf R_{\mf s^\vee}]$ for 
$\mf R_{\mf s_i^\vee} / \mf R_{\mf s^\vee}$ in $\mf R_{\mf s_i^\vee}$. By Lemma 
\ref{lem:4.1}.a these are also representatives for $W_{\mf s_i^\vee} / W_{\mf s^\vee}$.
By the definition of $W_{\mf s^\vee}$, the sum
\begin{equation}\label{eq:4.9}
\sum\nolimits_{r \in [\mf R_{\mf s_i^\vee} / \mf R_{\mf s^\vee}]} 
\mc H (r \cdot \mf s^\vee, \vec{\mb z})
\end{equation}
is direct. Every $r \in \mf R_{\mf s_i^\vee}$ determines a canonical isomorphism
\[
\mc H (r' \cdot \mf s^\vee, \vec{\mb z}) \to \mc H (r r' \cdot \mf s^\vee, \vec{\mb z}) 
\]
sending $N_s$ to $N_{rsr^{-1}}$. In this way we define an algebra structure on
\begin{equation}\label{eq:4.25}
\mc H (\mf s^\vee, W_{\mf s_i^\vee}, \vec{\mb z})  :=
\bigoplus_{w \in [\mf R_{\mf s_i^\vee} / \mf R_{\mf s^\vee}]} 
\mc H (w \cdot \mf s^\vee, \vec{\mb z}) 
\underset{\C [\mf R_{\mf s^\vee}, \kappa_{\mf s^\vee}]}{\otimes}
\C [\mf R_{\mf s_i^\vee}, \kappa_{\mf s_i^\vee}] .
\end{equation}
It is easy to see that $\mc H (\mf s^\vee, W_{\mf s_i^\vee}, \vec{\mb z})$ contains
$\mc H (\mf s^\vee, \vec{\mb z})$ as a Morita equivalent subalgebra. In particular
the irreducible representations of these two algebras can be parametrized by the same set.
Fix $\vec{z} \in \R_{>1}^d$. Via Theorem \ref{thm:1.1} we associate to any 
$(\phi,\rho) \in \Phi_e (G)^{\mf s^\vee}$ the unique representation
\[
\overline{M}(\phi,\rho,\vec{z},W_{\mf s_i^\vee}) \in 
\Irr ( \mc H (\mf s^\vee, W_{\mf s_i^\vee}, \vec{\mb z}) ) 
\]
which appears in $\mr{ind}_{\mc H (\mf s^\vee, \vec{\mb z})}^{
\mc H (\mf s^\vee, W_{\mf s_i^\vee}, \vec{\mb z})} \overline{M}(\phi,\rho,\vec{z})$.
Similarly we define the standard module $\overline{E}(\phi,\rho,\vec{z},W_{\mf s_i^\vee}) 
\in \mr{Mod} ( \mc H (\mf s^\vee, W_{\mf s_i^\vee}, \vec{\mb z}) )$ by 
\[
\overline{E}(\phi,\rho,\vec{z},W_{\mf s_i^\vee})^{[\mf R_{\mf s_i^\vee} : \mf R_{\mf s^\vee}]} 
\cong \mr{ind}_{\mc H (\mf s^\vee, \vec{\mb z})}^{
\mc H (\mf s^\vee, W_{\mf s_i^\vee}, \vec{\mb z})} \overline{E}(\phi,\rho,\vec{z}).
\]
With Lemma \ref{lem:4.1} we can build a twisted affine Hecke algebra
\[
\mc H (\mf s_i^\vee, W_{\mf s_i^\vee}^\circ, \lambda, \lambda^*, \vec{\mb z})
\rtimes \C [\mf R_{\mf s_L^\vee},\kappa_{\mf s_L^\vee}] \subset
\mc H (\mf s_i^\vee, W_{\mf s_i^\vee}^\circ, \lambda, \lambda^*, \vec{\mb z})
\rtimes \C [\mf R_{\mf s_i^\vee},\kappa_{\mf s_i^\vee}] = \mc H (\mf s_i^\vee,\vec{\mb z}).
\]
The map \eqref{eq:4.7} (with $r \cdot \rho_L$ instead of $\rho_L$) 
extends to an algebra homomorphism
\begin{equation}\label{eq:4.8}
\mc H (\mf s_i^\vee, W_{\mf s_i^\vee}^\circ, \lambda, \lambda^*, \vec{\mb z})
\rtimes \C [\mf R_{\mf s^\vee},\kappa_{\mf s^\vee}] \to \mc H (r \cdot \mf s^\vee,\vec{\mb z}) ,
\end{equation}
which sends $N_w$ to $N_w$. With the homomorphisms \eqref{eq:4.8} for all 
$r \in [\mf R_{\mf s_i^\vee} / \mf R_{\mf s^\vee}]$, we map $\mc H (\mf s_i^\vee, 
W_{\mf s_i^\vee}^\circ, \lambda, \lambda^*, \vec{\mb z}) \rtimes \C [\mf R_{\mf s_L^\vee},
\kappa_{\mf s_L^\vee}]$ "diagonally" to \eqref{eq:4.9}. This extends to an algebra
homomorphism
\begin{equation}\label{eq:4.10}
\mc H (\mf s_i^\vee,\vec{\mb z}) \to \mc H (\mf s^\vee, W_{\mf s_i^\vee}, \vec{\mb z}) .
\end{equation}
We want to determine the pullbacks of $\overline{M}(\phi,\rho,\vec{z},W_{\mf s_i^\vee})$
and $\overline{E}(\phi,\rho,\vec{z},W_{\mf s_i^\vee})$
along \eqref{eq:4.10}. However, this map may change the $W_{\mf s_i^\vee}$-stabilizers of 
points of $\mf s_i^\vee$, which makes it a little more cumbersome to describe its effect on 
(irreducible) representations. To solve this problem, we will translate it to twisted graded 
Hecke algebras.

From now on we assume that $\phi_L \in \Phi (L)$ is bounded (which is hardly a restriction, 
it can always be achieved by an unramified twist). We write 
\[
M^\vee = L^\vee_c \cap Z^1_{{G^\vee}_\Sc}(\tilde \phi_L |_{\mb W_F}),
\]
a quasi-Levi subgroup of 
$\tilde{G}^\vee_{\tilde \phi_L} = Z^1_{{G^\vee}_\Sc}(\tilde \phi_L |_{\mb W_F})$. 
Recall Theorem \ref{thm:1.2} and the twisted graded affine Hecke algebras
\begin{equation}\label{eq:4.11} 
\begin{aligned}
& \mh H (\phi_L, \rho_L, \vec{\mb r}) = \mh H ( G^\vee_{\phi_L} \times
X_\nr (L), (L^\vee_c \cap G^\vee_{\phi_L}) \times X_\nr (L), \rho_L, \vec{\mb r}) , \\
& \mh H (\tilde \phi_L, \rho_i, \vec{\mb r}) = \mh H ( \tilde G^\vee_{\tilde \phi_L} \times
X_\nr (\tilde L),  M^\vee \times X_\nr (\tilde L), \rho_i, \vec{\mb r}) .
\end{aligned}
\end{equation}
Let $W_{\mf s_i^\vee, \tilde \phi_L}$ be the stabilizer of $\tilde \phi_L$ in 
$W_{\mf s_i^\vee}$, decompose it as in \eqref{eq:2.23}, with an R-group 
$\mf R_{\mf s_i^\vee, \tilde \phi_L}$ that is complementary to a Weyl group.
Analogous to the construction of $\mc H (\mf s^\vee, W_{\mf s_i^\vee},\vec{\mb z})$,
we can build an algebra
\[
\mh H (\phi_L, W_{\mf s_i^\vee, \tilde \phi_L}, \rho_L, \vec{\mb r}) := 
\sum_{r \in [\mf R_{\mf s_i^\vee, \tilde \phi_L} / \mf R_{\mf s^\vee, \phi_L}]}
\mh H (\phi_L, r \cdot \rho_L, \vec{\mb r}) 
\underset{\C [\mf R_{\mf s^\vee, \phi_L}, \kappa_{\mf s^\vee}]}{\otimes}
\C [\mf R_{\mf s_i^\vee, \tilde \phi_L}, \kappa_{\mf s_i^\vee}] . 
\]
We note that $\mh H (\phi_L, W_{\mf s_i^\vee, \tilde \phi_L}, \rho_L, \vec{\mb r})$
contains $\mh H (\phi_L, \rho_L, \vec{\mb r})$ as a Morita equivalent subalgebra.
To construct a good homomorphism from $\mh H (\tilde \phi_L, \rho_i, \vec{\mb r})$ to
$\mh H (\phi_L, W_{\mf s_i^\vee, \tilde \phi_L}, \rho_L, \vec{\mb r})$, we need more
intermediate algebras. 

Consider the twisted graded Hecke algebra $\mh H ( {\tilde G}^\vee_{\tilde \phi_L} 
\times X_\nr (L), M^\vee \times X_\nr (L), \rho_i, \vec{\mb r})$. In \cite[Theorem 4.6]{AMS2}
its irreducible representations, with $\vec{\mb r}$ acting as a fixed $\vec{r} \in \R^d$, are 
parametrized by triples $(s,v,\rho)$ such that:
\begin{itemize}
\item $s \in \mr{Lie}({\tilde G}^\vee_{\tilde \phi_L} \times X_\nr (L))$ semisimple,
\item $u \in \mr{Lie}({\tilde G}^\vee_{\tilde \phi_L} \times X_\nr (L))$ nilpotent with $[s,u]=0$,
\item $\rho \in \Irr \big( \pi_0 \big(Z_{{\tilde G}^\vee_{\tilde \phi_L}}
(s,u) \big) \big)$ with $q\Psi_{Z_{{\tilde G}^\vee_{\tilde \phi_L}}(s)}(u,\rho) =
(M^\vee,\log (u_{\phi_L}), \rho_i)$.
\end{itemize}
The natural map $X_\nr (L) \to X_\nr (\tilde L)$ induces an algebra homomorphism
\begin{equation}\label{eq:4.18}
\mh H (\tilde \phi_L, \rho_i, \vec{\mb r}) \to  \mh H ( {\tilde G}^\vee_{\tilde \phi_L} 
\times X_\nr (L), M^\vee \times X_\nr (L), \rho_i, \vec{\mb r}).
\end{equation}
We put $G^\vee_{\tilde \phi_L} = G^\vee_{\phi_L} M^\vee$, so that
\begin{equation}\label{eq:4.12}
G^\vee_{\tilde \phi_L} / G^\vee_{\phi_L} \cong M^\vee / \big( M^\vee \cap G^\vee_{\phi_L} \big) 
= \big( L^\vee_c \cap Z^1_{{G^\vee}_\Sc} (\tilde \phi_L |_{\mb W_F}) \big) \big/ 
\big( L^\vee_c \cap Z^1_{{G^\vee}_\Sc} (\phi_L |_{\mb W_F}) \big) .
\end{equation}
The idea is that on the cuspidal level $G^\vee_{\tilde \phi_L}$ comes from $\tilde \phi_L$,
while still giving the same groups $W_?$ as $\phi_L$.
With Clifford theory we can write $\rho_i = \rho_L \rtimes \tau_i$, where $\tau_i$ is a (projective)
representation of $(\mc S_{\tilde \phi_L})_{\rho_L} / \mc S_{\phi_L}$.

\begin{lem}\label{lem:4.6}
Let $(\phi,\rho) \in \Phi_e (G)^{\mf s_L^\vee}$ and write $\tilde \phi = {}^L q (\phi),
\mc S'_\phi = \pi_0 (Z_{G^\vee_{\tilde \phi_L}}(\tilde \phi))$.
\enuma{
\item There is a natural isomorphism $(\mc S'_\phi)_\rho / \mc S_\phi \cong 
(\mc S_{\tilde \phi_L})_{\rho_L} / \mc S_{\phi_L}$. The projective actions of these groups
on, respectively, $V_\rho$ and $V_{\rho_L}$ give the same 2-cocycles.
\item The decomposition of $\mr{ind}_{\mc S_\phi}^{\mc S'_\phi}(\rho)$ into irreducible
representations is $\bigoplus_i (\rho \rtimes \tau_i )^{m_i}$, where $\mr{ind}_{
\mc S_{\phi_L}}^{\mc S_{\tilde \phi_L}} (\rho_L) = \bigoplus_i (\rho_L \rtimes \tau_i )^{m_i}$.
}
\end{lem}
\begin{proof}
By the remarks after \eqref{eq:4.2}, $\mc S_\phi$ is a normal subgroup of $\mc S'_\phi$.\\ 
(a) We apply \cite[Lemma 4.4]{AMS2}, with $Q$ replaced by the possibly disconnected 
group $G^\vee_{\phi_L}$ and cuspidal supports replaced by quasi-supports. This is allowed
because all input for the proof of \cite[Lemma 4.4]{AMS2} is established in that generality
in \cite[\S 4]{AMS2}. It shows in particular that $\rho \times \tau_i$ is an irreducible 
representation of $\mc S'_\phi$. As noted in \cite[(102)]{AMS2}, the irreducibility implies 
that the canonical image of $(\mc S_{\tilde \phi_L})_{\rho_L} / \mc S_{\phi_L}$ in 
$\mc S'_\phi / \mc S_\phi$ is the stabilizer of $\rho$ in the latter group. That 
$\rho \times \tau_i$ is well-defined implies that the two 2-cocycles mentioned in part (a) 
agree, via the group isomorphism.\\
(b) This follows directly from part (a) and Clifford theory (see for instance 
\cite[Proposition 1.1]{AMS1}).
\end{proof}

Let $\vec{r} \in \C^d$. Recall from \cite[Theorem 3.8]{AMS3} 
\begin{equation}\label{eq:4.19}
\begin{aligned}
& E(\phi,\rho,\vec{r}) = \mr{IM}^* E_{\log (u_\phi), \sigma_L, \vec{r}, \rho} \in
\mr{Mod} (\mh H (\phi_L,\rho_L,\vec{\mb r})) ,\\
& \sigma_L = \log \big( \phi (\Fr)^{-1} \phi_L (\Fr) \big) + \textup{d}\vec{\phi} 
\matje{\vec r}{0}{0}{-\vec{r}} \in \mr{Lie} \big( G^\vee_{\tilde \phi_L} \times X_\nr (L) \big) .
\end{aligned}
\end{equation}

\begin{lem}\label{lem:4.2}
There exists a canonical isomorphism
\[
\mh H (\phi_L, \rho_L, \vec{\mb r}) \cong \mh H ( G^\vee_{\tilde \phi_L} \times
X_\nr (L), M^\vee \times X_\nr (L), \rho_i, \vec{\mb r}) .
\]
In the notation of \cite{AMS3}, the pullback of 
$E(\phi,\rho,\vec{r}) \in \mr{Mod} (\mh H (\phi_L, \rho_L, \vec{\mb r}) )$ along this
isomorphism is $\mr{IM}^* E_{\log (u_\phi), \sigma_L, \vec{r}, \rho \rtimes \tau_i}$, and
similarly for the irreducible modules labelled $M$ instead of $E$.
\end{lem}
\begin{proof}
From \eqref{eq:4.11} we see that the only differences between the two algebras are $L^\vee_c 
\cap G^\vee_{\phi_L}$ versus $M^\vee$ and $\rho_L$ versus $\rho_i = \rho_L \rtimes \tau_i$.
In particular the torus $T = Z(M^\vee)^\circ$ is the same for both algebras, and
so is the commutative subalgebra $\mc O (\mr{Lie} \, T)$. The finite groups underlying the 
algebras are
\begin{align*}
& W_{\rho_L} = N_{G^\vee_{\phi_L}}( L^\vee_c \cap G^\vee_{\phi_L}, \rho_L ) / (L^\vee_c 
\cap G^\vee_{\phi_L} ), \\
& W_{\rho_i} = N_{G^\vee_{\tilde \phi_L}}(M^\vee, \rho_i) / M^\vee.
\end{align*}
Here the normalizer of $\rho_L$ must be interpreted as the group which 
stabilizes the $L^\vee_c \cap G^\vee_{\phi_L}$-equivariant local system on the conjugacy
class of $u_\phi$ (and similarly for the normalizer of $\rho_i$).
As $M^\vee$ normalizes $\rho_i$, \eqref{eq:4.12} implies that the entire group 
$G^\vee_{\tilde \phi_L}$ normalizes $\rho_i$.
Furthermore $Z_{M^\vee}(u_\phi)$ acts transitively on the set of inequivalent irreducible
constituents of $\rho_i |_{\mc S_{\phi_L}}$. We get
\[
W_{\rho_i} = N_{G^\vee_{\tilde \phi_L}}(M^\vee, \rho_L, \rho_i) M^\vee / M^\vee  \cong
N_{G^\vee_{\tilde \phi_L}}(M^\vee, \rho_L)/ (M^\vee \cap G^\vee_{\phi_L}) = W_{\rho_L}.
\]
As explained in the proof of Lemma \ref{lem:4.6}, we may apply \cite[Lemma 4.4]{AMS2}. 
It says that, for any fixed $\vec{r} \in \C^d ,\; \rho \mapsto \rho \rtimes \tau_i$ provides 
a bijection between the triples parametrizing 
$\Irr_{\vec r} \big(\mh H (\phi_L, \rho_L, \vec{\mb r}) \big)$ and the triples parametrizing 
$\Irr_{\vec r} \big(\mh H ( G^\vee_{\tilde \phi_L} \times X_\nr (L), 
M^\vee \times X_\nr (L), \rho_i, \vec{\mb r}) \big)$.

By \cite[Lemma 4.5]{AMS2} (generalized in the same way), this map between parameters
corresponds to pulling back the standard modules along the algebra isomorphism. With 
the notations from \cite{AMS2,AMS3}, this gives the desired description of the pullback
of $E(\phi,\rho,\vec{r})$. The construction of $M(\phi,\rho,\vec{r})$ from 
$E(\phi,\rho,\vec{r})$, namely as a distinguished irreducible quotient 
\cite[Theorem 4.6.a]{AMS2}, proceeds in the same way for standard and irreducible modules of
$\mh H ( G^\vee_{\tilde \phi_L} \times X_\nr (L), M^\vee \times X_\nr (L), \rho_i, \vec{\mb r})$.
Hence the pullback of $M(\phi,\rho,\vec{r})$ is given by the same parameters.
\end{proof}

We note that Lemma \ref{lem:4.2} works equally well with, instead of $\rho_L$, any other 
irreducible consituent of $\rho_i |_{\mc S_\phi}$. Using it for all $r \in 
[\mf R_{\mf s_i^\vee, \tilde \phi_L} / \mf R_{\mf s^\vee, \phi_L}]$, we embed 
$\mh H ( G^\vee_{\tilde \phi_L} \times X_\nr (L), M^\vee \times X_\nr (L), 
\rho_i, \vec{\mb r})$ diagonally in
\begin{equation}\label{eq:4.15}
\sum_{r \in [\mf R_{\mf s_i^\vee, \tilde \phi_L} / \mf R_{\mf s^\vee, \phi_L}]}
\mh H (\phi_L, r \cdot \rho_L, \vec{\mb r}) \subset 
\mh H (\phi_L, W_{\mf s_i^\vee, \tilde \phi_L}, \rho_L, \vec{\mb r}) .
\end{equation}
The inclusion $G^\vee_{\tilde \phi_L} \subset \tilde G^\vee_{\tilde \phi_L}$ induces
a monomorphism
\begin{equation}\label{eq:4.16}
\mh H ( G^\vee_{\tilde \phi_L} \times X_\nr (L), M^\vee \times X_\nr (L), 
\rho_i, \vec{\mb r}) \to \mh H ( \tilde G^\vee_{\tilde \phi_L} \times X_\nr (L), 
M^\vee \times X_\nr (L), \rho_i, \vec{\mb r}) .
\end{equation}
With respect to \eqref{eq:4.16}, we can extend the embedding from \eqref{eq:4.15} 
to an algebra homomorphism 
\begin{equation}\label{eq:4.13}
\mh H ( \tilde G^\vee_{\tilde \phi_L} \times X_\nr (L), M^\vee \times X_\nr (L), 
\rho_i, \vec{\mb r}) \to \mh H (\phi_L, W_{\mf s_i^\vee, \tilde \phi_L}, \rho_L, \vec{\mb r}) ,
\end{equation}
which is the identity on $\C [\mf R_{\mf s_i^\vee, \tilde \phi_L}, \kappa_{\mf s_i^\vee}]$.
We record the composition of \eqref{eq:4.18} and \eqref{eq:4.13}:
\begin{equation}\label{eq:4.17}
\mh H (\tilde \phi_L, \rho_i, \vec{\mb r}) \to 
\mh H (\phi_L, W_{\mf s_i^\vee, \tilde \phi_L}, \rho_L, \vec{\mb r}) .
\end{equation}
We collect the maps \eqref{eq:4.18}, \eqref{eq:4.16}, \eqref{eq:4.13}, \eqref{eq:4.17}
and Lemma \ref{lem:4.2} in a commutative diagram
\begin{equation}\label{eq:4.20}
\xymatrix{
\mh H (\tilde \phi_L, \rho_i, \vec{\mb r}) \ar[d] \ar@{=}[r] &
 \mh H ( \tilde G^\vee_{\tilde \phi_L} \times
X_\nr (\tilde L),  M^\vee \times X_\nr (\tilde L), \rho_i, \vec{\mb r}) \ar[d] \\
\mh H (\phi_L, W_{\mf s_i^\vee, \tilde \phi_L}, \rho_L, \vec{\mb r}) &
\mh H ( \tilde G^\vee_{\tilde \phi_L} \times X_\nr (L),  M^\vee \times X_\nr (L), 
\rho_i, \vec{\mb r}) \ar@{_{(}->}[l] \\
\mh H (\phi_L, \rho_L, \vec{\mb r}) \ar@{_{(}->}[u] \ar@{=}[r]^-{\sim} & 
\mh H ( G^\vee_{\tilde \phi_L} \times X_\nr (L),  M^\vee \times X_\nr (L), 
\rho_i, \vec{\mb r}) \ar@{_{(}->}[u]
}
\end{equation}

\subsection{Pullbacks of modules} \

\begin{lem}\label{lem:4.3}
Let $\vec{r} \in \R^d$ and recall the notations \eqref{eq:4.19}.

The pullback of $E(\phi,\rho,\vec{r},W_{\mf s_i^\vee}) \in \mr{Mod}(\mh H (\phi_L, 
W_{\mf s_i^\vee, \tilde \phi_L}, \rho_L, \vec{\mb r}))$ along \eqref{eq:4.17} is  
\[
\bigoplus\nolimits_{\tilde \rho} 
\mr{Hom}_{\pi_0 \big( Z_{G^\vee_{\tilde \phi_L}}(\sigma_L,\log u_\phi) \big)}
( \rho \rtimes \tau_i, \tilde \rho ) \otimes E(\phi,\tilde \rho, \vec{r}) ,
\]
where the sum runs over all $\tilde \rho \in \Irr (\mc S_{\tilde \phi})$ with
cuspidal quasi-support \cite[\S 5]{AMS1}\\ 
$q\Psi_{Z_{\tilde G^\vee_{\tilde \phi_L}}(\sigma_L)}
(\log u_\phi, \tilde \rho) = (M^\vee, \log (u_{\phi_L}), \rho_L \rtimes \tau_i)$.
The pullback of the irreducible module $M(\phi,\rho,\vec{r},W_{\mf s_i^\vee})$ 
can be expressed in the same way.
\end{lem}
\begin{proof}
From the diagram \eqref{eq:4.20} we see that we have to determine the pullback
of $E(\phi,\rho,\vec{r},W_{\mf s_i^\vee})$ to $\mh H ( \tilde G^\vee_{\tilde \phi_L} 
\times X_\nr (L),  M^\vee \times X_\nr (L), \rho_i, \vec{\mb r})$. The composition
of that pullback operation with the Morita equivalence 
\begin{equation}\label{eq:4.22}
\mr{Mod} \big( \mh H (\phi_L, \rho_L, \vec{\mb r}) \big) \to
\mr{Mod} \big( \mh H (\phi_L, W_{\mf s_i^\vee, \tilde \phi_L}, \rho_L, \vec{\mb r})\big)
\end{equation}
is the route $(3,1) \to (2,1) \to (2,2)$ in the diagram. From the explicit construction
of $\mh H (\phi_L, W_{\mf s_i^\vee, \tilde \phi_L}, \rho_L, \vec{\mb r})$ and the map
\eqref{eq:4.13}, we see that this operation agrees with pullback along 
Lemma \ref{lem:4.2} followed by induction along \eqref{eq:4.16}. That is the alternative
path $(3,1) \to (3,2) \to (2,2)$ in the diagram.

By \eqref{eq:1.3} $\mc S'_\phi$ is naturally isomorphic to 
$\pi_0 (Z_{G^\vee_{\tilde \phi_L}}(\sigma_L,u_\phi))$.
Now Lemma \ref{lem:4.2} entails that we can just as well determine the induction of
$\mr{IM}^* E_{\log (u_\phi), \sigma_L, \vec{r}, \rho \rtimes \tau_i}$ along
the lower right injection in the diagram. Now we are in the right position to apply
\cite[Proposition 1.5]{AMS3}, the crucial, highly-nontrivial step which justifies the 
entire setup. It says that the induced module can be expressed as
\begin{equation}\label{eq:4.21}
\bigoplus\nolimits_{\tilde \rho} \mr{Hom}_{\mc S'_\phi}
\big( \rho \rtimes \tau_i, \tilde \rho \big) \otimes \mr{IM}^*
E_{\log (u_\phi), \sigma_L, \vec{r}, \tilde \rho} .
\end{equation}
Here $\tilde \rho$ runs over those irreducible representations of  
$\pi_0 \big( Z_{\tilde G^\vee_{\tilde \phi_L}}(\sigma_L,\log u_\phi) \big)$ with the
condition on the cuspidal quasi-support as statement. By \eqref{eq:1.3}, this component 
group can be identified with $\mc S_{\tilde \phi}$.

In view of the commutativity of the diagram \eqref{eq:4.20}, it only remains to pullback
along the upper right map \eqref{eq:4.18}. That homomorphism only changes $X_\nr (L)$ to
$X_\nr (\tilde L)$, so it sends \eqref{eq:4.21} the module expressed in almost the same
way, only $\sigma_L$ interpreted as an element of Lie$\big( {\tilde G}^\vee_{\tilde \phi_L}
\times X_\nr (\tilde L) \big)$. With \eqref{eq:4.19} we can rewrite $\mr{IM}^* 
E_{\log (u_\phi), \sigma_L, \vec{r}, \tilde \rho}$ as $E(\phi,\tilde \rho,\vec{r})$,
which transforms \eqref{eq:4.21} into the desired expression. 

The same argument applies to $M(\phi,\rho,\vec{r},W_{\mf s_i^\vee})$. 
\end{proof}

In \eqref{eq:4.5} we regarded $\bigoplus\nolimits_i \mc O ({}^L q,\mf s_{i,L}^\vee)^* 
(\phi_L,\rho_L )^{m_i}$ as the pullback, along ${}^L q$, of $(\phi_L,\rho_L)$ to $\sum_i 
\mc H (\mf s_{i,L}^\vee, \vec{\mb z})$. Here the sum runs over the isotypical components
$\rho_i$ of $\mr{ind}_{\mc S_{\phi_L}}^{\mc S_{\tilde \phi_L}} \rho_L$. To be consistent 
with that when pulling back on the non-cuspidal level, we must consider the modules from 
Lemma \ref{lem:4.3} with multiplicity $m_i$.

Since the Morita equivalence \eqref{eq:4.22} is canonical in the given setting, we can 
regard $m_i$ times the module in Lemma \ref{lem:4.3} also as the canonical pullback of 
$E(\phi,\rho,\vec{r})$ via the diagram \eqref{eq:4.20}. Analogous considerations apply 
to $M(\phi,\rho,\vec{r})$. We write
\[
\mh H \big( {}^L q (\phi_L,\rho_L), \vec{\mb r} \big) := 
\sum\nolimits_i \mh H (\tilde \phi_L, \rho_i, \vec{\mb r}) .
\]

\begin{lem}\label{lem:4.4}
Accept the above interpretations of pullbacks. The pullback of $E(\phi,\rho,\vec{r})$
to $\mh H ({}^L q (\phi_L,\rho_L), \vec{\mb r})$ is
\[
\bigoplus\nolimits_{\tilde \rho \in \Irr (\mc S_{\tilde \phi})} 
\mr{Hom}_{\mc S_{\tilde \phi}}\big( \mr{ind}_{\mc S_\phi}^{
\mc S_{\tilde \phi}} \rho, \tilde \rho \big) \otimes E(\phi,\tilde \rho, \vec{r}) .
\]
This same holds for $M(\phi,\rho,\vec{r}) \in \Irr (\mh H (\phi_L,\rho_L,\vec{\mb r}))$.
\end{lem}
\begin{proof}
Lemma \ref{lem:4.6}.b and the above specification implies that the pullback of 
$E(\phi,\rho,\vec{r})$ is the sum, over all irreducible constituents $\rho_i$ of 
$\mr{ind}_{\mc S_{\phi_L}}^{\mc S_{\tilde \phi_L}} \rho_L$, of 
\begin{equation}\label{eq:4.23}
\bigoplus\nolimits_{\tilde \rho} \mr{Hom}_{\mc S'_\phi}
( (\rho \rtimes \tau_i)^{m_i}, \tilde \rho ) \otimes E(\phi,\tilde \rho, \vec{r}) ,
\end{equation}
with the conditions on $\tilde \rho$ as in Lemma \ref{lem:4.3}.
The conventions regarding representations constructed with Clifford theory \cite[\S 1]{AMS1}
say that $\rho \rtimes \tau_i$ as $\mc S_{\tilde \phi}$-representation equals 
$\mr{ind}^{\mc S_{\tilde \phi}}_{\mc S'_\phi} (\rho \rtimes \tau_i)$.
By Frobenius reciprocity we may replace $\mr{Hom}_{\mc S'_\phi}$ in
\eqref{eq:4.23} by $\mr{Hom}_{\mc S_{\tilde \phi}}$. When we sum over $i$, the condition
on the cuspidal quasi-support of $(\log u_\phi,\tilde \rho)$ from Lemma \ref{lem:4.3} becomes
that is conjugate to $(M^\vee, \log (u_{\phi_L}), \rho_L \rtimes \tau_i)$ for some $i$. 
Since $\tilde \rho$ must contain $\rho$ for \eqref{eq:4.23} to be nonzero, and since
\[
q\Psi_{Z_{G^\vee_{\phi_L}}(\sigma_L)}(\log u_\phi,\rho) = (M^\vee \cap G^\vee_{\phi_L},
\log (u_{\phi_L}), \rho_L),
\]
that condition on $\tilde \rho$ is automatic from \eqref{eq:4.23}. 
These constructions give us the pullback
\begin{equation}\label{eq:4.24}
\bigoplus\nolimits_i \bigoplus\nolimits_{\tilde \rho \in \Irr (\mc S_{\tilde \phi})} 
\mr{Hom}_{\mc S_{\tilde \phi}}
( (\rho \rtimes \tau_i )^{m_i}, \tilde \rho ) \otimes E(\phi,\tilde \rho, \vec{r}) .
\end{equation}
As $\mr{Hom}_{\mc S_{\tilde \phi}}( (\rho \rtimes \tau_i )^{m_i}, \tilde \rho )$ is nonzero
for at most one $i$, we may place $\oplus_i$ inside $\mr{Hom}_{\mc S_{\tilde \phi}}$ in
\eqref{eq:4.24}. By Lemma \ref{lem:4.6}.b 
\begin{equation}\label{eq:4.29}
\oplus_i (\rho \rtimes \tau_i )^{m_i} = \mr{ind}_{\mc S'_\phi}^{\mc S_{\tilde \phi}} 
\big( \oplus_i (\rho \rtimes \tau_i )^{m_i} \big) = \mr{ind}_{\mc S'_\phi}^{\mc S_{\tilde \phi}} 
\big( \mr{ind}_{\mc S_\phi}^{\mc S'_\phi} \rho \big) =  
\mr{ind}_{\mc S_\phi}^{\mc S_{\tilde \phi}} (\rho) ,
\end{equation}
so the above produces the desired expression.
\end{proof}

Now we are ready to formulate the pullback along ${}^L q$ in terms of 
twisted affine Hecke algebras. 
Let $\mc H (\mf s^\vee, W_{\mf s_i^\vee, \tilde \phi_L}, \phi_L, \vec{\mb z})$
be the algebra constructed from $\mc H (\mf s^\vee, \phi_L, \vec{\mb z})$ like 
$\mc H (\mf s^\vee, W_{\mf s_i^\vee}, \vec{\mb z})$ was constructed from
$\mc H (\mf s^\vee, \vec{\mb z})$ in \eqref{eq:4.25}, but this time only using the 
subgroup of $W_{\mf s_i^\vee}$ which stabilizes $\tilde \phi_L$. We note that 
\[
W_{\mf s_i^\vee, \tilde \phi_L} \supset W_{\mf s_i^\vee, \phi_L}
\]
by the $W_{\mf s_L^\vee}$-equivariance of \eqref{eq:4.4}. There are canonical homomorphisms
\begin{equation}\label{eq:4.28}
\mc H (\mf s^\vee, \phi_L, \vec{\mb z}) \to
\mc H (\mf s^\vee, W_{\mf s_i^\vee, \tilde \phi_L}, \phi_L, \vec{\mb z}) \leftarrow
\mc H (\mf s_i^\vee, \tilde \phi_L,\vec{\mb z}) ,
\end{equation}
where the first is the inclusion of a Morita equivalent subalgebra, and the second is
induced by \eqref{eq:4.4}.

\begin{thm}\label{thm:4.5}
Interpret pullbacks from $\mc H (\mf s^\vee, \vec{\mb z})$ to $\mc H ({}^L q (\mf s^\vee),
\vec{\mb z}) = \sum_i \mc H (\mf s_i^\vee,\vec{\mb z})$ as in Lemma \ref{lem:4.4} and 
\eqref{eq:4.10}, so via the algebras $\mc H (\mf s^\vee, W_{\mf s_i^\vee}, \vec{\mb z})$
and with multiplicities $m_i$ coming from the cuspidal level, as in \eqref{eq:4.5}.

Let $\vec{z} \in \R_{>0}^d$ and $(\phi,\rho) \in \Phi_e (G)^{\mf s^\vee}$. The 
pullback of $\overline{M}(\phi,\rho,\vec{z}) \in \Irr (\mc H (\mf s^\vee,\vec{\mb z}))$ is
\[
\bigoplus\nolimits_{\tilde \rho \in \Irr (\mc S_{\tilde \phi})} 
\mr{Hom}_{\mc S_{\tilde \phi}}\big( \mr{ind}_{\mc S_\phi}^{
\mc S_{\tilde \phi}} \rho, \tilde \rho \big) \otimes \overline{M}(\phi,\tilde \rho, \vec{z}) .
\]
The analogous for formula for the standard module $\overline{E}(\phi,\rho,\vec{z}) \in 
\mr{Mod} (\mc H (\mf s_L^\vee,\vec{\mb z}))$ holds as well.
\end{thm}
\begin{proof}
Recall from \eqref{eq:2.45} that every twisted graded Hecke algebra $\mh H$ admits an 
"analytic" version $\mh H^{an}$, obtained by replacing the algebra of regular functions on 
the appropriate complex vector space by the algebra of analytic function on that complex 
variety. As noted in \cite[\S 1.5]{SolAHA}, $\mh H^{an}$ has the same irreducible 
representations as $\mh H$. This also applies to $\mh H (\phi_L, W_{\mf s_i^\vee, 
\tilde \phi_L}, \rho_L, \vec{\mb r})$, when we embed $\mc O (\mr{Lie}(G^\vee_{\phi_L} 
\times X_\nr (L)) \times \C^d)$ diagonally in it, and then replace that subalgebra by 
$C^{an} (\mr{Lie}(G^\vee_{\phi_L} \times X_\nr (L)) \times \C^d)$.

Consider the commutative diagram
\begin{equation}\label{eq:4.27}
\xymatrix{
\mc H (\mf s^\vee, \vec{\mb z}) \ar@{_{(}->}[r] &
\mc H (\mf s^\vee, W_{\mf s_i^\vee}, \vec{\mb z}) &
\mc H (\mf s_i^\vee,\vec{\mb z}) \ar[l] \\
\mc H (\mf s^\vee, \phi_L, \vec{\mb z}) \ar@{_{(}->}[r] \ar@{_{(}->}[d] \ar@{_{(}->}[u] &
\mc H (\mf s^\vee, W_{\mf s_i^\vee, \tilde \phi_L}, \phi_L, \vec{\mb z}) \ar@{_{(}->}[d] \ar@{_{(}->}[u] &
\mc H (\mf s_i^\vee, \tilde \phi_L,\vec{\mb z}) \ar[l] \ar@{_{(}->}[d] \ar@{_{(}->}[u] \\
\mh H^{an} (\phi_L, \rho_L, \vec{\mb r}) \ar@{_{(}->}[r] &
\mh H^{an} (\phi_L, W_{\mf s_i^\vee, \tilde \phi_L}, \rho_L, \vec{\mb r}) &
\mh H^{an} (\tilde \phi_L, \rho_i, \vec{\mb r}) \ar[l] \\
\mh H (\phi_L, \rho_L, \vec{\mb r}) \ar@{_{(}->}[r] \ar@{_{(}->}[u]  &
\mh H (\phi_L, W_{\mf s_i^\vee, \tilde \phi_L}, \rho_L, \vec{\mb r}) \ar@{_{(}->}[u] &
\mh H (\tilde \phi_L, \rho_i, \vec{\mb r}) \ar[l] \ar@{_{(}->}[u] 
} 
\end{equation}
Here the arrows from $\mc H (?)$ to $\mh H^{an}(?)$ are induced by the map $\exp_{\phi_L}$, 
as in \eqref{eq:2.35}. The two outer arrows (of those three) induce equivalences between categories
of certain finite dimensional modules, determined by the requirement that the weights with respect
to a commutative subalgebra must lie in a certain region. These categories are given as (ii) and
(iii) before and after \eqref{eq:2.35}. The same kind of equivalence of categories holds for the
middle arrow in the diagram \eqref{eq:4.27}. One can see that by walking along the path $(2,2) \to
(2,1) \to (3,1) \to (3,2)$. The two vertical steps in that path are Morita equivalences and we
just observed that $(2,1) \to (3,1)$ provides the desired kind of equivalence between categories
of some modules. 

As remarked after Theorem \ref{thm:1.2}, $\overline{M}(\phi,\rho,\vec{z})$ comes from the 
$\mh H (\phi_L,\rho_L,\vec{\mb r})$-module $M(\phi,\rho,\vec{\log z})$.
The fourth row of the diagram \eqref{eq:4.27} is the left column of
\eqref{eq:4.20}, and we know from Lemma \ref{lem:4.4} how to pull back along that. With the 
equivalence of categories of finite dimensional modules of $\mh H (?)$ and $\mh H^{an}(?)$,
we can also handle pullback along the third row. Using the aforementioned equivalence between
certain module categories between the third and second row, we can transfer the result of
Lemma \ref{lem:4.4} to the second row in the diagram. But that no longer applies to all 
$\vec{r}$, we need to assume that $\vec{r} \in \R^d$. That is accounted for by the assumption
$\vec{z} \in \R_{>0}^d$.

To get to the first row of the diagram, we use the equivalences between the categories denoted
(i) and (ii) and the start of Paragraph \ref{par:2.3}. Thus we showed that the image of
$\overline{M}(\phi,\rho,\vec{z})$ under the Morita equivalence $(1,1) \to (1,2)$ and pullback
along $(1,2) \to (1,3)$ is
\[
\bigoplus\nolimits_{\tilde \rho} \mr{Hom}_{\mc S'_\phi}
( \rho \rtimes \tau_i, \tilde \rho ) \otimes \overline{M}(\phi,\tilde \rho, \vec{z}) .
\]
Now the same argument as for Lemma \ref{lem:4.4} proves that the pullback of 
$\overline{M}(\phi,\rho,\vec{z})$ to $\sum_i \mc H (\mf s_i^\vee,\vec{\mb z})$
has the desired shape. The same reasoning works for $\overline{E}(\phi,\rho,\vec{z})$.
\end{proof}

\section{Homomorphisms of reductive groups with commutative (co)kernel}
\label{sec:5}

Let $F$ be any local field. In this section we take a closer look at Conjecture
\ref{conj:A} for a homomorphism of connected reductive $F$-groups
\begin{equation}\label{eq:5.9}
f : \tilde{\mc G} \to \mc G .
\end{equation}
We assume that $f$ satisfies Condition \ref{cond:1}, that is, the kernel of 
d$f : \mr{Lie}(\tilde{\mc G}) \to \mr{Lie}(\mc G)$ is central and the cokernel of 
$f$ is a commutative $F$-group. Let us relate this to similar conditions:

\begin{lem}\label{lem:5.5}
Under the Conditions \ref{cond:1}
\enuma{
\item The kernel of $f$ is defined over $F$ and central in $\tilde{\mc G}$.
\item The restriction of \eqref{eq:5.9} to the derived groups
\[
f_\der : \tilde{\mc G}_\der \to \mc G_\der  
\]
is a central isogeny. 
}
\end{lem}
\begin{proof}
(a) By condition (i) the Lie algebra of ker$ (f)^\circ$ is central in $\mr{Lie}(\tilde{\mc G})$.
With \cite[Lemma 22.2]{Bor2} this implies that ker$(f)^\circ$ is contained in $Z(\tilde{\mc G})^\circ$. 
Further ker$ (f)$ is $F$-closed because $f$ is defined over $F$. Then \cite[Lemma 22.1]{Bor2}
says that ker$ (f)$ is defined over $F$ and central in $\tilde{\mc G}$.\\
(b) The  condition (ii) includes that coker$(f)$ is defined as $F$-group, so
$f(\tilde{\mc G})$ must be a closed normal $F$-subgroup of $\mc G$. Thus  
$f(\tilde{\mc G})$ contains $\mc G_\der$. Since a homomorphism from a torus
(e.g. $Z(\tilde{\mc G})^\circ$) to a simple group can never be surjective, 
$f(\tilde{\mc G}_\der)$ already contains $\mc G_\der$. The other conditions for
a central isogeny (see e.g. \cite[\S 22]{Bor2}) are part (a) and condition (i), 
restricted to $\tilde{\mc G}_\der$.
\end{proof}

As in Sections \ref{sec:3} and \ref{sec:4}, the dual homomorphism $f^\vee : G^\vee \to
\tilde G^\vee$ and ${}^L f = (f^\vee, \mr{id}_{\mb W_F})$ induce a canonical map
\[
\begin{array}{cccc}
\Phi ({}^L f) : & \Phi (G) & \to & \Phi (\tilde G) \\
& \phi & \mapsto & \tilde \phi := {}^L f \circ \phi 
\end{array}.
\]
By Lemma \ref{lem:5.5}.b $f$ induces a homomorphism of $F$-groups 
$f_\ad : \tilde{\mc G}_\ad \to \mc G_\ad$. 

\begin{lem}\label{lem:5.4}
The characters $\tau_{\phi,\mc G}(g)$ from Lemma \ref{lem:2.1} are compatible with the 
maps ${}^L f$ and $f_\ad$, in the following sense:
\[
\tau_{\phi,\mc G}( f_\ad (\tilde g)) (h) = \tau_{{}^L f \circ \phi, \tilde{\mc G}}(\tilde g) 
(f^\vee (h)) \qquad \tilde g \in \tilde{\mc G}_\ad , h \in (G^\vee )^{\phi (\mb W_F)} .
\]
\end{lem}
\begin{proof}
Let $\phi_\Sc : \mb W_F \to {G^\vee}_\Sc \rtimes \mb W_F$ be a lift of $\phi$. Then
${}^L f_\ad \circ \phi_\Sc$ is a lift of ${}^L f \circ \phi$. From \eqref{eq:2.16} we see
that $c_{f^\vee (h)} = {}^L f_\ad \circ c_h$. Recall from \eqref{eq:2.15} that
\[
\tau_{{}^L f \circ \phi, \tilde{\mc G}}(\tilde g) (f^\vee (h)) = 
\langle \tilde g, c_{f^\vee (h)} \rangle = \langle \tilde g, {}^L f_\ad \circ c_h \rangle .
\]
By the functoriality of the LLC for tori, the right hand side equals
$\langle f_\ad (\tilde g) , c_h \rangle = \tau_{\phi,\mc G}( f_\ad (\tilde g)) (h)$.
\end{proof}

We will show how $f$ produces, from an enhanced L-parameter for $G = \mc G (F)$, 
an L-parameter for $\tilde G = \tilde{\mc G}(F)$ with a possibly reducible enhancement. 
To make this canonical, we have to assume:
\begin{cond}\label{cond:3}
For every involved isomorphism of based root data $\tau$ there exists an isomorphism 
of reductive $F$-groups $\eta_{\tau,\mc G}$ as in Proposition \ref{prop:3.4}.c, which 
is canonical for the collection of L-parameters under consideration.
\end{cond}

The cokernel of $f$ is commutative, so $f(\tilde{\mc G})$ contains the derived group of
$\mc G$. Consequently $f(Z(\tilde{\mc G}))$ is contained in the centre of $\mc G$.
By \cite[Proposition 1.8]{BoTi}, applied to the $F$-torus $Z(\mc G)^\circ$, there
exists a central $F$-torus $\mc T \subset \mc G$ such that $\mc G = f(\tilde{\mc G}) \mc T$
and $f(\tilde{\mc G}) \cap \mc T$ is finite. It is easily seen that the homomorphism
$f \times \mr{id}_{\mc T} : \tilde{\mc G} \times \mc T \to \mc G$ also satisfies 
Condition \ref{cond:1}. Now we can factorize $f$ as
\begin{equation}\label{eq:5.1}
\tilde{\mc G} \xrightarrow{f_1} \tilde{\mc G} \times \mc T \xrightarrow{f_2} 
(\tilde{\mc G} \times \mc T) / \ker (f \times \mr{id}_{\mc T}) \xrightarrow{f_3} \mc G.
\end{equation}
Here $f_1$ is the inclusion of one factor in a direct product, $f_2$ is a quotient map for
a central $F$-subgroup and $f_3$ is an isomorphism of $F$-groups coming from 
$f \times \mr{id}_{\mc T}$.

\begin{prop}\label{prop:5.1}
Assume Condition \ref{cond:1} and let $\phi \in \Phi (G)$.
\enuma{
\item ${}^L f$ induces a canonical injection $\mc S_\phi \to \mc S_{\tilde \phi} = 
\mc S_{{}^L f (\phi)}$. The image is a normal subgroup and the cokernel is abelian.
\item Assuming Condition \ref{cond:3}, $f$ induces a canonical algebra homomorphism
${}^S f : \C [\mc S_\phi] \to \C[\mc S_{\tilde \phi}]$, which: 
\begin{itemize}
\item sends every $s \in \mc S_\phi$ to a nonzero element of $\C {}^L f (s)$,
\item is the identity on $\C [\mc Z_\phi]$.
\end{itemize}
\item Suppose that $q : \mc G \to \mc H$ is another homomorphism satisfying 
Conditions \ref{cond:1} and \ref{cond:3}. Then 
${}^L (q \circ f) = {}^L f \circ {}^L q$ and ${}^S (q \circ f) = {}^S f \circ {}^S q$.
}
\end{prop}
\begin{proof}(a) 
We use the factorization \eqref{eq:5.1} to establish the existence. We treat each homomorphism
$f_i$ separately.

For $f_1$ we can write $\phi = \tilde \phi \times \phi_T$ with $\phi_T \in \Phi (T)$. 
The component group of any L-parameter for the torus $T = \mc T (F)$ is trivial, so $\mc S_\phi
= \mc S_{\tilde \phi}$ and ${}^L f_1$ fixes this group pointwise.

For $f_2$ see Lemma \ref{lem:4.7}. 

For the isomorphism $f_3$ the map ${}^L f_3 : \mc S_\phi \to \mc S_{\tilde \phi}$ is just the 
bijection $f_3^\vee |_{Z^1_{{G^\vee}_\Sc}(\phi)}$. With these notions the group homomorphism
\begin{equation}\label{eq:5.3}
{}^L f = {}^L f_2 \circ {}^L f_3 : \mc S_\phi \to \mc S_{\tilde \phi}
\end{equation}
is injective and canonically determined by ${}^L f : {}^L G \to {}^L \tilde G$. The image of 
\eqref{eq:5.3} is normal with abelian cokernel because that holds for the separate factors, 
see also \cite[Lemma 5]{AdPr}.\\
(b) We define ${}^S f_1$ as the identity on $\C [\mc S_\phi]$ and ${}^S f_2 : \C[\mc S_\phi] 
\to \C [\mc S_{\tilde \phi}]$ as the $\C$-linear extension of 
${}^L f_2 : \mc S_\phi \to \mc S_{\tilde \phi}$.
Assuming Condition \ref{cond:3}, we may invoke \eqref{eq:3.9}. It rewrites 
$f_3 = \mr{Ad}(g) \circ \eta_{\mc R (f_3), \mc G}$ and defines
\[
\begin{array}{cccc}
{}^S f_3 : & \C[\mc S_\phi] & \to & \C [\mc S_{\tilde \phi}] \\
& s & \mapsto & \tau_{\mc G}(g)(s) f_3^\vee (s) 
\end{array}.
\]
By Lemma \ref{lem:2.1} ${}^S f_3 (s) = f_3^\vee (s)$ for any $s \in \mc Z_\phi$.
For later use, we record that by \eqref{eq:2.20} the 
pullback of $\tilde \rho \in \mr{Rep} (\mc S_{\tilde \phi})$ along ${}^S f_3$ is
\begin{equation}\label{eq:5.2}
{}^S f_3^* (\tilde \rho) = (\tilde \rho \circ f_3^\vee) \otimes \tau_{\phi}(g) . 
\end{equation}
The algebra monomorphism
\begin{equation}\label{eq:5.4}
{}^S f := {}^L f_2 \circ {}^L f_3 \circ \tau_{\mc G}(g) : 
\C [\mc S_\phi] \to \C [\mc S_{\tilde \phi}]
\end{equation}
could still depend on the choice of the factorization \eqref{eq:5.1}. To show its canonicity,
we consider another torus $\mc T'$ with the same properties as $\mc T$. Then $\mc T \mc T'$ 
is yet another such subtorus of $\mc G$, and we extend \eqref{eq:5.1} to a commutative diagram
\begin{equation}\label{eq:5.5}
\xymatrix{
\tilde{\mc G} \ar[r]^{f_1} \ar@{=}[d] & \tilde{\mc G} \times \mc T \ar[r]^-{f_2} \ar[d] & 
(\tilde{\mc G} \times \mc T) / \ker (f \times \mr{id}_{\mc T}) \ar[r]^-{f_3} \ar[d] & \mc G \ar@{=}[d] \\
\tilde{\mc G} \ar[r]^{f'_1} & \tilde{\mc G} \times \mc T \mc T' \ar[r]^-{f'_2} & 
(\tilde{\mc G} \times \mc T \mc T') / \ker (f \times \mr{id}_{\mc T \mc T'}) \ar[r]^-{f'_3} & \mc G
}
\end{equation}
The inclusion $\mc T \to \mc T \mc T'$ induces isomorphisms 
\begin{align*}
& \mc S_{{}^L f_3 \phi} \to \mc S_{{}^L f'_3 (\phi)} ,\\
& \mc S_{{}^L f_2 \circ {}^L f_3 \phi} \to \mc S_{{}^L f_2 \circ {}^L f'_3 (\phi)} .
\end{align*}
Conjugation with these two isomorphisms turns ${}^S f_3$ into ${}^S f'_3$ and 
${}^S f_2$ into ${}^S f'_2$, respectively. As the based root data of 
$(\tilde{\mc G} \times \mc T) / \ker (f \times \mr{id}_{\mc T})$ and
$(\tilde{\mc G} \times \mc T \mc T') / \ker (f \times \mr{id}_{\mc T \mc T'})$ can be identified
canonically, the vertical maps in \eqref{eq:5.5} transform $\eta_{\mc R (f_3),\mc G}$ into
$\eta_{\mc R (f'_3),\mc G}$. Hence
\[
\mr{Ad}(g) = f_3 \circ \eta_{\mc R (f_3),\mc G}^{-1} = f'_3 \circ \eta_{\mc R (f'_3),\mc G}^{-1}
\] 
and the diagram transforms ${}^S f_3 = {}^L f_3 \circ \tau_{\mc G}(g)$ into
${}^S f'_3 = {}^L f'_3 \circ \tau_{\mc G}(g)$. It follows that 
\[
{}^L f'_2 \circ {}^L f'_3 \circ \tau_{\mc G}(g) = {}^L f_2 \circ {}^L f_3 \circ \tau_{\mc G}(g) :
\C [\mc S_\phi] \to \C [\mc S_{\tilde \phi}] .
\]
This also holds with the roles of $\mc T$ and $\mc T'$ exchanged, so the factorizations 
\eqref{eq:5.1} from $\mc T$ and from $\mc T'$ give rise to the same map
${}^S f : \C [\mc S_\phi] \to \C [\mc S_{\tilde \phi}]$.\\
(c) Recall that $f^\vee$ is uniquely determined by $f$ and the choice of pinnings of $G^\vee$
and of $\tilde G^\vee$. The same goes for $q^\vee$. As both $f^\vee \circ q^\vee$ and 
$(q \circ f)^\vee$ send the pinning for $H^\vee$ to the pinning for $\tilde G^\vee$, they
are equal. Then also ${}^L \circ {}^L q = {}^L (q \circ f) : {}^L H \to {}^L \tilde G$.

Suppose that $\phi' \in \Phi (H)$ and $\phi = {}^L q \circ \phi'$. Part (a) says that
${}^L \circ {}^L q  = {}^L (q \circ f)$ as homomorphisms 
$\mc S_{\phi'} \to \mc S_{\tilde \phi}$. Factorize $q$ as in \eqref{eq:5.1}, with $\mc T$
renamed $\mc A$. Consider the commutative diagram
\[
\xymatrix{
\tilde{\mc G} \ar@{_{(}->}[r] \ar@{_{(}->}[d] & 
\tilde{\mc G} \times \mc T \ar@{->>}[r] \ar@{_{(}->}[dl] & 
(\tilde{\mc G} \times \mc T ) / \ker (f \times \mr{id}_{\mc T}) \ar@{_{(}->}[dl] \ar[r]^-{\sim} & 
\mc G \ar@{_{(}->}[dl] \\
\tilde{\mc G} \times \mc T \times \mc A \ar@{->>}[r] \ar@{->>}[dr] & 
\frac{\tilde{\mc G} \times \mc T }{\ker (f \times \mr{id}_{\mc T})} \times 
\mc A \ar@{->>}[d] \ar[r]^-{\sim} & \mc G \times \mc A \ar@{->>}[r] &
\frac{\mc G \times \mc A}{\ker (q \times \mr{id}_{\mc A})}  \ar[d]^-{\sim} \\
& (\tilde{\mc G} \times \mc T \times \mc A) / \mc N  \ar[urr]^-{\sim}  \ar[rr]^-{\sim}  
& & \mc H
}
\]
where $\mc N = \big( \ker (f \times \mr{id}_{\mc T}) \times \mc A \big) (f_3^{-1} \times \mr{id}_{\mc A}) 
(\ker \eta \times \mr{id}_{\mc A})$. The path from $\tilde{\mc G}$
to $\mc H$ along the upper left corner is the composition of the factorizations of $f$ and of $q$. 
The lower left track from $\tilde{\mc G}$ to $\mc H$ is a factorization of $q \circ f$ as in 
\eqref{eq:5.1}. In particular we obtain the lower horizontal map
$(q \circ f)_3 = q_3 \circ (f_3 \times \mr{id}_{\mc A})$.
Use \eqref{eq:3.9} to rewrite $q_3 = \mr{Ad}(h) \circ \eta_{\mc R (q_3), \mc H}$, so that
\begin{multline}\label{eq:5.11}
(q \circ f)_3 = \mr{Ad}(h) \circ \eta_{\mc R (q_3), \mc H} \circ 
\mr{Ad}(g) \circ \big( \eta_{\mc R (f_3), \mc G} \times \mr{id}_{\mc A} \big)\\
= \mr{Ad} \big( h \eta_{\mc R (q_3), \mc H} (g) \big) \circ 
\eta_{\mc R (q_3 \circ (f_3 \times \mr{id}_{\mc A})), \mc H} .
\end{multline}
Here $g \in G_\ad$ is also regarded as an element of the adjoint group of entry (2,4) in the above 
diagram. In other words, we identify $\eta_{\mc R (q_3), \mc H} (g)$ with 
$\eta_{\mc R (q_3), \mc H} \circ q_2 \circ q_1 (g)$. Lemma \ref{lem:5.4} tells us that 
\[
\tau_{\mc G}(g) \circ {}^L q (x) = \tau_{\phi,\mc G}(g)(q^\vee (x)) = 
\tau_{\phi',\mc H}(q_\ad (g)) (x) .
\]
Now we can compute
\begin{equation}\label{eq:5.10}
\begin{aligned}
{}^S f \circ {}^S q & = {}^L f \circ \tau_{\mc G}(g) \circ {}^L q \circ \tau_{\mc H}(h) \\
& = {}^L f \circ {}^L q \circ \tau_{\mc H}(q_\ad (g)) \circ \tau_{\mc H}(h) =
{}^L (q \circ f) \circ \tau_{\mc H}(q_\ad (g) h) . 
\end{aligned}
\end{equation}
Up to conjugation by an element of $H_\ad$, $q_\ad (g)$ equals 
$\eta_{\mc R (q_3), \mc H} \circ q_2 \circ q_1 (g)$. That conjugation does not change 
$\tau_{\mc H}(q_\ad (g))$, so the right hand side of \eqref{eq:5.10} equals
\begin{equation}\label{eq:5.12}
{}^L (q \circ f) \circ \tau_{\mc H}(h \eta_{\mc R (q_3), \mc H} \circ q_2 \circ q_1 (g)) =
{}^L (q \circ f) \circ \tau_{\mc H}(h \eta_{\mc R (q_3), \mc H} (g)) .
\end{equation}
Comparing with \eqref{eq:5.4} and \eqref{eq:5.11}, we see that \eqref{eq:5.12} 
equals ${}^S (q \circ f)$.
\end{proof}

Borel \cite[\S 10.3]{Bor} predicted that the LLC is functorial with respect to homomorphisms 
as in \eqref{eq:5.9}. This concerns the pullback of $\mc G(F)$-representations to 
$\tilde{\mc G}(F)$. For non-archimedean $F$, we will precisely formulate and prove such 
functoriality for representations of Hecke algebras associated to enhanced L-parameters. 
First we set up the homomorphisms between Hecke algebras induced by $f$ and ${}^L f$. 
Let $(\phi_L,\rho_L) \in \Phi_\cusp (L)$ be bounded and recall the notations from 
Section \ref{sec:4}.

\begin{prop}\label{prop:5.2}
Let $\rho_i \in \Irr (\mc S_{\tilde \phi})$ be such that ${}^S f^* (\rho_i)$ contains $\rho$.
Assume Conditions \ref{cond:3} and \ref{cond:1}.
\enuma{
\item There exist a twisted graded Hecke algebra $\mh H (\phi_L, W_{\mf s_i^\vee, \tilde{\phi_L}},
\rho_L, \vec{\mb r})$ and algebra homomorphisms
\[
\mh H (\phi_L, \rho_L, \vec{\mb r}) \xrightarrow{\imath_{(\phi_L,\rho_L)}} \mh H (\phi_L, 
W_{\mf s_i^\vee, \tilde{\phi_L}}, \rho_L, \vec{\mb r}) 
\xleftarrow{\mh H ({}^L f,\tilde \phi_L,\rho_i)} \mh H (\tilde \phi_L, \rho_i, \vec{\mb r})
\]
such that:
\begin{itemize}
\item[(i)] $\imath_{(\phi_L,\rho_L)}$ is the inclusion of a Morita equivalent subalgebra.
\item[(ii)] $\mh H ({}^L f,\tilde \phi_L,\rho_i)$ is induced by $f$ and ${}^L f$.
\item[(iii)] Consider the functor
\[
\mr{Mod}(\mh H (\phi_L, \rho_L, \vec{\mb r})) \to
\mr{Mod}(\mh H (\tilde \phi_L, \rho_i, \vec{\mb r})) ,
\]
obtained by composing the Morita equivalence from $\imath_{(\phi_L,\rho_L)}$ with pullback
along $\mh H ({}^L f,\tilde \phi_L,\rho_i)$. It is canonical, in the sense that it depends on 
${}^S f, f^\vee, \phi_L$ and $\rho_i$, but not on 
$\mh H (\phi_L, W_{\mf s_i^\vee, \tilde{\phi_L}}, \rho_L, \vec{\mb r})$.
\end{itemize}
\item Assume Condition \ref{cond:2} holds for $\mf s_i^\vee$. There exist a twisted affine
Hecke algebra $\mc H (\mf s^\vee, W_{\mf s_i^\vee}, \vec{\mb z})$ and algebra homomorphisms
\[
\mc H (\mf s^\vee, \vec{\mb z}) \xrightarrow{\imath_{\mf s^\vee}} 
\mc H (\mf s^\vee, W_{\mf s_i^\vee}, \vec{\mb z}) \xleftarrow{\mc H ({}^L f, \mf s_i^\vee)}
\mc H (\mf s_i^\vee, \vec{\mb z}) ,
\]
which satisfy the analogues of (i)--(iii). 
}
\end{prop}
\begin{proof}
We will use the factorization \eqref{eq:5.1}. An argument with the diagram \eqref{eq:5.5},
as in the proof of Proposition \ref{prop:5.1}, shows that the outcome does not depend on the 
choice of the factorization. Hence it suffices to prove the claims separately for the 
homomorphisms $f_1, f_2$ and $f_3$ from \eqref{eq:5.1}.\\
(a) For $f_2$ see Paragraphs \ref{par:2.2}--\ref{par:2.3}, in particular \eqref{eq:4.20} 
and the proof of Lemma \ref{lem:4.3}. 

For $f_1$ and $f_3$ we simply take
$\mc H (\mf s^\vee, W_{\mf s_i^\vee}, \vec{\mb z}) = \mc H (\mf s^\vee, \vec{\mb z})$, so
that (i) is automatic. We write $f_3 = \mr{Ad}(g) \circ \eta_{\mc R (f_3),\tilde{\mc G}}$
as in \eqref{eq:3.9}, which by Condition \ref{cond:3} is canonical. Then we put
\[
\mh H ({}^L f_3, \tilde \phi_L,\rho_i) = 
\mh H ({}^L \eta_{\mc R (f_3), \tilde{\mc G}}) \circ \alpha_g
\]
with $\alpha_g$ as in Proposition \ref{prop:2.4} and $\mh H ({}^L \eta_{\mc R (f_3), 
\tilde{\mc G}})$ as in \eqref{eq:3.2}. These maps are constructed in terms of $f_3$ and
${}^L f_3$, so (ii) holds. Claim (iii) for $f_3$ is a consequence of Lemma \ref{lem:2.5}
and Corollary \ref{cor:3.1}.

For $f_1$ we can write $\phi_L = \tilde \phi_L \times \phi_T$ with $\phi_T \in \Phi (T)$.
Then $\mc S_{\tilde \phi_L} = \mc S_{\phi_L}, \rho_i = \rho_L$ and
\[
\mh H (\tilde \phi_L, \rho_i , \vec{\mb r}) =
\mh H (\phi_L, \rho_L, \vec{\mb r}) \otimes \mc O \big( \mr{Lie}(X_\nr (T)) \big) .
\]
The homomorphism $\mh H ({}^L f_1,\tilde \phi_L,\rho_i)$ is just
\begin{equation}\label{eq:5.6} 
\mr{id} \otimes \mr{ev}_0 : \mh H (\phi_L, \rho_L, \vec{\mb r}) \otimes 
\mc O \big( \mr{Lie}(X_\nr (T)) \big) \to \mh H (\phi_L, \rho_L, \vec{\mb r}) .
\end{equation}
The pullback along this map is simply restriction from $\mh H (\tilde \phi_L, \rho_i , 
\vec{\mb r})$ to its subalgebra $\mh H (\phi_L, \rho_L, \vec{\mb r})$. Clearly, this
is canonical and determined by $f_1$.\\
(b) Here we can use the same argument as for part (a), only with different references.
For $f_2$ see \eqref{eq:4.25}, \eqref{eq:4.10} and Theorem \ref{thm:4.5}. The canonicity
follows from the commutative diagram \eqref{eq:4.27} and the canonicity in part (a).

For $f_3$ we replace $\alpha_g$ by $\alpha_g \circ \mr{Ad}(x_g)$ from 
\eqref{eq:2.37} and we replace Proposition \ref{prop:2.4} and Lemma \ref{lem:2.5} by
Theorem \ref{thm:2.9} (for which we need Condition \ref{cond:2}).

For $f_1$ we have
\[
\mc H (\mf s^\vee_i, \vec{\mb z}) = \mc H (\mf s^\vee, \vec{\mb z}) \otimes \mc O (\mf s_T^\vee)
\text{ where } \mf s_T^\vee = (X_\nr (T) \phi_T, 1) .
\]
As $\mc H ({}^L f_1, \mf s_i^\vee)$ we take 
\begin{equation}\label{eq:5.8}
\mr{id} \otimes \mr{ev}_{(\phi_T,1)} : \mc H (\mf s^\vee, \vec{\mb z}) \otimes 
\mc O (\mf s_T^\vee) \to \mc H (\mf s^\vee, \vec{\mb z}). \qedhere
\end{equation}
\end{proof}

Recall from Section \ref{sec:4} that for quotients by central subgroups we regarded \\
$\bigoplus\nolimits_i \mc O ({}^L q,\mf s_{i,L}^\vee)^* (\phi_L,\rho_L )^{m_i}$ as the pullback, 
along ${}^L q$, of $(\phi_L,\rho_L)$ to $\mc H ({}^L q (\mf s^\vee), \vec{\mb z})$. The same holds
for homomorphisms like $f_1$ and $f_3$ in \eqref{eq:5.1} (but then there is just one $i$ and
$m_i = 1$). In Lemma \ref{lem:4.2} and \eqref{eq:4.29} we showed that the same multiplicities
pop up in the pullbacks of enhancements on the non-cuspidal level.

For consistency, we must consider the canonical pullbacks from Proposition 
\ref{prop:5.2}.(iii) with multiplicities $m_i$ coming from 
\begin{equation}\label{eq:5.7}
\mr{ind}_{{}^L f (\mc S_\phi)}^{\mc S_{\tilde \phi}} 
\big( \rho_L \otimes \tau_{\mc G}(g) \big) = \bigoplus\nolimits_i \rho_i^{m_i} .
\end{equation}
From the construction of ${}^S f$ in Proposition \ref{prop:5.1} we see that 
\[
m_i = \dim \mr{Hom}_{\mc S_{\phi_L}}(\rho_L, {}^S f^* (\rho_i )).
\]

\begin{thm}\label{thm:5.3}
Let $\mf s^\vee$ be the inertial equivalence class for $\Phi_e (G)$ determined by
$(\phi_L, \rho_L) \in \Phi_\cusp (L)$. Let $\vec{z} \in \R_{>0}^d, \vec{r} = \log (\vec{z})$,
and $(\phi, \rho) \in \Phi_e (G)^{\mf s^\vee}$. 
Recall from Theorem \ref{thm:1.1} that 
\[
\overline{M}(\phi,\rho,\vec{z}), \overline{E}(\phi,\rho,\vec{z}) \in 
\mr{Rep}(\mc H (\mf s^\vee, \vec{\mb z}))
\]
and from Theorem \ref{thm:1.2} that (possibly after modifying $\phi_L$ by an unramified twist)
\[
M(\phi,\rho,\vec{r}), E(\phi,\rho,\vec{r}) \in \mr{Rep}(\mh H (\phi_L ,\rho_L ,\vec{\mb r})) .
\]
Assume Conditions \ref{cond:3} and \ref{cond:1} and endow the canonical pullbacks from
Proposition \ref{prop:5.2} with the multiplicities from \eqref{eq:5.7}.
\enuma{
\item The pullback of $M(\phi,\rho,\vec{r})$ to $\mh H ({}^L f (\phi_L,\rho_L), \vec{\mb r})
= \bigoplus_i \mh H (\tilde \phi_L, \rho_L, \vec{\mb r})$ is
\[
\bigoplus\nolimits_{\tilde \rho \in \Irr (\mc S_{\tilde \phi})} \Hom_{\mc S_\phi} 
\big( \rho, {}^S f^* (\tilde \rho) \big) \otimes M(\tilde \phi, \tilde \rho, \vec{r}) .
\]
\item Assume that Condition \ref{cond:2} holds for $\tilde{\mc G}$. The pullback of
$\overline{M}(\phi,\rho,\vec{z})$ to \\ $\mc H ({}^L f (\mf s^\vee),\vec{\mb z}) = 
\bigoplus\nolimits_i \mc H (\mf s_i^\vee,\vec{\mb z})$ is
\[
\bigoplus\nolimits_{\tilde \rho \in \Irr (\mc S_{\tilde \phi})} \Hom_{\mc S_\phi} 
\big( \rho, {}^S f^* (\tilde \rho) \big) \otimes \overline{M}(\tilde \phi, \tilde \rho, \vec{z}) .
\]
}
In (a) and (b) the same holds for the standard modules $E(\phi,\rho,\vec{r})$
and $\overline{E}(\phi,\rho,\vec{z})$.
\end{thm}
\begin{proof}
By Proposition \ref{prop:5.1} the statements are transitive for compositions of 
homomorphisms between reductive groups. Using the factorization \eqref{eq:5.1} it 
suffices to establish the theorem for the homomorphisms $f_1, f_2$ and $f_3$ separately.

For $f_3$ see \eqref{eq:5.2}, Lemma \ref{lem:2.5}, Theorem \ref{thm:2.9} and Corollary 
\ref{cor:3.1}.

For $f_2$ see Lemma \ref{lem:4.4} and Theorem \ref{thm:4.5}.

For $f_1$ we note that $\phi = \tilde \phi \times \chi \phi_T$ for some $\chi \in X_\nr (T)$.
In particular $\mc S_{\tilde \phi} = \mc S_\phi$. It is clear from \eqref{eq:5.6} that the
pullback of $M(\phi,\rho,\vec{r})$ (resp. $E(\phi,\rho,\vec{r})$) is
$M(\tilde \phi,\rho,\vec{r})$ (resp. $E(\tilde \phi,\rho,\vec{r})$). As ${}^S f_1 :
\C[ \mc S_\phi] \to \C [\mc S_{\tilde \phi}]$ is the identity,
\[
\bigoplus\nolimits_{\tilde \rho \in \Irr (\mc S_{\tilde \phi})} 
\mr{Hom}_{\mc S_\phi} \big( \rho, {}^S f_1^* (\tilde \rho) \big) = \left\{
\begin{array}{cl}
\C & \text{if } \tilde \rho = \rho ,\\
0 & \text{otherwise} .
\end{array}
\right.
\]
This matches $M(\tilde \phi,\rho,\vec{r})$ (and $E(\tilde \phi,\rho,\vec{r})$) with the asserted
pullback. When Condition \ref{cond:2} holds, we can draw the same conclusions about $f_1$ from
\eqref{eq:5.8}.
\end{proof}

We expect that in many cases a local Langlands correspondence can be described in terms of 
Hecke algebras, such that $(\phi,\rho)$ corresponds to $\overline{M}(\phi,\rho,\vec{z})$ 
or to $M(\phi,\rho,\vec{r})$, for suitable arrays of parameters $\vec{z}$ and $\vec{r}$. 
Having worked out what pullback along $f$ does to Langlands parameters and representations
of the associated Hecke algebras, we are in a good position to investigate the relations 
between a LLC and the pullback
\[
f^* : \Rep (G) \to \Rep (\tilde G) .
\]
Now $F$ may again be any local field. To get compatible notions of relevance, we extend the 
Kottwitz parameters $\zeta_{\mc G} = \zeta_{\tilde{\mc G}} \in \Irr (Z({G^\vee}_\Sc)^{\mb W_F})$ 
in the same way to $\zeta_{\mc G}^+ = \zeta_{\tilde{\mc G}}^+ \in \Irr (Z({G^\vee}_\Sc))$.
We need to assume that a LLC is known for all representations
of $G$ and of $\tilde G$ that will be involved. We denote the $G$-representation associated
to $(\phi,\rho) \in \Phi_e (G)$ by $\pi (\phi,\rho)$. Abbreviate
$\tilde \phi = {}^L f (\phi) \in \Phi (\tilde G)$ and recall ${}^S f : \C [\mc S_\phi] \to
\C [\mc S_{\tilde \phi}]$ from Proposition \ref{prop:5.1}.b. 

The preparations in this section lead to a more precise version of Conjecture \ref{conj:A}:

\begin{conj}\label{conj:5.7}
Assume that $f : \tilde{\mc G} \to \mc G$ satisfies Condition \ref{cond:1} and define 
relevance of enhanced L-parameters as above. For any $(\phi,\rho) \in \Phi_e (G)$:
\[
f^* (\pi (\phi,\rho)) = \bigoplus_{\tilde \rho \in \Irr (\mc S_{\tilde \phi})} \mr{Hom}_{\mc S_\phi} 
\big( \rho, {}^S f^* (\tilde \rho) \big) \otimes \pi (\tilde \phi, \tilde \rho) .
\]
\end{conj}

In view of $(q \circ f)^* = f^* \circ q^*$ and Proposition \ref{prop:5.1}.c, the conjecture
is transitive in $f$. Theorem \ref{thm:5.3} says that, when $F$ is non-archimedean, it holds 
for the Hecke algebras associated to enhanced L-parameters. 
A first consequence of Conjecture \ref{conj:5.7} is:

\begin{cor}\label{cor:5.8}
Suppose that Conjecture \ref{conj:5.7} holds for $(G,\phi)$ and $(\tilde G,\tilde \phi)$
(that is, for all relevant enhancements of these L-parameters). Then the L-packet
$\Pi_{\tilde \phi}(\tilde G)$ consists precisely of the irreducible constituents of 
$f^* \pi$ with $\pi \in \Pi_\phi (G)$. 
\end{cor}
\begin{proof}
First we note that by \cite{Sil} and \cite[Lemma 2.1]{Tad} $f^*$ preserves finite length 
and complete reducibility. In particular, $f^* \pi$ is a finite direct sum of irreducible
representations whenever $\pi$ is so. Conjecture \ref{conj:5.7} shows
that $f^*$ maps any $\pi \in \Pi_\phi (G)$ to a direct sum of members of 
$\Pi_{\tilde \phi}(\tilde G)$. 

For every $\tilde \rho \in \Irr (\mc S_{\tilde \phi})$ we can find a $\rho \in 
\Irr (\mc S_\phi)$ with $\mr{Hom}_{\mc S_\phi}(\rho, {}^S f^* (\tilde \rho)) \neq 0$. When 
$\tilde \rho$ is $\tilde G$-relevant, the $Z(\tilde{G}^\vee_\Sc)$-character is prescribed 
(namely as $\zeta_{\mc G}^+$). As $Z({G^\vee}_\Sc) = Z(\tilde{G}^\vee_{\ \,\Sc})$, 
and ${}^S f$ is the identity on $\C [\mc Z_\phi]$ (Proposition \ref{prop:5.1}.b), $\rho$ can only
appear in ${}^S f^* (\tilde \rho)$ if $\rho$ has the same $\mc Z_\phi$-character as 
$\tilde \rho$ (namely $\zeta_{\mc G}^+$). Hence the $\rho$ we found above is necessarily
$G$-relevant. In other words, every member $\pi (\tilde \phi, \tilde \rho)$ of 
$\Pi_{\tilde \phi}(\tilde G)$ appears as a constituent of $f^* \pi (\phi,\rho)$ for some
$\pi (\phi,\rho) \in \Pi_\phi (G)$.
\end{proof}

\section{The principal series of split groups} 
\label{sec:principal}

Let $\mc G$ be a connected reductive group. We fix a Haar measure on $G$, so that we can define 
convolution products of functions on $G$. For any open subgroup $U \subset G$ we let $\mc H (U)$ 
be the convolution algebra of locally constant, compactly supported functions $U \to \C$.

Throughout this section we assume that $\mc G$ is $F$-split, and we pick a $F$-split maximal 
torus $\mc S$ of $\mc G$ and a Borel subgroup $\mc B$ containing $\mc S$.
By the principal series of $G = \mc G (F)$ we understand those $G$-representations that can be 
made from subquotients of $I_B^G (\chi)$, where $\chi$ is a character of a maximal $F$-split 
torus $S$ (inflated to $B = \mc B (F))$. We denote the set of irreducible principal series
representations of $G$ by $\Irr (G,S)$. In \cite{ABPSprin} such representations were classified,
and a local Langlands correspondence for them was constructed. This relies on the types and
Hecke algebras which were exhibited and investigated by Roche \cite{Roc}. 
To make use of these results, we must impose certain mild restrictions on the residual 
characteristic $p$ of $F$, see \cite[p. 378--379]{Roc}. (Probably these conditions can be
relaxed, but we do not try that here.)

Every Bernstein component for the principal series of $G$ is given by an inertial equivalence
class $\mf s = [S,\chi]_G$. Let $\hat \chi \in \Phi (S)$ be the Langlands parameter of $\chi$.
The LLC for tori also associates to $\mf s$ an inertial equivalence class 
$\mf s^\vee = [S^\vee, \hat \chi,1]$ for $\Phi_e (G)$. From the naturality of the LLC for tori 
and the canonical isomorphism $N_G (S) / S \cong N_{G^\vee}(S^\vee) / S^\vee$,
we get a group isomorphism $W_{\mf s} \cong W_{\mf s^\vee}$. We note that 
\begin{equation}\label{eq:6.44}
X_\nr (S) \cong S^\vee \text{ acts simply transitively on } 
T_{\mf s} \text{ and on } T_{\mf s^\vee}.
\end{equation}
Let $\Phi (G,S)$ be the collection of $\phi \in \Phi (G)$ for which $\phi (\mb W_F) \subset
S^\vee \times \mb W_F$. (Recall $\phi$ is only defined up to $G^\vee$-conjugation; we mean this 
requirement should hold for a representative of $\phi$.) Let $\Phi_e (G,S)$ be the subset of
$\Phi_e (G)$ with L-parameter in $\Phi (G,S)$. As $\mc G$ is $F$-split, $G$-relevance of 
$(\phi,\rho)$ means that $\rho \in \Irr (\mc S_\phi)$ is trivial on 
the image $\mc Z_\phi$ of $Z({G^\vee}_\Sc)$. In other words, $\rho$ factors through 
$\mc S_\phi / \mc Z_\phi \cong \pi_0 (Z_{{G^\vee}_\ad} (\mr{im} \phi))$. We express the
LLC for principal series representations as
\begin{equation}\label{eq:6.1} 
\begin{array}{ccc}
\Phi_e (G,S) & \longleftrightarrow & \Irr (G,S) \\
(\phi,\rho) & \mapsto & \pi (\phi,\rho) 
\end{array}.
\end{equation}
Roche \cite{Roc} constructs a $\mf s$-type $(J_\chi ,\tau_\chi)$ in terms of a pinning 
$(\mc G,\mc B, \mc S, (x_\alpha )_{\alpha \in \Delta})$ and the character $\chi$ of $S$. 
Let $e_{\tau_\chi} \in \mc H (J_\chi)$ be the central idempotent associated to the
character $\tau_\chi$ of $J_\chi$. The Hecke algebra of this type is 
\[
\mc H (G,J_\chi,\tau_\chi) := e_{\tau_\chi} \mc H (G) e_{\tau_\chi}. 
\]
By virtue of types, there is an equivalence of categories
\begin{equation}\label{eq:6.2}
\begin{array}{ccc}
\Rep (G)^{\mf s} & \longleftrightarrow & \mr{Mod}(\mc H (G, J_\chi, \tau_\chi)) \\
V & \mapsto & e_{\tau_\chi} V \\
\mc H (G) e_{\tau_\chi} \underset{\mc H (G, J_\chi, \tau_\chi)}{\otimes} M & \lmapsto & M
\end{array}. 
\end{equation}
Let us consider the Hecke algebras associated to $\mf s$ and $\mf s^\vee$.
It was shown in \cite[\S 8]{Roc} that 
\begin{equation}\label{eq:6.3}
\mc H (G,J_\chi,\tau_\chi) \cong \mc H (Z_{G^\vee}(\hat \chi |_{\mb I_F} )^\circ) \rtimes 
\pi_0 (Z_{G^\vee}(\hat \chi |_{\mb I_F})) .
\end{equation}
In the terminology of \cite[\S 2]{AMS3}, the right hand side becomes
\begin{multline}\label{eq:6.4}
\mc H (Z_{G^\vee}(\hat \chi |_{\mb I_F} )^\circ, T^\vee, 1) / (\mb z - q_F) \rtimes 
\mf R_{\mf s^\vee} = \\ \mc H (Z_{G^\vee}(\hat \chi |_{\mb I_F} ), T^\vee, 1) / (\mb z - q_F) 
= \mc H (\mf s^\vee, \mb z) / (\mb z - q_F) .
\end{multline}
(Here $\mb z$ is derived from $\vec{\mb z}$ in Section \ref{sec:1} by setting all 
entries of $\vec{\mb z}$ equal.) In particular \eqref{eq:6.3} and \eqref{eq:6.4} show
that the R-group can be expressed as
\[
\mf R_{\mf s} \cong \mf R_{\mf s^\vee} \cong \pi_0 (Z_{G^\vee}(\hat \chi |_{\mb W_F})) .
\]
This R-group can be nontrivial, but its 2-cocycle $\kappa_{\mf s^\vee}$ \eqref{eq:2.4}
is trivial. (That follows also from \cite[Theorem 4.4]{ABPSprin}.) We record that 
\eqref{eq:6.3} and \eqref{eq:6.4} provide an algebra epimorphism
\begin{equation}\label{eq:6.5}
\mr{ev}_{\mb z = q_F} : \mc H (\mf s^\vee, \mb z) \to \mc H (G,J_\chi,\tau_\chi) .
\end{equation}
\begin{lem}\label{lem:6.2}
The torus $T_{\mf s^\vee}$ has a canonical basepoint $(\phi_1,\rho_1 = 1)$.
It is fixed by $W_{\mf s^\vee}$ and $\phi_1 (\Fr_F) = \Fr_F$.
\end{lem}
\begin{proof}
The algebra \eqref{eq:6.3} comes with a canonical basepoint $t_1$ of $T_{\mf s}$, fixed by
$W_{\mf s}$. In terms of \eqref{eq:2.51}, it is the unique point of $T_{\mf s}$ such
that $t_1 (\lambda (\varpi_F)) = 1$ for all $\lambda \in X_* (S)$. Via \eqref{eq:6.5},
$t_1$ gives rise to a canonical basepoint $(\phi_1,\rho_1)$ of $T_{\mf s^\vee}$.
Here $\mc S_{\phi_1} = \{1\}$, so $\rho_1$ is trivial and we need not mention it. 
By the $W_{\mf s}$-equivariance of the isomorphism $T_{\mf s} \cong T_{\mf s^\vee}$ 
(an instance of the LLC for tori), $W_{\mf s^\vee}$ fixes $\phi_1$.

Artin reciprocity translates \eqref{eq:2.51} (for the split torus $S$) to 
\begin{equation}\label{eq:6.42}
\begin{array}{ccc}
\mr{Hom}(S,\C^\times) & \cong & 
\mr{Hom}(S_{\cpt},\C^\times) \times \Hom (X_* (S), \C^\times) \\
\text{\rotatebox[origin=c]{270}{$\cong$}} & & \text{\rotatebox[origin=c]{270}{$\cong$}} \\ 
\mr{Hom}(F^\times, S^\vee) & \cong &
\mr{Hom}(\mf o_F^\times, S^\vee) \times \mr{Hom}( \varpi_F^\Z, S^\vee) \\
\text{\rotatebox[origin=c]{270}{$\cong$}} & & \text{\rotatebox[origin=c]{270}{$\cong$}} \\ 
\hspace{-3mm} \mr{Hom}(\mb W_F / [\mb W_F,\mb W_F], S^\vee) \!\! & \cong & \!\!
\mr{Hom}(\mb I_F / [\mb W_F,\mb W_F], S^\vee) \times \mr{Hom}(\mb W_F / \mb I_F, S^\vee) \hspace{-4mm}
\end{array}
\end{equation}
From $t_1 \in \mr{Hom}(S_{\cpt},\C^\times)$ we get $\phi_1 \in \mr{Hom}(\mb I_F / 
[\mb W_F,\mb W_F], S^\vee)$, which by definition of the diagram means that 
$\phi_1 (\Fr_F) = 1$ in $S^\vee$ and $\phi_1 (\Fr_F) = \Fr_F$ in ${}^L S$.
\end{proof}

The construction of \eqref{eq:6.1} in \cite{ABPSprin} can be divided in steps involving
\eqref{eq:6.2}, \eqref{eq:6.5} and Theorem \ref{thm:1.1}:
\begin{equation}\label{eq:6.6}
\begin{array}{ccccccc}
\Phi_e (G)^{\mf s^\vee} & \leftrightarrow & \Irr \big( \mc H (\mf s^\vee, \mb z) / 
(\mb z - q_F) \big) & \leftrightarrow & \Irr \big( \mc H (G,J_\chi,\tau_\chi) \big) &
\leftrightarrow & \Irr (G)^{\mf s}\!\! \\
(\phi,\rho) & \mapsto & \overline{M}(\phi,\rho) & \leftrightarrow &
e_{\tau_\chi} \pi (\phi,\rho) & \lmapsto &
\pi (\phi,\rho)\!\!
\end{array}
\end{equation}
Let $f : \tilde{\mc G} \to \mc G$ be a homomorphismp of $F$-split connected reductive
groups, such that the kernel and the cokernel of $f$ are commutative. Functoriality for 
\eqref{eq:6.6} with respect to $f$ was already investigated in \cite[\S 17]{ABPSprin}.
Our desired result was proven in \cite[Proposition 17.7]{ABPSprin}. However, it uses
hidden assumptions on $f$. Namely, in \cite[Lemma 17.1]{ABPSprin} it is assumed that 
$f_\der : \tilde{\mc G}_\der \to \mc G_\der$ is a central isogeny. That is no problem
for us, because we checked in Lemma \ref{lem:5.5} that it follows from Condition
\ref{cond:1}. Furthermore, in \cite[p .57]{ABPSprin} it is claimed that the type 
$(J_{\tilde \chi}, \tau_{\tilde \chi}) = (J_{\chi \circ f}, \tau_{\chi \circ f})$ for 
$\tilde{\mf s} = [f^{-1}(S), \chi \circ f]_{\tilde G}$ equals
$(f^{-1} (J_\chi),\tau_\chi \circ f)$. This can only be guaranteed if
\begin{equation}\label{eq:6.7}
\tilde{\mc G}(\mf o_F) = f^{-1}(\mc G (\mf o_F)) . 
\end{equation}
The condition \eqref{eq:6.7} is fulfilled when $f$ is the quotient map for a
central $F$-subgroup of $\tilde{\mc G}$ and when $f$ is the inclusion of
$\tilde{\mc G}$ in $\tilde{\mc G} \times \mc T$ for some ($F$-split) torus $\mc T$.
But \eqref{eq:6.7} need not hold the conjugation action of $G_\ad$ on $G$. The classes
in $G_\ad / G \cong T_{AD}/T$ which do not come from an element of $T_{AD,\cpt}$ via
\eqref{eq:2.53} cannot be represented by an element of $G_\ad$ which stabilizes
$\mc G (\mf o_F)$.

To deal with this, we set things up more precisely. We fix a pinning\\
$(\tilde{\mc G}, \tilde{\mc B}, \tilde{\mc T}, (x_{\tilde \alpha} )_{\tilde \alpha 
\in \tilde \Delta})$ for $\tilde G$. By \cite[Lemma 17.2]{ABPSprin} the pullback 
functor $f^* : \Rep (G) \to \Rep (\tilde{G})$ sends $\Rep (G)^{\mf s}$ to
$\Rep (\tilde G )^{\tilde{\mf s}}$, which is a Bernstein component for the principal
series of $\tilde{G}$.

Like in \eqref{eq:5.1}, we decompose $f$ as
\[
\tilde{\mc G} \xrightarrow{f_1} \tilde{\mc G} \times \mc T \xrightarrow{f_2} 
\tilde{\mc G}_2 := (\tilde{\mc G} \times \mc T) / \ker (f \times \mr{id}_{\mc T}) 
\xrightarrow{f_3} \mc G.
\]
The pinning of $\tilde{G}$ determines a unique pinning of $\tilde{G}_2$.
By Proposition \ref{prop:3.4} and Theorem \ref{thm:3.2}.a there exists a unique 
$F$-isomorphism $\tilde{\mc G}_2 \to \mc G$ which respects the pinning and induces
\[
\mc R (f_3) : \mc R (\tilde{\mc G}_2, \tilde{\mc S}_2) \to \mc R (\mc G, \mc S) .
\]
As in \eqref{eq:3.9}, there exists a unique $g_3 \in G_\ad$ such that 
$f_3 = \mr{Ad}(g_3) \circ \eta_{\mc R (f_3),\mc G}$. In this way the pinnings 
serve to satisfy Condition \ref{cond:3}.

For $\phi \in \Phi (G,S)$, the L-parameter $\tilde \phi = {}^L f \circ \phi$
belongs to $\Phi (\tilde{G}, \tilde{S})$. Finally we can prove that Conjecture 
\ref{conj:5.7} holds throughout the principal series.

\begin{thm}\label{thm:6.3}
Let $f : \tilde{\mc G} \to \mc G$ be a homomorphism of $F$-split connected reductive
groups, satisfying Condition \ref{cond:1}. Assume that the residual characteristic $p$ does 
not belong to the bad cases from \cite[p. 378--379]{Roc}. 
For any $(\phi,\rho) \in \Phi_e (G,S)$:
\[
f^* (\pi (\phi,\rho)) = 
\bigoplus\nolimits_{\tilde \rho \in \Irr (\mc S_{\tilde \phi})} \mr{Hom}_{\mc S_\phi} 
\big( \rho, {}^S f^* (\tilde{\rho}) \big) \otimes \pi (\tilde \phi, \tilde \rho) .
\]
\end{thm}
\begin{proof}
As discussed above, for the maps $f_1$ and $f_2$ in \eqref{eq:5.1} this was shown in
\cite[\S 17]{ABPSprin}. For $f_2$ the setup in \cite{ABPSprin} differs slightly for ours.
In effect it uses the right hand column in the diagram \eqref{eq:4.20},  
not the left hand column like we do in the proof of Theorem \ref{thm:4.5}. But in the
proof of Lemma \ref{lem:4.4} we checked that the various paths in the diagram \eqref{eq:4.20}
boil down to the same pullback operation, so \cite[\S 17]{ABPSprin} produces the same
representations as we consider.

We note though that in \cite[Proposition 17.7]{ABPSprin} the different notation
\begin{equation}\label{eq:6.8}
f^* (\pi (\phi,\rho)) = \pi \big(\tilde \phi, 
\mr{ind}_{\mc S_\phi / \mc Z_\phi}^{\mc S_{\tilde \phi} / \mc Z_{\tilde \phi}} \rho \big)
\end{equation}
is used. Let us reconcile it with the statement of the theorem. By the commutativity of 
the (co)kernel of $f$ there are canonical identifications 
\[
{G^\vee}_\Sc = \tilde{G}^\vee_{\ \Sc} ,\; Z({G^\vee}_\Sc) = Z(\tilde{G}^\vee_{\ \Sc} ) ,\;
{}^L f (\mc Z_\phi) = \mc Z_{\tilde \phi} .
\]
In particular $\mr{ind}_{\mc S_\phi / \mc Z_\phi}^{\mc S_{\tilde \phi} / \mc Z_{\tilde \phi}} 
\rho$ is identified with $\mr{ind}_{\mc S_\phi}^{\mc S_{\tilde \phi}} \rho$.
Hence \eqref{eq:6.8} equals
\[
\pi \big(\tilde \phi, \mr{ind}_{\mc S_\phi}^{\mc S_{\tilde \phi}} \rho \big) =
\bigoplus\nolimits_{\tilde \rho \in \Irr (\mc S_{\tilde \phi})} \mr{Hom}_{\mc S_{\tilde \phi}}
\big( \mr{ind}_{\mc S_\phi}^{\mc S_{\tilde \phi}} \rho , \tilde \rho \big) 
\otimes \pi (\tilde \phi, \tilde \rho) .
\]
With Frobenius reciprocity we rephrase this and \eqref{eq:6.8}:
\begin{equation}\label{eq:6.9}
\bigoplus\nolimits_{\tilde \rho \in \Irr (\mc S_{\tilde \phi})} \mr{Hom}_{\mc S_\phi} 
\big( \rho, {}^L f^* (\tilde{\rho}) \big) \otimes \pi (\tilde \phi, \tilde \rho)
\end{equation}
As ${}^S f = {}^L f : \C [\mc S_\phi] \to \C [\mc S_{\tilde \phi}]$ for $f = f_1$ and 
$f = f_2$, \eqref{eq:6.9} agrees with the formula in the theorem.

Having dealt with $f_1$ and $f_2$, we focus on $f_3$. To ease the notation, we will assume
in the remainder of the proof that $f = f_3$, i.e. that $f : \tilde{\mc G} \to \mc G$
is an isomorphism.

For a root $\alpha \in \Phi (\mc G,\mc S)$, let $\alpha^\vee \in
\Phi (\mc G, \mc S)^\vee$ be the corresponding coroot and put
$\tilde \alpha = \alpha \circ f \in \Phi (\tilde{\mc G}, \tilde{\mc S})$. Then $f$ sends
the root subgroup $\mc U_{\tilde \alpha} \subset \tilde{\mc G}$ to $\mc U_\alpha$ and
\[
\chi \circ \alpha^\vee = \chi \circ f \circ f^{-1} \circ \alpha^\vee = 
\tilde \chi \circ \tilde \alpha^\vee .
\]
Suppose now that $f$ maps the pinning of $\tilde G$ to the pinning of $G$. This implies
that $f(U_{\tilde \alpha,r}) = U_{\alpha,r}$ in the notation from \cite[p. 366]{Roc} 
(which comes from \cite[\S 6]{BrTi1}). As $\mc R (f)$ preserves positivity of roots,
one sees from \cite[\S 2]{Roc} that $f(J_{\tilde \chi}) = J_\chi$ and 
\[
(J_\chi, \tau_\chi) = (f^{-1}(J_{\tilde \chi}), \tau_{\tilde \chi} \circ f) . 
\]
Knowing this, the proof in \cite[\S 17]{ABPSprin} works again.

By Proposition \ref{prop:3.4}.a, every isomorphism $f : \tilde{\mc G} \to \mc G$ is
the composition of an isomorphism preserving the pinnings and conjugation by an element
of $G_\ad$. So it remains to consider Ad$(g) \in \mr{Aut}(G)$ for $g \in G_\ad$.
As explained in \eqref{eq:2.33}, the actions of $G_\ad$ on $\Irr (G)$ and on $\Phi_e (G)$
factor through
\[
G_\ad / G \cong S_{AD} / S . 
\]
In particular we may and will assume that $g \in S_{AD} = (\mc S / Z(\mc G))(F)$, so 
that Ad$(g)$ stabilizes $\chi$ and $\mf s = [S,\chi]_G$. We write $g = g_c g_x$ as
in \eqref{eq:2.53}, with $g_c \in S_{AD,\cpt}$ and $g_x \in X_* (S)$.
By Lemma \ref{lem:6.2}, \eqref{eq:6.42} and Lemma \ref{lem:2.7}, Condition \ref{cond:2} holds
and the automorphism Ad$(x_g)$ from Proposition \ref{prop:2.8} is available.

By Lemma \ref{lem:6.2} we can choose a lift $\phi_\Sc : \mb W_F \to S^\vee \times \mb W_F$
of $\phi_1$ with $\phi_\Sc (\Fr_F) = \Fr_F$. Then \eqref{eq:2.16} shows that $c_h (\Fr_F) = 1$
for every $h \in Z_{G^\vee}(\phi (\mb W_F))$. Since $X_* (S_{AD})$ is embedded in $S_{AD}$ 
via evaluation at $\varpi_F$ and Artin reciprocity sends $\Fr_F$ to $\varpi_F$,
\[
\tau_{\phi_1}(g_x) (h) = \langle g_x, c_h \rangle = \langle x_g, c_h (\Fr_F) \rangle = 1 .
\]
As noted after \eqref{eq:2.37}, we now have $\alpha_g = \alpha_{g_c}$ and Ad$(x_g) =
\mr{Ad}(x_{g_x})$. This enables us to consider $g_c$ and $g_x$ separately.

The compactness of $g_c$ implies that $\mr{Ad}(g_c)(J_\chi) = J_\chi$. This allows us to
define an algebra automorphism
\begin{equation}\label{eq:6.10}
\begin{array}{cccc}
\mc H (\mr{Ad}(g_c)) : & \mc H (G,J_\chi,\tau_\chi) & \to & \mc H (G,J_\chi,\tau_\chi) \\
& f & \mapsto & f \circ \mr{Ad}(g_c)^{-1}
\end{array}.
\end{equation}
We note that for any $V \in \Rep (G)$:
\begin{equation}\label{eq:6.12}
e_{\tau_\chi} (\mr{Ad}(g_c)^* V) = \mc H (\mr{Ad}(g_c))^* (e_{\tau_\chi} V) \in
\mr{Mod} \big( \mc H (G,J_\chi,\tau_\chi) \big) .
\end{equation}
Let us compare $\alpha_{g_c}$ with $\mc H (\mr{Ad}(g_c))$ via \eqref{eq:6.5}.
Both are the identity on $\mc O (T_{\mf s})$. For $w \in W_{\mf s}$, represented by
$\dot w \in N_G (S)$:
\[
\mr{Ad}(g_c)(N_w) = \mr{Ad}(g_c)( e_{\tau_\chi} N_w (\dot w) \dot{w} e_{\tau_\chi})  
= e_{\tau_\chi} N_w (\dot w) g_c \dot{w} g_c^{-1} e_{\tau_\chi} ,
\]
where $g_c \dot{w} g_c^{-1} \in G$ is regarded as a multiplier of $\mc H (G)$.
As $g_c \in S_{AD,\cpt}$,  $g_c \dot w g_c^{-1} S_\cpt = \dot w S_\cpt$, and
$\mr{Ad}(g_c) (N_w) = \lambda N_w$ for some $\lambda \in \C^\times$. For $w \in
W_{\mf s}^\circ \cong W_{\mf s^\vee}^\circ$, \cite[Proposition 14.1.4]{ABPSprin} says that 
$\lambda = 1$, which agrees with $\alpha_{g_c}(N_w) = N_w$. For $w = r \in \mf R_{\mf s}
\cong \mf R_{\mf s^\vee}$ we obtain
\begin{equation}\label{eq:6.11}
\begin{aligned}
\mr{Ad}(g_c)(N_r) & = e_{\tau_\chi} N_r (\dot r) 
\dot r \dot{r}^{-1} g_c \dot{r} g_c^{-1} e_{\tau_\chi} \\
& = e_{\tau_\chi} N_r (\dot r) \dot r e_{\tau_\chi}(\dot{r}^{-1} 
g_c \dot{r} g_c^{-1})^{-1} e_{\tau_\chi} \\
& = e_{\tau_\chi} N_r (\dot r) \dot{r} e_{\tau_\chi} \chi (\dot{r}^{-1} g_c \dot r g_c^{-1}) \\
& = N_r \langle \dot{r}^{-1} g_c \dot r g_c^{-1}, \hat \chi \rangle .
\end{aligned}
\end{equation}
Let $\hat \chi_\Sc : \mb W_F \to {S^\vee}_\Sc \times \mb W_F$ be a lift of $\hat \chi$.
Let $r^\vee \in N_{G^\vee}(S^\vee)$ be a representative for $r \in \mf R_{\mf s^\vee}$.
By the $W(\mc G,\mc S)$-equivariance of the LLC for $S$ and by \eqref{eq:2.16}, the right 
hand side of \eqref{eq:6.11} can be written as
\[
N_r \langle g_c, r^\vee \hat \chi_\Sc (r^\vee)^{-1} {\hat \chi}_\Sc^{-1} \rangle =
N_r \langle g_c, c_{r^\vee} \rangle = N_r \tau_{\hat \chi}(g_c) (r^\vee) .
\]
Once again this agrees with \eqref{eq:2.54}. With \eqref{eq:6.5} we conclude that
\begin{equation}\label{eq:6.40}
\mr{ev}_{\mb z = q_F} \circ \alpha_{g_c} = \mc H (\mr{Ad}(g_c)) \circ \mr{ev}_{\mb z = q_F} . 
\end{equation}
Now \eqref{eq:6.12}, Theorem \ref{thm:2.9} and \eqref{eq:6.6} imply that
\begin{equation}\label{eq:6.17}
\mr{Ad}(g_c)^* \pi (\phi,\rho) = \pi (\phi, \rho \otimes \tau_{\phi}(g_c)^{-1}). 
\end{equation}
Finally we look at Ad$(g_x)$, which does not stabilize $\mc G (\mf o_F)$ and $J_\chi$
(unless $g_x = 1$). The automorphism Ad$(g_x)$ of $\mc H (G)$ restricts to 
an algebra isomorphism
\begin{equation}\label{eq:6.15}
\begin{array}{ccc}
\mc H (G,J_{\tilde \chi}, \tau_{\tilde \chi}) & \to & 
\mc H (G,J_{\chi},\tau_{\chi}) \\
N_w & \mapsto & N_{g_x w g_x^{-1}} 
\end{array}.
\end{equation}
To compare the source and target of \eqref{eq:6.15}, we use the naive algebra isomorphism
\begin{equation}\label{eq:6.13}
\begin{array}{ccc}
\mc H (G,J_\chi, \tau_\chi) & \to & \mc H (G,J_{\tilde \chi},\tau_{\tilde \chi}) \\
N_w & \mapsto & N_w 
\end{array}.
\end{equation}
The Morita equivalence
\begin{equation}\label{eq:6.14}
\begin{array}{ccccc}
\hspace{-3mm} \mr{Mod}\big( \mc H (G,J_{\tilde \chi},\tau_{\tilde \chi}) \big) \!\! & 
\to & \Rep (G)^{\mf s} & \to & \mr{Mod}\big( \mc H (G,J_{\chi},\tau_{\chi}) \big) \\
M & \mapsto & \! \mc H (G) e_{\tau_{\tilde \chi}} 
\underset{\mc H (G, J_{\tilde \chi}, \tau_{\tilde \chi})}{\otimes} M \! & \mapsto & 
\! e_{\tau_\chi} \mc H (G) e_{\tau_{\tilde \chi}} 
\underset{\mc H (G, J_{\tilde \chi}, \tau_{\tilde \chi})}{\otimes} M \hspace{-5mm}
\end{array}  
\end{equation}
is just given by composing representations with \eqref{eq:6.13}.
The composition of \eqref{eq:6.13} and \eqref{eq:6.15} is the automorphism
\begin{equation}\label{eq:6.16}
\begin{array}{cccc}
\beta_{g_x} : & \mc H (G,J_{\chi}, \tau_{\chi}) & \to & \mc H (G,J_{\chi},\tau_{\chi}) \\
 & N_w & \mapsto & N_{g_x w g_x^{-1}} 
\end{array}.
\end{equation}
We can regard $g_x w g_x^{-1}$ as an element of $X^* (T_{\mf s}) \rtimes W_{\mf s} \cong
X^* (T_{\mf s^\vee}) \rtimes W_{\mf s^\vee}$. Then it can be written as
\[
g_x w g_x^{-1} = x_g w x_g^{-1} = w (w^{_1} x_g w x_g^{-1}) = w (w^{-1}(x_g) - x_g) 
\]
where $w^{-1}(x_g) - x_g \in X^* (T_{\mf s^\vee})$. Proposition \ref{prop:2.8}
says that $\beta_g$ is the specialization of Ad$(x_g)$ at $\mb z = q_F$.
From this and \eqref{eq:6.14} we deduce that, for any $V \in \Rep (G)^{\mf s}$:
\[
\mr{Ad}(g_x)^* V \cong \mc H (G) e_{\tau_\chi} \underset{\mc H (G, J_{\chi}, 
\tau_\chi)}{\otimes} \beta_{g_x}^* (e_{\tau_\chi} V) \cong 
\mc H (G) e_{\tau_\chi} \underset{\mc H (G, J_{\chi}, 
\tau_\chi)}{\otimes} \mr{Ad}(x_g)^* (e_{\tau_\chi} V) .
\]
Together with Theorem \ref{thm:2.9} and \eqref{eq:6.6} this entails
\begin{equation}\label{eq:6.18}
\mr{Ad}(g_x)^* \pi (\phi,\rho) = \pi (\phi, \rho \otimes \tau_{\phi}(g_x)^{-1}). 
\end{equation}
Combining \eqref{eq:6.17} and \eqref{eq:6.18}, we get the same formula for all
$g \in S_{AD}$, and hence for all $g \in G_\ad$. For $\tilde \rho \in 
\Irr (\mc S_{\tilde \phi})$ \eqref{eq:5.2} gives ${}^S \mr{Ad}(g)^* (\tilde \rho) = 
\tau_{\phi}(g) \otimes \tilde \rho$. Thus \eqref{eq:6.18} is equivalent
to the desired expression
\[
\bigoplus_{\tilde \rho \in \Irr (\mc S_{\tilde \phi})} \mr{Hom}_{\mc S_\phi} 
\big( \rho, {}^S \mr{Ad}(g)^* (\tilde{\rho}) \big) \otimes \pi (\phi, \tilde \rho) 
= \pi (\phi, \rho \otimes \tau_{\phi}(g)^{-1}) = \mr{Ad}(g)^* \pi (\phi,\rho) .
\qedhere
\]
\end{proof}

\section{Unipotent representations} 
\label{sec:unip}

In this paragraph $\mc G$ is a connected reductive $F$-group, which splits over an
unramified extension of $F$. We start with recalling some properties of unipotent
representations of $G = \mc G (F)$.

Let $\mc B (\mc G,F)$ be the enlarged Bruhat--Tits building of $\mc G (F)$. To every
facet $\mf f$ of $\mc B (\mc G,F)$ a parahoric subgroup $P_{\mf f}$ is attached.
Let $U_{\mf f}$ be the pro-unipotent radical of $P_{\mf f}$. The quotient $P_{\mf f} /
U_{\mf f}$ has the structure of a finite reductive group 
$\overline{\mc G^\circ_{\mf f}}(k_F)$ over $k_F = \mf o_F / \varpi_F \mf o_F$. 
An irreducible representation $\sigma$ is called cuspidal unipotent if it arises 
by inflation from a cuspidal unipotent representation of $P_{\mf f} / U_{\mf f}$.

An irreducible $G$-representation $\pi$ is called unipotent if, for some facet $\mf f$,
$\pi |_{P_{\mf f}}$ contains a cuspidal unipotent representation of $P_{\mf f}$. We denote
the set of such representations by $\Irr_\unip (G)$. 
The category of unipotent $G$-representations is defined as the product of all Bernstein
blocks $\Rep (G)_{\mf s}$ for which $\Irr (G)_{\mf s}$ consists of irreducible unipotent
$G$-representations.

Let $\hat P_{\mf f}$ be the pointwise stabilizer of $\mf f$ in $G$, a compact group 
containing $P_{\mf f}$ as open subgroup. It turns out \cite[\S 1.16]{LusUni1} that any 
cuspidal unipotent representation $(\sigma, V_\sigma)$ of $P_{\mf f}$ is can be extended to a 
representation of $\hat P_{\mf f}$ on $V_\sigma$, say $\hat \sigma$. It is known from 
\cite[Theorem 4.7]{Mor2} that $(\hat P_{\mf f}, \hat \sigma)$ is a type for a single 
Bernstein block, which we call $\Rep (G)_{(\hat P_{\mf f}, \hat \sigma)}$. By construction 
the irreducible representations therein are 
\[
\Irr (G)_{(\hat P_{\mf f}, \hat \sigma)} = \{ \pi \in \Irr (G) : 
\pi |_{\hat P_{\mf f}} \text{ contains } \hat \sigma \} .
\]
It follows from \cite[1.6.b]{LusUni1} that $\Rep (G)_{(\hat P_{\mf f}, \hat \sigma)}$ and
$\Rep (G)_{(\hat P_{\mf f'}, \hat \sigma)}$ are disjoint if $\mf f$ and $\mf f'$ lie in
different $G$-orbits.

Let $e_{\hat \sigma} \in \mc H (\hat P_{\mf f})$ be the central idempotent associated to
$\hat \sigma$. Then $e_{\hat \sigma} \mc H (G) e_{\hat \sigma}$ is a subalgebra of $\mc H (G)$,
Morita equivalent to the two-sided ideal of $\mc H (G)$ corresponding to 
$\Rep (G)_{(\hat P_{\mf f}, \hat \sigma)}$. Let $(\hat \sigma^*, V_{\sigma}^*)$ be the
contragredient of $\hat \sigma$. Recall from \cite[\S 2]{BuKu} that 
\begin{multline*}
\mc H (G,\hat P_{\mf f}, \hat \sigma) = \\ 
\big\{ f \in \mc H (G) \otimes_\C \mr{End}_\C (V_{\sigma}^*) : f (p_1 g p_2) = \hat \sigma^* 
(p_1) f(g) \hat \sigma^* (p_2) \;  \forall g \in G, p_1,p_2 \in \hat P_{\mf f} \big\} .
\end{multline*}
By \cite[(2.12)]{BuKu} there is a natural isomorphism
\[
\mc H (G,\hat P_{\mf f}, \hat \sigma) \otimes \mr{End}_\C (V_{\sigma}) \cong
e_{\hat \sigma} \mc H (G) e_{\hat \sigma} .
\]
As $(\hat P_{\mf f}, \hat \sigma)$ is a type, \cite[Theorem 4.3]{BuKu} gives equivalences of
categories
\begin{equation}\label{eq:6.21}
\begin{array}{ccccc}
\Rep (G)_{(\hat P_{\mf f}, \hat \sigma)} & \leftrightarrow & \mr{Mod} \big( e_{\hat \sigma} \mc H (G) 
e_{\hat \sigma} \big) &  \leftrightarrow & \mr{Mod} \big( \mc H (G,\hat P_{\mf f}, \hat \sigma) \big) 
\hspace{-5mm} \\
V & \mapsto & e_{\hat \sigma} V & \mapsto & \mr{Hom}_{\hat P_{\mf f}}(\hat \sigma,V) \\
\!\! \mc H (G) e_{\hat \sigma} \underset{e_{\hat \sigma} \mc H (G) e_{\hat \sigma}}{\otimes}
(M \otimes V_\sigma) & \lmapsto & M \otimes V_\sigma & \lmapsto & M
\end{array}
\end{equation}
Suppose for the moment that $\mf f$ is a minimal facet of $\mc B (\mc G,F)$, or equivalently that
$P_{\mf f}$ is a maximal parahoric subgroup of $G$. Then $\hat \sigma$ can be further extended
to a representation of $\mr{Stab}_G (\mf f) = N_G (P_{\mf f})$, say $\sigma^N$. For every 
character $\psi \in \Irr (N_G (P_{\mf f}) / \hat P_{\mf f})$, $\sigma^N \otimes \psi$ is another 
such extension, and they are all of this form. It follows from \cite[\S 2]{LusUni1} that 
\[
\mr{ind}_{N_G (P_{\mf f})}^G (\sigma^N \otimes \psi) \in \Irr_{\cusp,\unip}(G) ,
\]
where the right hand side denotes the set of irreducible supercuspidal unipotent 
$G$-representations. Moreover the map
\begin{equation}\label{eq:6.19}
\Irr (N_G (P_{\mf f}) / \hat P_{\mf f}) \to \Irr (G)_{(\hat P_{\mf f}, \hat \sigma)} :
\psi \mapsto \mr{ind}_{N_G (P_{\mf f})}^G (\sigma^N \otimes \psi)
\end{equation}
is a bijection. By allowing $\psi$ to be a character of the abelian group 
$N_G (P_{\mf f}) / P_{\mf f}$, \eqref{eq:6.19} extends to a bijection
\[
\Irr (N_G (P_{\mf f}) / P_{\mf f})) \to \{ \pi \in \Irr (G) : \pi |_{P_{\mf f}}
\text{ contains } \sigma \} ,
\]
see \cite[(1.18)]{FOS}.

For a general facet $\mf f$, let $\mc L$ be a Levi $F$-subgroup of $\mc G$ such that
$P_{L,\mf f} = P_{\mf f} \cap L$ is a maximal parahoric subgroup of $L = \mc L (F)$.
Then $P_{L,\mf f} / U_{L,\mf f} \cong P_{\mf f} / U_{\mf f}$ and any cuspidal unipotent
representation $\sigma$ of $P_{L,\mf f}$ has a natural extension to $P_{\mf f}$. 
Similarly every extension $\hat \sigma \in \Irr (\hat P_{L,\mf f})$ of $\sigma$ extends
naturally to an irreducible representation of $\hat P_{\mf f}$ on $V_\sigma$, which we 
continue to denote by $\hat \sigma$. According to \cite[Corollary 3.10]{Mor2}
\begin{equation}\label{eq:6.20}
(\hat P_{\mf f}, \hat \sigma) \text{ is a cover of } (\hat P_{L,\mf f}, \hat \sigma) .
\end{equation}
In particular $\Rep (G)_{(\hat P_{\mf f}, \hat \sigma)} = \Rep (G)_{\mf s}$, where
$\mf s = [L, \mr{ind}_{N_L (P_{L,\mf f})}^L (\sigma^N)]_G$.  
By \cite[Theorem 3.3.b]{SolLLCunip}
\begin{equation}\label{eq:6.22}
\mc H (L,\hat P_{L,\mf f},\hat \sigma) \cong \C [X_{\mf f}]
\end{equation}
for a lattice $X_{\mf f}$ as in \cite[Proposition 3.1]{SolLLCunip}. Furthermore 
$\C [X_{\mf f}]$ embeds canonically in $\mc H (G,\hat P_{\mf f}, \hat \sigma)$ as the 
commutative subalgebra from the Bernstein presentation. By \cite[Lemma 3.4.b]{SolLLCunip}
the group $W_{\mf s}$ is naturally isomorphic to the finite Weyl group of the root system
underlying $\mc H (G,\hat P_{\mf f}, \hat \sigma)$. As in \cite[Proposition 3.3.c and 
Lemma 3.4]{SolLLCunip}, we pick a $W_{\mf s}$-equivariant homeomorphism
\begin{equation}\label{eq:6.26}
\Irr (L)_{(\hat P_{L,\mf f}, \hat \sigma)} \to \Irr (X_{\mf f}) . 
\end{equation}
A local Langlands correspondence for supercuspidal unipotent representations of $G$
(and of its Levi subgroups) was exhibited in \cite{FOS}. It provides a bijection between
inertial equivalence classes for $\Irr_\unip (G)$ and inertial equivalence classes for
$\Phi_{\nr,e}(G)$. If we write it as $\mf s \leftrightarrow \mf s^\vee$, then 
$W_{\mf s} \cong W_{\mf s^\vee}$ \cite[Proposition 4.2]{SolLLCunip}. Suppose that 
$\Rep (G)_{\mf s} = \Rep (G)_{(\hat P_{\mf f},\hat \sigma)}$. By \cite[Theorem 4.4]{SolLLCunip} 
the LLC on the cuspidal level canonically determines an algebra isomorphism
\begin{equation}\label{eq:6.23}
\mc H (\mf s^\vee, \vec{\mb z}) / (\vec{\mb z} - \vec{z}) \to \mc H (G,\hat P_{\mf f},\hat \sigma) ,
\end{equation}
for a unique $\vec z \in \R_{>1}^d$. In particular $\mf R_{\mf s^\vee} = \mf R_{\mf s} = 1$. 

Together with \eqref{eq:6.21} and Theorem \ref{thm:1.1}, \eqref{eq:6.23} provides bijections
\begin{equation}\label{eq:6.24}
\begin{array}{ccccccc}
\Irr (G)_{\mf s} & \leftrightarrow & \Irr \big( \mc H (G,\hat P_{\mf f},\hat \sigma) \big) &
\leftrightarrow & \Irr_{\vec z} \big( \mc H (\mf s^\vee, \vec{\mb z}) \big) &
\leftrightarrow & \Phi_e (G)^{\mf s^\vee} \\
\pi (\phi,\rho) & \mapsto & \mr{Hom}_{\hat P_{\mf f}}(\hat \sigma, \pi (\phi,\rho)) &
\leftrightarrow & \overline{M}(\phi,\rho,\vec z) & \lmapsto & (\phi,\rho)
\end{array}
\end{equation}
The union of the bijections \eqref{eq:6.24}, over all eligible $\mf s$ and $\mf s^\vee$,
is a local Langlands correspondence for unipotent representations:
\[
\Irr_\unip (G) \longleftrightarrow \Phi_{\nr,e}(G) .
\]
In \cite[\S 5]{SolLLCunip} it was shown to possess many desirable properties, concerning
supercuspidality, discrete series, temperedness, central characters, parabolic induction
and the action of $H^1 (\mb W_K, Z(G^\vee))$. Amongst Borel's desiderata \cite[\S 10.3]{Bor}
only the functoriality with respect to homomorphisms of reductive groups was missing in
\cite{SolLLCunip}. We fix that here.

\subsection{Central quotient maps} \

Let us consider the case where $f$ is a quotient map 
\[
q : \tilde{\mc G} \to \mc G = \tilde{\mc G} / \mc N ,
\]
such that $q_\der : \tilde{\mc G}_\der \to \mc G_\der$ is a central isogeny. 

\begin{lem}\label{lem:6.4}
For $(\phi,\rho) \in \Phi_{\nr,\cusp} (G)$:
\[
q^* \pi (\phi,\rho) = 
\bigoplus_{\tilde \rho \in \Irr (\mc S_{\tilde \phi}) :
\tilde \rho |_{\mc S_\phi} \: \mr{contains} \: \rho} \pi (\tilde \phi, \tilde \rho) =
\bigoplus_{\tilde \rho \in \Irr (\mc S_{\tilde \phi})} \mr{Hom}_{\mc S_\phi}(\rho,
\tilde \rho) \otimes \pi (\tilde \phi, \tilde \rho) .
\]
\end{lem} 
\begin{proof}
Write $\pi = \pi (\phi,\rho) \in \Irr_{\cusp,\unip}(G)$. For any weakly unramified
character $\chi \in X_\Wr (G)$, $q^* (\chi)$ is a weakly unramified character of $G$, and
\[
q^* (\pi \otimes \chi) = q^* (\pi) \otimes q^* (\chi) .
\] 
If $z$ is the image of $\chi$ in $Z(G^\vee)_\Fr$, $q^\vee (z)$ is the image of
$q^* (\chi)$ in $Z({\tilde G}^\vee)_\Fr$. By construction (cf. the proof of Theorem 2
in \cite[p. 37]{FOS}) $\pi (z \phi,\rho) = \pi (\phi,\rho) \otimes \chi$. For similar
reasons
\[
q^* (\pi \otimes \chi) = q^* (\pi (z \phi, \rho)) = q^* (\pi (\phi,\rho)) \otimes q^* (\chi)
\]
is obtained from $q^* (\pi)$ by applying $q^\vee (z)$ to the enhanced L-parameters of
the constituents of $q^* (\pi)$. For a suitable choice of $\chi$, $\pi \otimes \chi$ is
trivial on the split part $Z(\mc G)_s (F)$ of the centre of $G$ \cite[p. 37]{FOS}. 
Hence it suffices to prove the lemma in the case that $\pi \big|_{Z(\mc G)_s (F)} = 1$.

This allows us to replace $\mc G$ by $\mc G / Z(\mc G)_s$ similarly for $\tilde{\mc G}$.
In other words, we may assume that $Z(\mc G)^\circ$ and $Z(\tilde{\mc G})^\circ$ are
$F$-anisotropic. By \cite[Lemmas 15.3--15.5]{FOS} all relevant objects for $G$ can be 
identified with those for $G_\der$ (and similarly for $\tilde G$). Then we may replace
$\mc G$ by $\mc G_\der$ and $\tilde{\mc G}$ by $\tilde{\mc G}_\der$. So from now on we may 
assume that $\mc G$ and $\tilde{\mc G}$ are semisimple, and that the quotient map
$\tilde{\mc G} \to \mc G = \tilde{\mc G} / \mc N$ is a central isogeny. In particular
we can identify the Bruhat--Tits buildings $\mc B (\tilde{\mc G},F)$ and $\mc B (\mc G,F)$

With \eqref{eq:6.19} we can write $\pi = \mr{ind}_{N_G (P_{\mf f})}^G (\sigma^N)$. Then 
\[
q^{-1}(N_G (P_{\mf f})) = \mr{Stab}_{\tilde G}(\tilde{P_{\mf f}}) =
N_{\tilde G}(\tilde P_{\mf f})
\]
and $q^* ( \sigma^N ) \vphantom{\big|}$ is an extension of a cuspidal unipotent representation 
of the parahoric subgroup $\tilde{P_{\mf f}}$ to $N_{\tilde G}(\tilde P_{\mf f})$. From the
classification \eqref{eq:6.19} we see that 
\begin{equation}\label{eq:6.25}
q^* (\pi) = q^* \big( \mr{ind}_{N_G (P_{\mf f})}^G (\sigma^N) \big) =
\bigoplus_{g \in G / N_G (P_{\mf f}) q(\tilde G)} \!\!\! \mr{ind}^{\tilde G}_{\mr{Ad}(g)^{-1} 
N_{\tilde G}(\tilde P_{\mf f})} \big( q^* (\sigma^N) \circ \mr{Ad}(g) \big) . \vspace{-3mm}
\end{equation}
The sum is indexed by $\tilde G$-orbits of facets of $\mc B (\mc G,F)$ in the $G$-orbit of 
$\mf f$, and by \cite[1.6.b]{LusUni1} all summands live in different Bernstein components.

Cuspidal unipotent representations of finite reductive are essentially independent
of the isogeny type \cite[\S 3]{Lus-Che}, so $q^* (\sigma)$ is an irreducible 
$\tilde P_{\mf f}$-representation. It follows that $q^* (\hat \sigma)$ and 
$q^* (\sigma^N)$ are also irreducible representations on $V_\sigma$ (of $\hat{\widetilde P_{\mf f}}$
and $N_{\tilde G}(\tilde P_{\mf f})$ respectively). In particular every term on the right
hand side in \eqref{eq:6.25} is irreducible and supercuspidal.

When $\mc G = \tilde{\mc G}_\ad$, \cite[Lemma 13.5]{FOS} says that the set of indices
in \eqref{eq:6.25} is equivariantly in bijection with the set of extensions of $\rho$
to a $\tilde \rho \in \Irr (\mc S_{\tilde \phi})$. That proves the lemma in the case
$\mc G = \tilde{\mc G}_\ad$.

For other $\mc G$, we can consider the sequence
\[
\tilde{\mc G} \to \mc G \to \tilde{\mc G}_\ad = \mc G_\ad .
\]
Comparing \cite[Lemma 13.5]{FOS} for $\tilde{\mc G},\tilde{\mc G}_\ad$ and for
$\mc G, \mc G_\ad$, we see that it also holds for $\tilde{\mc G}, \mc G$. That
establishes the first equality of the lemma.

Furthermore \cite[(13.13)]{FOS} for $\tilde{\mc G},\tilde{\mc G}_\ad$ shows that
$\mr{Hom}_{\mc S_\phi}(\rho, \tilde \rho)$ has dimension $\leq 1$. The group
$\mc S_\phi$ for $\mc G_\ad (F)$ is contained in $\mc S_\phi$ for $\mc G (F)$, so
$\dim_\C \mr{Hom}_{\mc S_\phi}(\rho, \tilde \rho) \leq 1$ in our generality.
\end{proof}

Recall the notations from Section \ref{sec:4}. Let $(\phi_L, \rho_L) \in
\Phi_{\nr,\cusp}(L)$ and put $\tilde{\mc L} = q^{-1}(\mc L)$, a Levi $F$-subgroup 
of $\tilde{\mc G}$. When we write $\mr{ind}_{\mc S_{\phi_L}}^{\mc S_{\tilde \phi_L}}
(\rho_L) = \bigoplus_i \rho_i^{m_i}$, Lemma \ref{lem:6.4} says that $m_i = 1$ for
all $i$, and that
\[
q^* \pi (\phi_L, \rho_L) = \bigoplus\nolimits_i \pi (\tilde{\phi}, \rho_i) 
\in \Rep (\tilde L) .
\]
Fix a unipotent type $(\hat P_{L,\mf f}, \hat \sigma)$ such that 
$\pi (\phi_L,\rho_L) \in \Irr (L)_{(\hat P_{L,\mf f}, \hat \sigma)}$. For every
$\rho_i$ as above we choose a $l_i \in L$ such that, in \eqref{eq:6.25} for $L$,
\[
\pi (\tilde \phi, \rho_i) = \mr{ind}^{\tilde L}_{\mr{Ad}(l_i)^{-1} 
N_{\tilde L}(P_{\tilde L,\mf f})} \big( q^* (\sigma^N) \circ \mr{Ad}(l_i) \big) .
\]
We may take $l_1 = 1$, so that $\pi (\tilde \phi, \rho_1) = \mr{ind}^{\tilde L}_{ 
N_{\tilde L}(P_{\tilde L,\mf f})} ( q^* (\sigma^N))$.
Writing $\mf f_i = l_i^{-1} \mf f$, $\hat \sigma_i = q^* (\hat \sigma) \circ
\mr{Ad}(l_i)$ and $\sigma_i^N = q^* (\sigma^N) \circ \mr{Ad}(l_i)$, we obtain
\[
q^* \pi (\phi_L, \rho_L) = \bigoplus\nolimits_i \mr{ind}^{\tilde L}_{
N_{\tilde L}(P_{\tilde L,\mf f_i})} \big( \sigma_i^N \big) . 
\]
By Lemma \ref{lem:6.4} the pullback of $\Rep (G)_{(\hat P_{L,\mf f}, \hat \sigma)}$
along $q$ contains $\Irr (\tilde G)_{(\hat P_{\mf f'}, \hat \sigma')}$ precisely when
$(\hat P_{\mf f'}, \hat{\sigma'})$ is $\tilde G$-conjugate to one of the 
$(\hat P_{\mf f_i}, \hat \sigma_i)$. In \cite[(60)]{SolLLCunip} it was checked that $\mc H 
(G,\hat P_{\mf f}, \hat \sigma)$ and $\mc H (\tilde G, \hat P_{\mf f_i}, \hat \sigma_i)$
have the same Bernstein presentation, except for the lattices $X_{\mf f}$ and $X_{\mf f_i}$.
From \eqref{eq:6.26} and $q^*$ we get a morphism of algebraic varieties
\begin{equation}\label{eq:6.27}
\begin{array}{ccc}
\Irr (L)_{(\hat P_{L,\mf f},\hat \sigma)} \cong \Irr (X_\mf f) & \to & 
\Irr (\tilde L)_{(\hat P_{\tilde L,\mf f_i},\hat \sigma_i)} \cong \Irr (X_{\mf f_i}) \\
\pi (\phi_L,\rho_L) \otimes \chi & \mapsto & \pi (\tilde \phi_L, \rho_i) \otimes q^* (\chi)
\end{array} \qquad \chi \in X_\nr (L) .
\end{equation}
Dualizing \eqref{eq:6.27}, we see that $q$ descends to $X_{\mf f_i} \to X_{\mf f}$. 
That induces an algebra homomorphism
\begin{equation}\label{eq:6.30}
\mc H (q, \mf f_i, \hat \sigma_i) : \mc H (\tilde G, \hat P_{\mf f_i}, \hat \sigma_i) \to
\mc H (G,\hat P_{\mf f}, \hat \sigma) ,
\end{equation}
which is the identity on the span of the $N_w \; (w \in W_{\mf s})$. 

\begin{prop}\label{prop:6.5}
The pullback along $q$, the homomorphisms \eqref{eq:6.30} and the
equivalences of categories \eqref{eq:6.21} form a commutative diagram
\[
\begin{array}{ccc}
\Rep (G)_{\mf s} & \longrightarrow & \Mod \big( \mc H (G,\hat P_{\mf f}, \hat \sigma) \big) \\
\downarrow q^* & & \downarrow \bigoplus_i \mc H (q, \mf f_i, \hat \sigma_i)^* \\
\Rep (\tilde G)_{q^* (\mf s)} & \longrightarrow & 
\bigoplus_i \Mod \big( \mc H (\tilde G, \hat P_{\mf f_i}, \hat \sigma_i) \big) 
\end{array}. 
\]
The subcategories $\Rep (\tilde G)_{(\hat P_{\mf f_i}, \hat \sigma_i)}$ of $\Rep (\tilde G)$ 
associated here to the various $i$ are all different. 
\end{prop}
\begin{proof}
Notice that 
\[
\mc H (q, \mf f_i, \hat \sigma_i) = \mc H (q, \mf f_1, \hat \sigma_1) \circ \mr{Ad}(l_i) . 
\]
As $\mr{Ad}(l_i) \circ q = q \circ \mr{Ad}(l_i)$, it suffices to consider the subdiagram
involving only $\mc H (q,f_1, \hat \sigma_1)$ and (instead of $q^*$)
the composition of $q^*$ and the projection of Rep$(\tilde G)$ on 
$\Rep (\tilde G)_{(\hat P_{\mf f_1},\hat \sigma_1)}$. For $V \in \Rep (G)_{\mf s}$, the
restricted version of the diagram works out to
\begin{equation}\label{eq:6.32}
\xymatrix{
V \ar@{|->}[r] \ar@{|->}[d] & \mr{Hom}_{\hat P_{\mf f}}(\hat \sigma, V) \ar@{|..>}[d] \\
\mr{pr}_{\mf s_1} \circ q^* (V) \ar@{|->}[r] & \mr{Hom}_{\hat P_{\mf f_1}}(\hat \sigma_1, q^* (V))
= \mr{Hom}_{q (\hat P_{\mf f_1})}(\hat \sigma, V) 
}.
\end{equation}
Here the dotted arrow is a natural inclusion (coming from $q(\hat P_{\mf f_1}) \subset 
\hat P_{\mf f}$). We have to check that this dotted arrow is an isomorphism and a 
specialization of the right hand column in the diagram of the proposition.

Here we run into the problem that in general (e.g. when $\dim \tilde{\mc G} < \dim
\mc G$) there is no good map $\mc H (\tilde G) \to \mc H (G)$. To overcome this
we use the algebra $\mc H^\vee (G)$ of essentially left-compact distributions on
$G$, which was introduced in \cite{BeDe}. It naturally contains both $\mc H (G)$
and a copy of $G$. From \cite[\S 1.2]{BeDe} it is known that Mod$(\mc H^\vee (G))$
is naturally equivalent with Rep$(G)$. Moreover $\mc H (G)$ is a two-sided ideal
in $\mc H^\vee (G)$, so 
\begin{equation}\label{eq:6.29}
\mc H^\vee (G) e_{\hat \sigma} = \mc H (G) e_{\hat \sigma}.
\end{equation}
An advantage of $\mc H^\vee (G)$ over $\mc H (G)$ is that it is functorial in $G$,
see \cite[Theorem 3.1]{Moy}. In particular $q$ induces an algebra homomorphism
\[
\mc H^\vee (q) : \mc H^\vee (\tilde G) \to \mc H^\vee (G) .  
\]
The algebra homomorphism $\mc H^\vee (q)$ maps 
\begin{equation}\label{eq:6.31}
\mc H (\tilde G,\hat P_{\mf f_1}, \hat \sigma_1) \otimes \mr{End}_\C (V_{\hat \sigma}) \cong
e_{\hat \sigma_1} \mc H (\tilde G) e_{\hat \sigma_1}  
\end{equation}
to a subalgebra of $\mc H^\vee (G)$ supported on 
\[
q ( \hat P_{\mf f_1} W_{\mf s} X_{\mf f_1} \hat P_{\mf f_1} )
= q ( \hat P_{\mf f_1}) W_{\mf s} q (X_{\mf f_1}) q ( \hat P_{\mf f_1}) .
\]
From the support and the multiplication relations we see that the image of \eqref{eq:6.31}
under $\mc H^\vee (q)$ is isomorphic to the subalgebra of 
$\mc H (G,\hat P_{\mf f}, \hat \sigma) \otimes \mr{End}_\C (V_{\hat \sigma})$ obtained
by replacing $X_{\mf f}$ by the sublattice $q(X_{\mf f_1})$. The canonical injection
\[
\mc H^\vee (q) \big( e_{\hat \sigma_1} \mc H (\tilde G) e_{\hat \sigma_1} \big) \to
e_{\hat \sigma} \mc H^\vee (G) e_{\hat \sigma} 
\]
is given by $h \mapsto e_{\hat \sigma} h e_{\hat \sigma}$. For an element $T$ of 
\eqref{eq:6.31}, the comparison of the above with \eqref{eq:6.30} results in
\begin{equation}\label{eq:6.33}
(\mc H (q,\mf f_1, \hat \sigma_1) \otimes \mr{id}) (T) =
e_{\hat \sigma} \mc H^\vee (q)(T) e_{\hat \sigma} .
\end{equation}
This shows that the dotted arrow in \eqref{eq:6.32} corresponds to 
$\mc H (q,\mf f_1,\hat \sigma_1)^*$.

Let $\mc P$ be a parabolic $F$-subgroup of $\mc G$ with Levi factor $\mc L$. 
Its inverse image $\tilde{\mc P}$ in $\tilde{\mc G}$ is a parabolic $F$-subgroup 
with Levi factor $\tilde{\mc L}$. As $\mc N \subset \mc P$, $q$ induces an isomorphism
of $F$-varieties $\tilde G / \tilde P \to G / P$. Consequently the (normalized) parabolic
induction functors for $G$ and $\tilde G$ are related by
\[
q^* (I_P^G (\pi_L)) = I_{\tilde P}^{\tilde G}(q^* (\pi_L)) \qquad \pi_L \in \Rep (L) . 
\]
For a fixed $\pi_L = \pi (\phi_L,\rho_L) \in \Irr_{\cusp,\unip}(L)$ and $\chi \in X_\nr (L)$,
Lemma \ref{lem:6.4} says that $q^* (\pi_L \otimes \chi) = q^* (\pi_L) \otimes q^* (\chi)$ is
multiplicity-free. More precisely:
\begin{equation}\label{eq:6.45}
\mr{pr}_{\mf s_{L,i}} (q^* (\pi (\phi_L,\rho_L) \otimes \chi)) = 
\pi (\tilde \phi_L, \tilde \rho_L) \otimes q^* (\chi)
\end{equation}
for a unique $\tilde \rho_L \in \Irr (\mc S_{\tilde \phi})$. 

For almost all central characters of $\mc H (\tilde G, \hat P_{\mf f_i}, \hat \sigma_i)$
(i.e. on a Zariski-dense open subset of the space of central characters), that algebra
has just one irreducible representation with the given central character. That follows
from standard representation theory of affine Hecke algebras \cite{Mat}. It is easy to
construct this irreducible representation: take a character of Bernstein's maximal 
commutative subalgebra, in generic position, and induce it to  
$\mc H (\tilde G, \hat P_{\mf f_i}, \hat \sigma_i)$. This is an analogue
of the \cite[Th\'eor\`eme IV.1]{Sau}, which says that parabolic induction for representations
of reductive $p$-adic groups preserves irreducibility of representations in generic position.

For $\chi$ in a non-empty Zariski-open subset of $X_\nr (L)$, the $L$-representation 
\eqref{eq:6.45} determines a central character of $\mc H (G,\hat P_{\mf f}, \hat \sigma)$ 
which is in generic position. For such $\chi$ the $G$-representation 
$I_P^G (\pi (\phi_L,\rho_L) \otimes \chi)$, its image in 
$\Mod \big( \mc H (G,\hat P_{\mf f}, \hat \sigma) \big)$ and the further image  
under $\mc H (q,\mf f_1,\hat \sigma_1)^*$ are irreducible. It follows that whenever 
$V =  I_P^G (\pi (\phi_L,\rho_L) \otimes \chi)$ for such a $\chi$, all terms in 
\eqref{eq:6.32} are irreducible. In particular, for such $V$ the diagram \eqref{eq:6.32} 
commutes, the dotted arrow is an isomorphism and 
\begin{equation}\label{eq:6.34}
\mr{pr}_{\mf s_i} \circ q^* \big( I_P^G (\pi (\phi_L,\rho_L) \otimes \chi) \big) =
I_{\tilde P}^{\tilde G} \big( \pi (\tilde \phi_L, \tilde \rho_L) \otimes q^* (\chi) \big) .
\end{equation}
But the structure of \eqref{eq:6.34} as a representation
of the compact group $\hat P_{\mf f_1}$ does not depend on $\chi \in X_\nr (L)$, so the
dotted arrow in \eqref{eq:6.32} is an isomorphism for all standard modules $I_P^G (\pi)$ with 
$\pi \in \Irr (L)_{\mf s_L}$. Hence $\mr{Hom}_{q (\hat P_{\mf f})}(\hat \sigma, V) = 
\mr{Hom}_{\hat P_{\mf f}}(\hat \sigma, V)$ for all $V \in \Rep (G)_{\mf s}$, which means
that the diagrams from \eqref{eq:6.32} and from the statement of the proposition always commute.

This also shows that the left hand map in \eqref{eq:6.32} kills all terms from 
$(\mf f_i, \hat \sigma_i)$ with $i \neq 1$. Hence 
$\Rep (\tilde G)_{(\hat P_{\mf f}, \hat \sigma)}$ differs from all the subcategories 
$\Rep (\tilde G)_{(\hat P_{\mf f_i}, \hat \sigma_i)}$ of $\Rep (\tilde G)$ with $i \neq 1$.
\end{proof}

\begin{lem}\label{lem:6.6}
Conjecture \ref{conj:5.7} holds for unipotent representations, with respect to the quotient 
map $q : \tilde{\mc G} \to \mc G = \tilde{\mc G} / \mc N$.
\end{lem}
\begin{proof}
It follows from \cite[(58)]{SolLLCunip} that the affine Hecke algebras 
$\mc H (\mf s^\vee,\vec{\mb z})$ and $\mc H (\mf s^\vee_i,\vec{\mb z})$ have almost the same
Bernstein presentation, only the tori can differ. The map
\begin{equation}\label{eq:6.35}
\mf s_L^\vee \to \mf s^\vee_{\tilde L,i} : (\phi_L, \rho_L) \mapsto (\tilde \phi_L, \rho_i)
\end{equation}
induces an algebra homomorphism 
\[
\mc H (\mf s_i^\vee ,\vec{\mb z}) \to \mc H (\mf s^\vee ,\vec{\mb z})
\]
as in \eqref{eq:4.8}. Notice that $\mc H (\mf s^\vee, W_{\mf s_i}, \vec{\mb z}) =
\mc H (\mf s^\vee, \vec{\mb z})$, because $W_{\mf s^\vee_i} = W_{\mf s^\vee}$. Lemma \ref{lem:6.4}
shows that, via the LLC \eqref{eq:6.24}, \eqref{eq:6.35} translates to
\begin{equation}\label{eq:6.36}
\mr{pr}_{\mf s_{\tilde L,i}} \circ q^* : \Irr (L)_{\mf s_L} \to \Irr (\tilde L)_{\mf s_{\tilde L,i}} . 
\end{equation}
From \eqref{eq:4.8}, \eqref{eq:6.23} and \eqref{eq:6.30} we get a commutative diagram
\begin{equation}\label{eq:6.37}
\begin{array}{ccc}
\mc H (\mf s^\vee, \vec{\mb z}) & \to & \mc H (G, \hat P_{\mf f}, \hat \sigma) \\
\uparrow & & \uparrow \\
\mc H (\mf s^\vee_i,\vec{\mb z}) & \to & \mc H (\tilde G, \hat P_{\mf f_i}, \hat \sigma_i) 
\end{array},
\end{equation}
in which the horizontal arrows become isomorphisms upon specializing $\vec{\mb z}$ to $\vec z$.
By Theorem \ref{thm:4.5} 
\begin{multline}\label{eq:6.38}
\bigoplus_i \mc H (q, \mf f_i, \hat \sigma_i)^* \, \overline{M}(\phi,\rho,\vec{z}) =
\bigoplus_{\tilde \rho \in \Irr (\mc S_{\tilde \phi})}
\mr{Hom}_{\mc S_{\tilde \phi}} (\mr{ind}_{\mc S_\phi}^{\mc S_{\tilde \phi}}(\rho), 
\tilde \rho) \otimes \overline{M} (\tilde \phi, \tilde \rho, \vec{z}) \\
= \bigoplus\nolimits_{\tilde \rho \in \Irr (\mc S_{\tilde \phi})} \mr{Hom}_{\mc S_\phi} 
(\rho, {}^S q^* (\tilde \rho)) \otimes \overline{M} (\tilde \phi, \tilde \rho, \vec{z}) .
\end{multline}
By Proposition \ref{prop:6.5} and the commutativity of \eqref{eq:6.37}, 
\eqref{eq:6.38} implies that
\[
q^* \pi (\phi,\rho) = \bigoplus\nolimits_{\tilde \rho \in \Irr (\mc S_{\tilde \phi})} 
\mr{Hom}_{\mc S_\phi} (\rho, {}^S q^* (\tilde \rho)) \otimes \pi (\tilde \phi, \tilde \rho) . 
\qedhere
\]
\end{proof}

\subsection{Isomorphisms} \

We move on to functoriality of the LLC for unipotent representations with respect to
isomorphisms of reductive groups. First we consider the action of $G_\ad$ on $G$ by
conjugation. Recall from \eqref{eq:2.33} that it suffices investigate Ad$(g)$ with
$g \in T_{AD} = (\mc T / Z(\mc G))(F)$.

\begin{lem}\label{lem:6.11}
Condition \ref{cond:2} holds in this setting.
\end{lem}
\begin{proof}
From \eqref{eq:6.23} we already concluded that $\mf R_{\mf s^\vee} = 1$, so part (i) 
holds trivially.

Since $\phi$ is unramified and $\mc G$ splits over an unramified extension of $F$, there
exists $s \in G^\vee$ such that $\phi (\Fr_F \gamma) = (s, \Fr_F \gamma)$ for all 
$\gamma \in \mb I_F$. By \cite[Lemma 6.5]{Bor} we may assume that $s \in T_c^\vee$. Then
\[
X_\nr (T_{AD}) = ( T_c^{\mb I_F})^\circ_{\Fr_F} = ( T_c^{\phi (\mb I_F)})^\circ_{\phi (\Fr_F)}
\]
equals the torus called $(T_J )_{\phi (\Fr_F)}$ in \cite[Proposition 3.13]{AMS3}. Via the 
natural map $X_\nr (T_{AD}) \to T_{\mf s^\vee}$ every $\alpha \in R_{\mf s^\vee}$ lifts to 
a root for $X_\nr (T_{AD})$ and from \cite[(89)]{AMS3} we know that 
$\alpha^\vee \in X_* (X_\nr (T_{AD}))$. For every $x \in X^* (X_\nr (T_{AD}))$:
\[
s_\alpha (x) - x = -\langle x ,\, \alpha^\vee \rangle \alpha \in \Z \alpha .
\]
It follows that $w(x) - x$ lies in the root lattice $\Z R_{\mf s^\vee}$, for every
$w \in W_{\mf s^\vee} \cong W(R_{\mf s^\vee})$. 
In particular $w (x) - x \in X^* (T_{\mf s^\vee})$, which says that part (ii) holds.
\end{proof}

We fix a "fundamental" chamber $C_0$ in the apartment $\mh A$ of $\mc B (\mc G,F)$ 
corresponding to the tori $\mc S$ and $\mc T$. Upon replacing 
$P_{\mf f}$ by a $G$-conjugate, we may assume that $\mf f \subset \overline{C_0}$. 

\begin{lem}\label{lem:6.7}
Let $g \in T_{AD,\cpt}$. 
\enuma{
\item The action \eqref{eq:2.20} of $g$ on $\Phi_{\nr,e}(G)$ is trivial.
\item Ad$(g)$ stabilizes $(P_{\mf f},\sigma)$ and $(\hat P_{\mf f},\hat \sigma)$.
\item The action of Ad$(g)$ on $\Rep (G)_{(\hat P_{\mf f},\hat \sigma)}$ is trivial.
} 
\end{lem}
\begin{proof}
(a) For $\phi \in \Phi_\nr (G)$ and $h \in (G^\vee )^{\phi (\mb W_F)}$, \eqref{eq:2.16} gives
\[
c_h \in H^1 (\mb W_F / \mb I_F, Z({G^\vee}_\Sc)) \cong X_\Wr (G_\ad) . 
\]
Since $g$ lies in the unique parahoric subgroup of the unramified torus $T_{AD}$, $g$ lies in
a parahoric subgroup of $G_\ad$. Hence it lies in the kernel of any weakly unramified character.
In particular
\[
1 = \langle g, c_h \rangle = \tau_{\phi,\mc G}(g)(h) . 
\]
This says that $\tau_{\phi,\mc G}(g) = 1$, and then \eqref{eq:2.20} reduces to the trivial action.\\
(b) The compact element $g$, or equivalently the automorphism Ad$(g)$ of $G$, fixes $\mh A$ pointwise. 
In particular Ad$(g)$ stabilizes $\mf f$, $P_{\mf f}$ and $\hat P_{\mf f}$. 
Let $P_{\mf f,\ad}$ be the parahoric subgroup of $G_\ad$ associated to $\mf f$. As $g$ fixes
$\mf f$ pointwise, $g \in \hat P_{\mf f,\ad}$.

As in \cite[(18) and (19)]{SolLLCunip}, we consider the group
\[
\Omega = \{ \omega \in N_G (S) / (N_G (S) \cap P_{C_0}) : \omega (C_0) = C_0 \} .
\]
It was remarked after \cite[(19)]{SolLLCunip} that $\Omega$ is isomorphic to the image of the Kottwitz
isomorphism for $G$. As domain of the Kottwitz isomorphism one can take $X_* (S) / \Z \Phi (G,S)$,
so $\Omega$ is naturally isomorphic to the abelian group $X_* (S) / \Z \Phi (G,S)$. Then
$\hat P_{\mf f} / P_{\mf f}$ is naturally isomorphic to the pointwise stabilizer of $\mf f$
in $\Omega$. Furthermore the torsion subgroup of $\Omega$ embeds naturally in its version for $G_\ad$,
so $G \to G_\ad$ induces an injection
\begin{equation}\label{eq:6.39}
\hat P_{\mf f} / P_{\mf f} \hookrightarrow \hat P_{\mf f,\ad} / P_{\mf f,\ad} .
\end{equation}
By \cite[Proposition 3.15]{Lus-Che} there exists a cuspidal unipotent representation $\sigma_\ad$
of $P_{\mf f,\ad}$, whose pullback to $P_{\mf f}$ is $\sigma$. Choose an extension $\hat \sigma_\ad$
of $\sigma_\ad$ to $\hat P_{\mf f,\ad}$ (still on the same vector space $V_\sigma$). 
Any two such extensions differ by a character of
$\hat P_{\mf f,\ad} / P_{\mf f,\ad}$. With \eqref{eq:6.39} we can arrange that the pullback of
$\hat \sigma_\ad$ to $\hat P_{\mf f}$ is $\hat \sigma$. Using $g \in \hat P_{\mf f,\ad}$ we get
\begin{multline*}
(\mr{Ad}(g)^* \hat \sigma) (p) = \hat \sigma (g p g^{-1}) = \hat \sigma_\ad (g p g^{-1}) \\
= \hat \sigma_\ad (g) \hat \sigma_\ad (p) \hat \sigma_\ad (g)^{-1} 
= \hat \sigma_\ad (g) \hat \sigma (p) \hat \sigma_\ad (g)^{-1} \qquad p \in \hat P_{\mf f} .
\end{multline*}
Hence $\mr{Ad}(g)^* \hat \sigma$ is equivalent with $\hat \sigma$. 
The same calculation shows that $\mr{Ad}(g)^* \sigma$ is equivalent with $\sigma$.\\
(c) As in the case of principal series representations \eqref{eq:6.10}, part (b) entails 
that Ad$(g)$ naturally defines an automorphism $\mc H (\mr{Ad}(g))$ of 
$\mc H (G,\hat P_{\mf f},\hat \sigma)$. By \eqref{eq:6.12} it implements Ad$(g)^*$ on 
Mod$(\mc H (G,P_{\mf f}, \hat \sigma))$. The same calculations as between \eqref{eq:6.12} and 
\eqref{eq:6.40} show that $\mc H (\mr{Ad}(g))$ is the specialization at $\vec{\mb z} = \vec{z}$
of $\alpha_g$ from \eqref{eq:2.54}. Here $\mf R_{\mf s} = 1$, so both $\alpha_g$ and
$\mc H (\mr{Ad}(g))$ are the identity maps. 
\end{proof}

Recall from \eqref{eq:2.51} and \eqref{eq:2.53} that $X_* (T_{AD})$ is embedded in $T_{AD}$
via evaluation at $\varpi_F$, and that $T_{AD} = T_{AD,\cpt} \times X_* (T_{AD})$.

\begin{lem}\label{lem:6.8}
Conjecture \ref{conj:5.7} holds for unipotent representations, for Ad$(g)$
with $g \in X_* (T_{AD})$.  
\end{lem}
\begin{proof}
An argument analogous to \eqref{eq:6.15}--\eqref{eq:6.18}, now also using the isomorphism 
\eqref{eq:6.23}, gives a commutative diagram
\begin{equation}\label{eq:6.43}
\xymatrix{
\Rep (G)_{\mf s} \ar[r]^-{\sim} \ar[d]_{\mr{Ad}(g)^*} & 
\Mod \big( \mc H (G,\hat P_{\mf f}, \hat \sigma) \big) \ar[d]_{\mr{Ad}(x_g)^*} 
\ar[r]_{(\mr{ev}_{\vec{\mb z} = \vec{z}} )^* } &
\Mod (\mc H (\mf s^\vee, \vec{\mb z})) \ar[d]_{\mr{Ad}(x_g)^*} \\
\Rep (\tilde G)_{\tilde{\mf s}} \ar[r]^-{\sim}  & 
\Mod \big( \mc H (\tilde G, \hat P_{\tilde{\mf f}}, \hat{\tilde \sigma}) \big) 
\ar[r]_{(\mr{ev}_{\vec{\mb z} = \vec{z}} )^*} &
\Mod (\mc H (\tilde{\mf s^\vee}, \vec{\mb z})) 
} .
\end{equation}
As Artin reciprocity sends $\Fr_F$ to $\varpi_F$, $\tau_\phi (g)$ depends only on
$\phi (\Fr_F)$, see \eqref{eq:6.42}. 

Recall that Condition \ref{cond:2} was checked in Lemma \ref{lem:6.11}. The algebra
isomorphism $\alpha_g$ from \eqref{eq:2.54} is just the identity. Namely, as $W_{\mf s^\vee}$
is the Weyl group of the root system $\Phi_{\mf s^\vee}$, Lemma \ref{lem:2.1} tells us that
$\tau_{\mc S_{\phi_L}}(g)(\dot w) = 1$ for all $w \in W_{\mf s^\vee}$.

Now we apply Theorem \ref{thm:2.9}, which tells us that the effect of Ad$(x_g)^*$
on enhancements of an L-parameter $\phi$ is tensoring by $\tau_\phi (g)^{-1}$. The diagram
\eqref{eq:6.43} and \eqref{eq:6.24} then show that 
\[
\mr{Ad}(g)^* \pi (\phi,\rho) = \phi (\phi, \rho \otimes \tau_\phi (g)^{-1}) .
\]
As in the last lines of the proof of Theorem \ref{thm:6.3}, this formula is equivalent
to Conjecture \ref{conj:5.7} for Ad$(g)$.
\end{proof}

Let $\zeta_{\mc G}^+, \zeta^+_{\tilde{\mc G}} \in \Irr (Z({G^\vee}_\Sc))$ be the characters
used to define relevance of enhanced L-parameters in \eqref{eq:1.4}. Let $\mc G^*$ be the 
quasi-split inner form of $\mc G$ recall the notations from Section \ref{sec:3}. Consider a 
$\mb W_F$-equivariant isomorphism of based root data
\[
\tau : \mc R (\tilde{\mc G}^*, \tilde{\mc T}^*) \to \mc R (\mc G^*, \mc T^*) 
\]
such that $\tau (\zeta_{\tilde{\mc G}}^+) = \zeta_{\mc G}^+$. By Theorem \ref{thm:3.2}
$\tau^\vee$ induces a $\mb W_F$-equivariant isomorphism $\eta_\tau^\vee : G^\vee \to 
{\tilde G}^\vee$. Pick an isomorphism of $F$-groups $\eta_{\tau, \mc G} : 
\tilde{\mc G} \to \mc G$ as in Proposition \ref{prop:3.4}.c.

\begin{lem}\label{lem:6.9}
Conjecture \ref{conj:5.7} holds for unipotent representations, for $\eta_{\tau, \mc G}$ as above.
\end{lem}
\begin{proof}
By Proposition \ref{prop:3.4}.d, $\tau$ determines $\eta_{\tau,\mc G}$ uniquely up to
conjugation by compact elements $g_c \in Z_{G_\ad}(S)$. Considering $g_c$ in $G_\ad / G$,
\eqref{eq:2.33} and \eqref{eq:2.53} allow us to realize it in $T_{AD,\cpt}$. Then Lemma 
\ref{lem:6.7} says that Ad$(g_c)$ acts trivially on both $\Irr_\unip (G)$ and
$\Phi_{\nr,e}(G)$. As a consequence, the action of $\eta_{\tau, \mc G}$ on unipotent
representations and the action of ${}^L \eta_{\tau,\mc G}$ on enhanced unramified 
L-parameters are uniquely determined by $\tau$. This means that Condition \ref{cond:3}
is fulfilled. 

It was shown in \cite[Theorem 2.(3)]{FOS} that the LLC for supercuspidal unipotent 
representations is equivariant with respect to automorphisms of root data. The conventions
in \cite[p.8--9]{FOS} make this statement precise: it means that 
\begin{equation}\label{eq:6.41}
\begin{aligned}
& \eta_{\tau,\mc G}^* : \Irr_{\cusp,\unip}(G) \to \Irr_{\cusp,\unip}(\tilde G) 
\qquad \text{corresponds to} \\
& \Phi_e ({}^L \eta_{\tau,\mc G}) : \Phi_{\nr,\cusp}(G) \to \Phi_{\nr,\cusp}(\tilde G) .
\end{aligned}
\end{equation}
Write $\eta_{\tau,\mc G}^{-1}(\hat P_{\mf f}) = P_{\tilde{\mf f}} ,\; \eta_{\tau,\mc G}^* 
(\hat \sigma) = \hat{\tilde \sigma}$ and $\eta_{\tau,\mc G}^* (\mf s) = \tilde{\mf s}$.  
For Bernstein components and types, \eqref{eq:6.41} means that $\Phi_e ({}^L \eta_{\tau,\mc G})
(\mf s^\vee) = \tilde{\mf s}^\vee$ and
\[
\Phi_e ({}^L \eta_{\tau,\mc G}) (\phi_{\pi_L},\rho_{\pi_L}) = (\phi_{\eta_{\tau,\mc G}^* \pi_L}, 
\rho_{\eta_{\tau,\mc G}^* \pi_L})  \qquad \pi_L \in \Irr (L)_{\mf s_L} .
\]
Then $\Phi_e ({}^L \eta_{\tau,\mc G})$ restricts to a homeomorphism $\mf s_L^\vee \to
\tilde{\mf s}_{\tilde L}^\vee$. We define an algebra homomorphism
\[
\begin{array}{cccc}
\mc H ({}^L \eta_{\tau,\mc G}) : & \mc H (\tilde{\mf s}^\vee, \vec{\mb z}) & \to &
\mc H (\mf s^\vee, \vec{\mb z}) \\
& N_w f & \mapsto & 
N_{(\tau^\vee)^{-1} w} f \circ \Phi_e ({}^L \eta_{\tau,\mc G}) \big|_{\mf s_L^\vee}
\end{array} ,
\]
as in \eqref{eq:3.2}. Now we are in a setting analogous to that for principal series and
isomorphisms of root data. We can construct a diagram
\[
\xymatrix{
\Rep (G)_{\mf s} \ar[r]^-{\sim} \ar[d]_{\eta_{\tau,\mc G}^*} & 
\Mod \big( \mc H (G,\hat P_{\mf f}, \hat \sigma) \big) \ar[d]_{\mc H (\eta_{\tau,\mc G})^*} 
\ar[r]_{(\mr{ev}_{\vec{\mb z} = \vec{z}} )^* } &
\Mod (\mc H (\mf s^\vee, \vec{\mb z})) \ar[d]_{\mc H ({}^L \eta_{\tau,\mc G})^*} \\
\Rep (\tilde G)_{\tilde{\mf s}} \ar[r]^-{\sim}  & 
\Mod \big( \mc H (\tilde G, \hat P_{\tilde{\mf f}}, \hat{\tilde \sigma}) \big) 
\ar[r]_{(\mr{ev}_{\vec{\mb z} = \vec{z}} )^*} &
\Mod (\mc H (\tilde{\mf s^\vee}, \vec{\mb z})) 
} .
\]
The same arguments as in the principal series case (see the proofs of Theorem \ref{thm:6.3}
and \cite[Lemma 17.5]{ABPSprin}), the diagram commutes. Corollary \ref{cor:3.1} translates
the right column to the desired effect on enhanced L-parameters.
\end{proof}

With the previous lemmas at hand, it is only a small step to prove Conjecture \ref{conj:5.7}
for unipotent representations, with respect to all eligible homomorphisms.

\begin{thm}\label{thm:6.10}
Let $\mc G, \tilde{\mc G}$ be connected reductive $F$-groups which split over an unramified
extension of $F$. Let $f : \tilde{\mc G} \to \mc G$ be a homomorphism satisfying
Conditions \ref{cond:1}. Let $(\phi,\rho) \in \Phi_{\nr,e}(G)$ and write
$\tilde \phi = {}^L f \circ \phi \in \Phi_\nr (\tilde G)$. Then
\[
f^* (\pi (\phi,\rho)) = \bigoplus\nolimits_{\tilde \rho \in \Irr (\mc S_{\tilde \phi})}
\mr{Hom}_{\mc S_\phi} \big(\rho, {}^S f^* (\tilde \rho) \big) \otimes \pi (\tilde \phi, \tilde \rho) .
\]
\end{thm}
\begin{proof}
We factorize $f$ as in \eqref{eq:5.1}, and to avoid confusion we add a prime to the 
involved torus. We consider the three factors of $f$ separately. For 
\[
f_2 : \tilde{\mc G} \times \mc T' \to \tilde{\mc G} \times \mc T' / \ker (f \times \mr{id}_{\mc T'})
\]
see Lemma \ref{lem:6.6}. By Proposition \ref{prop:3.4} and \eqref{eq:2.53} we can write
\[
f_3 = \mr{Ad}(g_x) \circ \mr{Ad}(g_c) \circ \eta_{\mc R (f_3),\mc G} :
\tilde{\mc G} \times \mc T' / \ker (f \times \mr{id}_{\mc T'}) \to \mc G ,
\]
where $g_x \in X_* (T_{AD})$ and $g_c \in T_{AD,\cpt}$. For Ad$(g_c)$ see Lemma \ref{lem:6.7},
for Ad$(g_x)$ see Lemma \ref{lem:6.8} and for $\eta_{\mc R (f_3),\mc G}$ see Lemma \ref{lem:6.9}.

It remains to prove the theorem for the inclusion
\[
f_1 : \tilde{\mc G} \to \tilde{\mc G} \times \mc T',  
\]
where $\mc T'$ is an unramified torus. 
Any $\pi \in \Irr_\unip (\tilde{G} \times T')$ is of the form 
$\tilde \pi \otimes \chi_{T'}$ with $\tilde \pi \in \Irr_\unip (\tilde{G})$
and $\chi_{T'} \in X_\Wr (\mc T (K))$. The restriction (or pullback) of $\pi$ to
$\tilde{G}$ is $\tilde \pi$, which is again irreducible. 

Let $(\tilde \phi,\tilde \rho)$ and $(\phi_{T'},1)$ be the enhanced L-parameters
of, respectively,  $\tilde \pi$ and $\chi_{T}$, via \cite[Theorem 5.1]{SolLLCunip}.
By \cite[Lemma 5.2]{SolLLCunip} the enhanced L-parameter of $\pi$ is 
$(\tilde \phi \times \phi_{T'},\tilde \rho \otimes 1)$. Since 
${(\tilde{\mc G} \times \mc T)^\vee}_\Sc = {\tilde{\mc G}^\vee}_\Sc$, we can identify 
$\mc S_{\tilde \phi \times \phi_{T'}}$ with $\mc S_{\tilde \phi}$ and 
$\tilde \rho \otimes 1$ with $\tilde \rho$. Summarizing:
\[
f_1^* (\pi) = f_1^* \pi (\tilde \phi \times \phi_{T'}, \tilde \rho \otimes 1) =
\tilde \pi = \pi (\tilde \phi, \tilde \rho) .
\]
As ${}^S f_1$ is just the identity, this can also be expressed as
\[
\bigoplus\nolimits_{\rho' \in \Irr (\mc S_{\tilde \phi})} \mr{Hom}_{\mc S_{\tilde \phi \times 
\phi_{T'}}} \big( \tilde \rho \otimes 1, {}^S f_1^* (\rho') \big) 
\otimes \pi (\tilde \phi, \rho') . \qedhere
\]
\end{proof}

\section{Well-known groups}
\label{sec:8}

In this section $F$ can be any local field. We verify that the LLC for several well-known 
groups is functorial in the sense of Conjecture \ref{conj:5.7}. 

\subsection{$GL_n$ and its inner forms} \
\label{par:GLn}

First we consider the LLC for $GL_n (F)$, established in \cite{Lan,LRS,HaTa,Hen1,Zel}.
We denote it by $\pi \mapsto \phi_\pi$ and $\phi \mapsto \pi (\phi,1)$ (where 1 means the
enhancement of $\phi$, which necessarily trivial).

The adjoint group of $GL_n$ is $PGL_n$ and the canonical map
\[
GL_n (F) \to PGL_n (F)
\]
is surjective. Hence the actions of $PGL_n (F)$ on $\Irr (GL_n (F))$ and on $\Phi_e (GL_n (F))$
are trivial. By Corollary \ref{cor:3.3} the actions of $\mr{Aut}_F (GL_n)$ on these sets factor
through the group of automorphisms of the based root datum of $GL_n$ (with respect to the diagonal
torus $\mc T$). The only nontrivial automorphism of $\mc R (GL_n, \mc T)$ is
\begin{equation}\label{eq:8.1}
\tau : x \mapsto (n \, n \!-\! 1 \cdots 2 \, 1) (-x)    \qquad x \in \Z^n .
\end{equation}
Denote the inverse transpose of a matrix $A$ by $A^{-T}$ and let $P_\sigma \in GL_n (\Z)$
be the permutation matrix associated to $\sigma \in S_n$. Then \eqref{eq:8.1} lifts to the 
automorphism
\begin{equation}\label{eq:8.2}
\begin{array}{clcl}
\eta_\tau : & GL_n (F) & \to & GL_n (F) \\ 
 & g & \mapsto & \mr{Ad} \big( \mr{diag}(1,-1,1,\cdots,(-1)^n) P_{(n \, n-1 \cdots 2 \, 1)} \big) g^{-T}
\end{array}.
\end{equation}
The advantage of including an adjoint action in $\eta_\tau$ is that now it preserves the 
standard pinning of $GL_n (F)$. We note that the same formula defines an automorphism
$\eta_\tau^\vee$ of $GL_n (\C)$. Recall that as Langlands parameters for $GL_n (F)$ we can take
homomorphisms $\mb W_F \times SL_2 (\C) \to GL_n (\C)$, that is $n$-dimensional representations of
$\mb W_F \times SL_2 (\C)$. More or less by definition 
\[
\Phi ({}^L \eta_\tau)(\phi) = \eta_\tau^\vee \circ \phi = \phi^\vee 
\qquad \phi \in \Phi (GL_n (F)) .
\]
It is well-known \cite[Theorem 1]{GeKa} that $\pi \circ \eta_\tau \cong \tau^\vee$ 
(the contragredient of $\tau$) for any $\pi \in \Irr (GL_n (F))$. 
By \cite[Theorem 1.3]{AdVo}, when $F$ is archimedean,
\begin{equation}\label{eq:8.3}
\pi (\phi^\vee,1) = \pi (\phi,1)^\vee  \in \Irr (GL_n (F)) 
\end{equation}
for all $\phi \in \Phi (GL_n (F))$. 

When $F$ is non-archimedean, \cite[Th\'eor\`eme 4.2]{Hen1} and \cite[Theorem 15.7]{LRS} show
\eqref{eq:8.3} for all irreducible $n$-dimensional representations $\phi$ of $\mb W_F$. Such $\phi$ 
(trivial on $SL_2 (\C)$ and with trivial enhancement) form precisely the cuspidal Langlands parameters 
for $GL_n (F)$, see \cite[Example 6.11]{AMS1}. With the work of Bernstein and Zelevinsky \cite{Zel},
the LLC is extended from the cuspidal level to the whole of $\Irr (GL_n (F))$ and $\Phi_e (GL_n (F))$.
It was shown in \cite[\S 2.9]{Hen2} that this LLC satisfies \eqref{eq:8.3} for all
$(\phi,1) \in \Phi_e (GL_n (F))$. In other words,
\begin{equation}\label{eq:8.4}
\pi (\Phi_e ({}^L \eta_\tau) (\phi,1)) = \eta_\tau^* (\pi (\phi,1)) 
\qquad (\phi,1) \in \Phi_e (GL_n (F)).
\end{equation}
Consider a division algebra $D$ with centre $F$. Assume that $\dim_F (D) = d^2$ and $m = n/d \in \N$.
Then $G = GL_m (D)$ is an inner twist of $GL_n (F)$, and every inner twist looks like this. 
The LLC for $GL_m (D)$ follows from the LLC for $GL_n (F)$ via the Jacquet--Langlands correspondence
\cite{JaLa,DKV,Bad}, see \cite[\S 2]{ABPSSL1}. 

We take a brief look at the construction. Recall that a $G$-representation
$\pi$ is called essentially square-integrable (indicated by a subscript "ess$L^2$") if the 
restriction of $\pi$ to the derived group $G_\der$ is square-integrable. According to \cite{DKV}
and \cite[Th\'eor\`eme 1.1]{Bad} there is a canonical bijection
\[
JL_D : \Irr_{\mr{ess}L^2}(GL_n (F)) \to \Irr_{\mr{ess}L^2}(GL_m (D)) 
\]
determined by the equality
\begin{equation}\label{eq:8.26}
(-1)^n \mr{tr}_\pi (g') = (-1)^m \mr{tr}_{JL_D (\pi)}(g) 
\end{equation}
whenever $g \in GL_m (D)$ and $g' \in GL_n (F)$ are matching regular elliptic elements. Here
matching is determined simply by an equality of characteristic polynomials.

The LLC on essentially square-integrable $G$-representations is given by 
\[
\pi (\phi,\zeta_D) = JL_D (\pi (\phi,1))  \qquad \phi \in \Phi_{\mr{disc}}(GL_m (D)) .
\]
In particular such representations correspond precisely to discrete L-parameters. 
The LLC for Irr$(G)$ is derived from that for $\Irr_{\mr{ess}L^2}(G)$ by parabolic induction
and Langlands quotients, see \cite[(13)]{ABPSSL1}. 

For this $G$ we have to be careful not to forget the enhancements of L-parameters. It is easy
to see that for any $\phi \in \Phi (GL_n (F))$ the natural map $Z(SL_n (\C)) \to \mc S_\phi$ 
is surjective. So any enhancement of $\phi$ can be regarded as a character of $Z(SL_n (\C)) 
\cong \Z / n \Z$. Then $(\phi,\rho)$ is relevant for $GL_m (D)$ if and only if $\rho$ equals
the character $\zeta_D \in \Irr \big( Z(SL_n (\C)) \big)$ determined by $GL_m (D)$ via the
Kottwitz homomorphism. We remark that $\zeta_D$ has order $d = n /m$. As the notation suggests, 
it depends only on $D$ -- it is also the parameter of $D^\times$ as inner twist of $GL_d (F)$.

The Kottwitz parameter for the opposite division algebra $D^{op}$ is $\zeta_{D^{op}} = \zeta_D^{-1}$.
Sometimes $D^{op} \cong D$ even when $D$ is not commutative. Namely, this happens when
$\zeta_D$ has order two, and then $D$ is a quaternion algebra over $F$. In that case there is
a canonical $F$-algebra isomorphism $D \to D^{op} : d \mapsto \bar d$ , a generalization of the 
conjugation map on the quaternion algebra over $\R$. It induces an isomorphism of $F$-groups 
\begin{equation}\label{eq:8.25}
GL_m (D) \to GL_m (D^{op}) : g \mapsto \bar g  \hspace{1cm} \text{when } \zeta_D \text{ has order 2.}
\end{equation}
If $g \in GL_m (D)$ matches $g' \in GL_n (F)$, then $\bar g \in GL_m (D^{op})$ also matches $g'$, 
because $g$ and $\bar g$ have the same characteristic polynomial. It follows that \eqref{eq:8.25} 
intertwines the local Langlands correspondences for $GL_m (D)$ and $GL_m (D^{op})$.

Like for $GL_n (F)$, the adjoint quotient map
\[
q : G = GL_m (D) \to G_\ad = PGL_m (D)
\]
is surjective, and $PGL_m (D)$ acts trivially on $\Irr (GL_m (D))$ and $\Phi_e (GL_m (D))$. 
As before, the only nontrivial action of $\mr{Aut}_F (GL_n)$ can come from \eqref{eq:8.1} and
\eqref{eq:8.2}. 
Since $\eta_\tau^\vee$ acts on $Z(GL_n (\C))$ by inversion, 
$\rho \circ (\eta_\tau^\vee)^{-1} = \rho^{-1}$. This means that $\Phi_e ({}^L \eta_\tau)$ sends
$\Phi_e (GL_m (D))$ to $\Phi_e (GL_m (D^{op}))$. 

Furthermore, on closer inspection $g \mapsto g^{-T}$ only is an automorphism of $GL_m (D)$ if
$D = D^{op}$. In general it gives an isomorphism $GL_m (D^{op}) \to GL_m (D)$. Hence we should
regard \eqref{eq:8.2} as an isomorphism of $F$-groups
\[
\eta_\tau : GL_m (D^{op}) \to GL_m (D) .
\]
When $\zeta_D$ has order two, we can (implicitly) precompose $\eta_\tau$ with \eqref{eq:8.25}, 
and regard it as an automorphism of $GL_m (D)$. In that case it is known that
$\eta_\tau^* \pi \cong \pi^\vee$ for all $\pi \in GL_m (D)$ \cite[Theorem 3.1]{Rag}.

\begin{prop}\label{prop:8.1}
For any isomorphism of $F$-groups $f : GL_m (D^{op}) \to GL_m (D)$:
\[
f^* \pi (\phi,\zeta_D) = \pi (\Phi_e ({}^L f) (\phi,\zeta_D)) = \pi (\phi^\vee, \zeta_{D^{op}}) .
\]
\end{prop}
\begin{proof}
By Corollary \ref{cor:3.3} and the surjectivity of $q$, it suffices to consider $f = \eta_\tau$.

In \eqref{eq:8.26} $\eta_\tau (g')^{-1} \in GL_n (F)$ matches with $\eta_\tau (g)^{-1} \in 
GL_m (D^{op})$. Combining that with \eqref{eq:8.4}, we see that
\begin{multline*}
(\eta_\tau)^* \pi (\phi,\zeta_D) = JL_{D^{op}}(\eta_\tau^* \pi (\phi,1)) = 
JL_{D^{op}}(\pi (\phi^\vee,1)) \\
= \pi (\phi^\vee, \zeta_{D^{op}}) = \pi \big( \Phi_e ({}^L \eta_\tau) (\phi,\zeta_D ) ) \big) .
\end{multline*} 
A standard Levi subgroup of $G$ has the form 
\[
M = \prod\nolimits_i GL_{m_i}(D) \qquad \text{with } \sum\nolimits_i m_i = m . 
\]
Then $\eta_\tau^{-1}(M) = \prod_i GL_{m_i}(D^{op})$ and one checks that
$\eta_\tau : \eta_\tau^{-1}(M) \to M$ equals the product of the maps $\eta_\tau$ for the 
various $GL_{m_i}(D^{op})$. We denote the LLC for $\Irr_{\mr{ess}L^2}(M)$ by $\pi_M$.

For $\phi = \prod_i \phi_i$ with $\phi_i \in \Phi_{\mr{disc}}(GL_{m_i}(D))$, we find
\begin{equation}\label{eq:8.5}
\eta_\tau^* \pi_M (\phi,\zeta_D) = 
\eta_\tau^* \big( \bigotimes\nolimits_i \pi (\phi_i,\zeta_D) \big)
= \bigotimes\nolimits_i \pi ( \phi_i^\vee, \rho_D^{-1}) = 
\pi_{\eta_\tau^{-1}(M)} (\phi^\vee, \zeta_{D^{op}}) .
\end{equation}
Choose the parabolic subgroup $P = M U$ as in \cite[(13)]{ABPSSL1}. By construction 
$\pi (\phi,\zeta_D)$ is the unique Langlands quotient $L(P,\pi_M (\phi,\zeta_D))$ of
$I_P^G (\pi_M (\phi,\zeta_D))$. From \eqref{eq:8.5} we deduce
\[
\eta_\tau^* I_P^G (\pi_M (\phi,\zeta_D)) \cong I_{\eta_\tau^{-1}(P)}^{GL_m (D^{op})}
\big( \pi_{\eta_\tau^{-1}(M)} (\phi^\vee, \zeta_{D^{op}}) \big)
\]
and
\begin{multline*}
\eta_\tau^* \pi (\phi,\zeta_D) = \eta_\tau^* L(P,\pi_M (\phi,\zeta_D)) =
L \big( \eta_\tau^{-1}(P), \pi_{\eta_\tau^{-1}(M)} (\phi^\vee, \zeta_{D^{op}}) \big) \\
= \pi (\phi^\vee, \zeta_{D^{op}}) = \pi \big( \Phi_e ({}^L \eta_\tau) (\phi,\zeta_D) \big). \qedhere
\end{multline*}
\end{proof}

Notice that, because $PGL_m (D)$ is a quotient of $GL_m (D)$, $\Irr (PGL_m (D))$ can be considered
as a subset of $\Irr (GL_m (D))$. As L-parameter of any $\pi \in \Irr (PGL_m (D))$ one just takes
the L-parameter $\phi_{q^* \pi}$ of $q^* \pi \in \Irr (GL_m (D))$. 

Let us check that $\phi_{q^* \pi}$ really is a L-parameter for $PGL_m (D)$.
Via the LLC, the central character of $q^* \pi$ corresponds to $\det \circ \phi_{q^* \pi}$. 
(This is well-known for $GL_n (F)$, cf. the aforementioned references. It carries over to $GL_m (D)$ 
because every step in the construction of the LLC for that group preserves central characters.) 
But $q^* \pi$ is trivial on $\ker q = Z(GL_m (D)) = F^\times$, so $\det \circ \phi_{q^* \pi} = 1$.
In other words, the image of $\phi_{q^* \pi}$ lies in $SL_n (\C) = (PGL_n)^\vee = (PGL_m (D))^\vee$.

The component groups $\mc S_\phi$ for $PGL_m (D)$ are the same as for $GL_m (D)$. This allows one to
define the LLC for $PGL_m (D)$ as
\begin{equation}\label{eq:8.6}
\begin{array}{ccc}
\Irr (PGL_m (D)) & \longleftrightarrow & \Phi_e (PGL_m (D)) \\
\pi & \mapsto & (\phi_{q^* \pi}, \zeta_D) \\
\pi (q^\vee \circ \phi, \zeta_D) & \lmapsto & (\phi,\zeta_D) 
\end{array}.
\end{equation}
We note that by \eqref{eq:8.6}
\begin{equation}\label{eq:8.7}
\text{Conjecture \ref{conj:5.7} holds for the quotient map } q : GL_m (D) \to PGL_m (D) .
\end{equation}

\subsection{$SL_n$ and its inner forms} \
\label{par:SLn}

Let $n,F,D$ be as in the previous paragraph. Recall the reduced norm map 
$G = GL_m (D) \to F^\times$. We denote its kernel by $G^\sharp = SL_m (D)$. Then $SL_m (D)$ is
an inner twist of $SL_n (F)$, and every inner twist looks like this. A LLC for $SL_m (D)$ was
achieved in \cite{GeKn,HiSa,ABPSSL1}. It comes from the LLC for $GL_m (D)$, in agreement with
Conjecture \ref{conj:5.7}. 

Let us provide more background. For $\pi \in \Irr (G)$ we define 
\[
X^G (\pi) := \{ \chi \in \Irr (G / G^\sharp) : \pi \otimes \chi \cong \pi \} .
\]
For every $\chi \in X^G (\pi)$ we choose a nonzero $I_\chi \in \mr{Hom}_G (\pi \otimes \chi, 
\pi)$. By \cite[Lemma 2.4]{HiSa} the $I_\chi$ form a basis of $\mr{End}_{G^\sharp}(\pi)$. 
As each $I_\chi$ is unique up to scalars, it follows that there exists a 2-cocycle $\kappa_\pi$ 
and an algebra isomorphism 
\begin{equation}\label{eq:8.16}
\mr{End}_{G^\sharp}(\pi) \cong \C [X_G (\pi), \kappa_\pi] . 
\end{equation}
By \cite[Corollary 2.10]{HiSa} and \cite[(19)]{ABPSSL1}
\begin{equation}\label{eq:8.8}
\pi \cong \bigoplus\nolimits_{\rho \in \Irr (\mr{End}_{G^\sharp} (\pi))} \rho \otimes
\mr{Hom}_{\mr{End}_{G^\sharp}(\pi)}(\rho,\pi) ,
\end{equation}
as representations of $\mr{End}_{G^\sharp}(\pi) \oplus \C [G^\sharp]$.

Let $\imath : SL_m (D) \to GL_m (D)$ be the inclusion. Then
$\imath^\vee : GL_n (\C) \to PGL_n (\C)$ is the canonical projection. Every $\phi^\sharp \in
\Phi (SL_m (D))$ can be lifted to a $\phi \in \Phi (GL_m (D))$, and the choice of the lift
does not matter. Furthermore ${}^L \imath \circ \phi = \phi^\sharp$ and $\mc S_\phi \subset 
\mc S_{\phi^\sharp}$. The maps ${}^L \imath : \mc S_\phi \to \mc S_{\phi^\sharp}$ and
${}^S \imath : \C [\mc S_\phi] \to \C [\mc S_{\phi^\sharp}]$ are just the inclusions.

When $D$ is isomorphic to the quaternion algebra $\mh H$ (so $F = \R$), $X_G (\pi) = 1$ 
and $[\mc S_{\phi^\sharp} : \mc S_\phi] = 2$ for all $\pi \in \Irr (GL_m (D))$ and all
$\phi \in \Phi (GL_m (D))$, see \cite[Theorem 3.4]{ABPSSL1}.

Assume for the moment that $D \not\cong \mh H$. As explained in \cite[(21)]{ABPSSL1}, 
there is a canonical isomorphism
\begin{equation}\label{eq:8.14}
a : \mc S_{\phi^\sharp} / \mc S_\phi \to X^G (\pi (\phi,\zeta_D)) ,
\end{equation}
determined by $s \phi s^{-1} = a_s \phi = \hat \chi \phi$, where $\hat \chi = a_s$.
By \cite[Lemma 12.5]{HiSa} and \cite[Theorem 3.2]{ABPSSL1}, \eqref{eq:8.16} and 
\eqref{eq:8.14} can be combined to isomorphisms
\begin{equation}\label{eq:8.9}
\mr{End}_{G^\sharp}(\pi (\phi,\zeta_D)) \cong \C [X_G (\pi (\phi,\zeta_D)), 
\kappa_{\pi (\phi,\zeta_D)}] \cong e_{\zeta_D} \C [\mc S_{\phi^\sharp}] .
\end{equation}
In these sources $s \in \mc S_{\phi^\sharp}$ is mapped to a scalar multiple of
$I_\chi \in \mr{Hom}_G (\pi(\phi,\zeta_D) \otimes \chi, \pi (\phi,\zeta_D) )$. But it
is more natural to send $s$ to an element of $\mr{Hom}_G (\pi(\phi,\zeta_D), \\
\pi (\phi,\zeta_D) \otimes \chi)$, that agrees among others with \cite{Ree1,ABPSprin}.

We recall the crucial relation from \cite[Theorem 12.4]{HiSa}:
\begin{equation}\label{eq:8.17}
I_\chi I_{\chi'} I_\chi^{-1} I_{\chi'}^{-1} = \zeta_D (s s' s^{-1} (s')^{-1}) 
\qquad s, s' \in \mc S_{\phi^\sharp}, \hat \chi = a_s, \hat{\chi'} = a_{s'}.
\end{equation}
Notice that $I_\chi I_{\chi'} I_\chi^{-1} I_{\chi'}^{-1} \in \C \mr{Id}$ equals
$I_\chi^{-1} I_{\chi'}^{-1} I_\chi I_{\chi'}$. We rewrite \eqref{eq:8.17} as
\[
I_\chi^{-1} I_{\chi'}^{-1} I_\chi I_{\chi'} = \zeta_D \big( s s' s^{-1} (s')^{-1} \big) .
\]
As in \cite[Lemma 12.5]{HiSa}, this enables us to exhibit an isomorphism \eqref{eq:8.9}
which sends $e_{\zeta_D} s \in e_{\zeta_D} \C [\mc S_{\phi^\sharp}]$ to a scalar
multiple $I_s$ of $I_\chi^{-1} \in \mr{Hom}_G (\pi(\phi,\zeta_D), \pi (\phi,\zeta_D) \otimes \chi)$.
It is this isomorphism which we use to construct the LLC. In particular the $\rho$'s in 
\eqref{eq:8.8} can be regarded as representations of $\mc S_{\phi^\sharp}$ on which 
$\mc S_\phi$ acts as $\zeta_D \mr{id}_{V_\rho}$. Next we define
\begin{equation}\label{eq:8.10}
\pi (\phi^\sharp, \rho) = \mr{Hom}_{\mr{End}_{G^\sharp}(\pi (\phi,\zeta_D))}
(\rho,\pi (\phi,\zeta_D)) \in \Irr (G^\sharp).
\end{equation}
When $D \cong \mh H$, this does not fit entirely, because \eqref{eq:8.14} is not injective.
That is caused by the non-surjectivity of the reduced norm map $\mh H \to \R$. A solution
is provided by strong rational forms, in the sense of \cite[Definition 2.6 and Problem 9.3]{Vog}.
(See \cite[Chapter 2]{ABV} and \cite[\S 5.1]{Kal2} for closely related notions.)
As worked out in \cite[Example 2.11]{Vog}, there are exactly two non-split strong rational forms 
of $SL_{2m}(\R)$, parametrized by the 1-cocycles $\mr{Gal}(\C / \R) \to SL_{2m}(\C)$ sending
$\mr{id}_\C$ to 1 and complex conjugation to $\pm \matje{0}{1}{-1}{0}^{\oplus m}$.
These are equivalent as inner twists of $SL_{2m}$ but not as strong rational forms. 

Every $\phi^\sharp \in \Phi (SL_m (\mh H))$ can be put into a standard shape, as in the proof
of \cite[Theorem 2.2]{ABPSSL1}. Then $\phi$ comes from a discrete L-parameter for
$GL_1 (\mh H)^m$, with image in $GL_2 (\C)^m$. The computations in the proof of 
\cite[Theorem 3.4]{ABPSSL1} show that we can present $\mc S_{\phi^\sharp}$ as
\begin{equation}\label{eq:8.30}
\mc S_{\phi^\sharp} = \big\langle \big\{ 1, e^{\pi i / m}, \matje{i}{0}{0}{-i}^{\oplus m},
e^{\pi i / m} \matje{i}{0}{0}{-i}^{\oplus m} \big\} \big\rangle 
\big/ \langle e^{2 \pi i / m} \rangle .
\end{equation}
The $SL_m (\mh H)$-relevance condition becomes $\rho (e^{\pi i / m}) = -1$, so there are
precisely two relevant enhancements of $\phi^\sharp$. We match them with the two strong 
rational forms under consideration, in the same way for every $\phi^\sharp$. In other words,
we decree that every such $\rho$ is relevant for just one strong rational form, say 
$G^\sharp_\rho$. Then we can define a bijective LLC by
\begin{equation}\label{eq:8.29}
\pi (\phi^\sharp, \rho) := \pi (\phi, \zeta_{\mh H}) \in \Irr (G^\sharp_\rho) . 
\end{equation}

\begin{thm}\label{thm:8.2}\textup{\cite{GeKn,HiSa,ABPSSL1}} \\
With the above notations $\imath^* \pi (\phi,\zeta_D)$ equals (for $D \not\cong \mh H$)
\[
\bigoplus_{\rho \in \Irr (\mc S_{\phi^\sharp}) : 
\rho |_{\mc S_\phi} = \zeta_D } \rho \otimes \pi (\phi^\sharp, \rho) 
= \bigoplus_{\rho \in \Irr (\mc S_{\phi^\sharp})} \Hom_{\mc S_\phi} (\zeta_D, {}^S \imath^*(\rho)) 
\otimes \pi (\phi^\sharp, \rho) .
\]
When $D \cong \mh H$, this holds if we restrict the direct sums to relevant 
enhancements $\rho$.
\end{thm} 
\begin{proof}
When $D \cong \mh H$, \eqref{eq:8.30} shows that $\mc S_{\phi^\sharp}$ has order four.
In particular it is commutative, so $\dim \rho = 1$ and the theorem reduces to \eqref{eq:8.29}.

Suppose that $D \not\cong \mh H$. 
Via our isomorphism \eqref{eq:8.9}, we can reformulate \eqref{eq:8.8} as 
\begin{equation}\label{eq:8.11}
\imath^* \pi (\phi,\zeta_D) = \bigoplus\nolimits_{\rho \in \Irr (e_{\zeta_D} 
\C [\mc S_{\phi^\sharp}] )} \rho \otimes \pi (\phi^\sharp, \rho) .
\end{equation}
Since $\mc S_\phi$ is central in $\mc S_{\phi^\sharp}$ and ${}^S \imath$ is just the inclusion,
for any $\rho$ in \eqref{eq:8.11} :
\[
\rho \cong \Hom_{\mc S_\phi}(\zeta_D, \rho) = \Hom_{\mc S_\phi} (\zeta_D, {}^S \imath^* (\rho)) . 
\qedhere
\]
\end{proof}

The above can be pushed to one further instance of Conjecture \ref{conj:5.7}. We note that
the canonical map 
\begin{equation}\label{eq:8.12}
f : SL_m (D) \to PGL_m (D)
\end{equation}
factors as 
\[
SL_m (D) \xrightarrow{\imath} GL_m (D) \xrightarrow{q} PGL_m (D) .
\]
From Theorem \ref{thm:8.2} and \eqref{eq:8.7} we see that 
\begin{equation}\label{eq:8.13}
\text{Conjecture \ref{conj:5.7} holds for } f : SL_m (D) \to PGL_m (D) .
\end{equation}
Notice that $PGL_m (D)$ is the adjoint group of $SL_m (D)$. This is a case where the adjoint
quotient map \eqref{eq:8.12} can be far from surjective. Like in the split case, the image
of $f$ consists of all elements of $PGL_m (D)$ whose reduced norm is $1 \in F^\times / F^{\times n}$.
(Dividing out to the subgroup of $n$-th powers in $F^\times$ makes the reduced norm map
well-defined on $PGL_m (D)$.) When $D \cong \mh H$, Nrd$(GL_m (D)) = \R_{>0} = \R^{\times 2}$,
so $f$ is surjective. Otherwise Nrd$: GL_m (D) \to F^\times$ is surjective, and we find  
\[
PGL_m (D) / f (SL_m (D)) \cong F^\times / F^{\times n} .
\]
This group is finite if $n$ is coprime to the characteristic of $F$. If char$(F)$ divides $n$,
then $F^\times / F^{\times n}$ is infinite (but still compact).

\begin{prop}\label{prop:8.3}
Let $g \in PGL_m (D)$ and $(\phi^\sharp, \rho) \in \Phi_e (SL_m (D))$. Then
\[
\mr{Ad}(g)^* \pi (\phi^\sharp,\rho) = \phi (\phi^\sharp, \rho \otimes \tau_{\phi^\sharp}(g)^{-1}) 
= \bigoplus_{\tilde \rho \in \Irr (\mc S_{\phi^\sharp})} \Hom_{\mc S_\phi} \big( \rho, 
{}^S \mr{Ad}(g)^*(\tilde \rho) \big) \otimes \pi (\phi^\sharp, \tilde \rho) . 
\]
\end{prop}
\begin{proof}
When $D \cong \mh H$, $g$ lies in the image of $SL_m (D)$ and there is nothing to prove.
Therefore we may assume that $D \not\cong \mh H$.

Regard $g$ as a character of the group $X^G (\pi (\phi,\zeta_D))$. Via \eqref{eq:8.14} $g$ 
determines a character of $\mc S_{\phi^\sharp} / \mc S_\phi$, which we claim
is none other than $\tau_{\phi^\sharp}(g)$. 

Multiplying $\phi \in \Phi (GL_m (D))$ by a map $\mb W_F \to Z(GL_n (\C))$, we can adjust it to
a lift $\phi_\Sc : \mb W_F \times SL_2 (\C) \to SL_n (\C)$ of $\phi^\sharp$. Comparing 
\eqref{eq:2.16} and \eqref{eq:8.14}, we find
\[
\hat \chi = a_s = s \phi |_{\mb W_F} s^{-1} \phi^{-1} |_{\mb W_F} = 
s \phi_\Sc s^{-1} \phi_\Sc^{-1} = c_s . 
\]
By \eqref{eq:2.15} 
\[
\tau_{\phi^\sharp}(g) (s) = \langle g, c_s \rangle = \langle g, \hat \chi \rangle = \chi (g) ,
\]
which proves our claim. We plug this into \eqref{eq:8.10}:
\begin{equation}\label{eq:8.15}
\begin{aligned}
\pi (\phi^\sharp, \rho \otimes \tau_{\phi^\sharp}(g)^{-1}) 
& = \mr{Hom}_{e_{\zeta_D} \C [\mc S_{\phi^\sharp}]}(\rho \otimes 
\tau_{\phi^\sharp}(g)^{-1},\pi (\phi,\zeta_D)) \\
& = \mr{Hom}_{\C [X^G (\pi (\phi,\zeta_D)), 
\kappa_{\pi (\phi,\zeta_D)}]}(\rho \otimes g^{-1},\pi (\phi,\zeta_D)) .
\end{aligned}
\end{equation}
Recall that $\chi \in X^G (\pi (\phi,\zeta_D))$ acts on $\pi (\phi,\rho)$ via 
$I_s \in \mr{Hom}_G (\pi(\phi,\zeta_D), \pi (\phi,\zeta_D) \otimes \chi)$. 
If $\psi$ lies in the right hand side of \eqref{eq:8.15} and $v \in V_\rho$, then
\begin{multline*}
\pi (\phi,\zeta)(g) \psi \big( \rho (\chi) v \big) = \chi (g) \pi (\phi,\zeta)(g) 
\psi \big( (\rho \otimes g^{-1}) (\chi) v \big) \\
= (\pi (\phi,\zeta) \otimes \chi) (g) I_s \big( \psi (v) \big) 
= I_s \big( \pi (\phi,\zeta)(g) \psi (v) \big) ,
\end{multline*}
which shows that $\pi (\phi,\zeta)(g) \psi \in \mr{Hom}_{\C [X^G (\pi (\phi,\zeta_D)), 
\kappa_{\pi (\phi,\zeta_D)}]}(\rho,\pi (\phi,\zeta_D))$. It follows that \eqref{eq:8.15} 
can be identified with 
\begin{equation}\label{eq:8.18}
\pi (\phi,\zeta_D)(g^{-1}) \mr{Hom}_{\C [X^G (\pi (\phi,\zeta_D)), \kappa_{\pi (\phi,\zeta_D)}]}
(\rho,\pi (\phi,\zeta_D)) = \pi (\phi,\zeta_D)(g^{-1}) \pi (\phi^\sharp, \rho) . \hspace{-4mm}
\end{equation}
Recall that
\[
\mr{Ad}(g)^* \pi (\phi^\sharp,\rho) (g^\sharp) =  \pi (\phi^\sharp,\rho) (g g^\sharp g^{-1})
\qquad g^\sharp \in G^\sharp .
\]
Considering $\pi (\phi^\sharp,\rho)$ as a $G^\sharp$-subrepresentation of $\pi (\phi,\zeta_D)$,
$\mr{Ad}(g)^* \pi (\phi^\sharp,\rho)$ is the $G^\sharp$-representation on the vector subspace
$\pi (\phi,\zeta_D)(g^{-1}) \pi (\phi^\sharp,\rho)$. This agrees with the formulas \eqref{eq:8.15}
and \eqref{eq:8.18} for $\pi (\phi^\sharp, \rho \otimes \tau_{\phi^\sharp}(g)^{-1})$. 
The final equality in the lemma follows as in the last lines of the proof of Theorem \ref{thm:6.3}.
\end{proof}

\begin{ex}\label{ex:SLn}
We work out an example which also fits in Sections \ref{sec:principal} and \ref{sec:unip}.

Consider $G^\sharp = SL_n (F)$ and $G = GL_n (F)$. Let $T^\sharp = \mc T^\sharp (F)$ 
be the diagonal torus in $G^\sharp$ and let $B^\sharp$ be the Borel subgroup of upper 
triangular matrices. Write $\zeta_n = \exp (2 \pi i / n) \in \C^\times$, a primitive $n$-th root of 
unity. Let $\chi \in \Irr (T^\sharp)$ be the unramified character with parameter 
$(1,\zeta_n,\zeta_n^2, \ldots, \zeta_n^{n-1})$. It is well-known that the parabolically induced 
representation $I_{B^\sharp}^{G^\sharp}(\chi)$ is a direct sum of $n$ inequivalent irreducible 
subrepresentations. 

To analyse this more concretely, we can use the method from \cite{Lus-Gr}. The stabilizer of 
$\chi$ in the Weyl group $W(G^\sharp,T^\sharp) \cong S_n$ is 
$\langle 1 \, 2 \ldots n \rangle \cong \Z / n \Z$. Via some equivalences and reduction steps, 
$I_{B^\sharp}^{G^\sharp}(\chi)$ becomes a standard representation of the twisted graded Hecke 
algebra $\C[\C^n / \C] \rtimes W_\chi$, namely $\mr{ind}_{\C[\C^n / \C]}^{\C[\C^n / \C] \rtimes W_\chi} 
(\mr{ev}_0)$. The irreducible constituents of the latter representation have dimension one and 
parametrized canonically by the $W_\chi$-characters 
that they afford. In this way we associate to every $\tau \in \Irr (W_\chi)$ an irreducible 
direct summand $I_{B^\sharp}^{G^\sharp}(\chi)_\tau$ of $I_{B^\sharp}^{G^\sharp}(\chi)$.

The constituents of $I_{B^\sharp}^{G^\sharp}(\chi)$ form one L-packet, whose L-parameter 
$\phi^\sharp$ factors via ${}^L T^\sharp$ and encodes $\chi$. That is, $\phi^\sharp$ is trivial on 
$\mb I_F \times SL_2 (\C)$, and it sends $\Fr_F$ to $(1,\zeta_n,\zeta_n^2, \ldots, \zeta_n^{n-1}) 
\in (T^\sharp)^\vee$. When $\phi \in \Phi (G)$ is a lift of $\phi^\sharp$, we can identify
$I_{B^\sharp}^{G^\sharp}(\chi)$ with $\mr{Res}_{G^\sharp}^G \pi (\phi, \zeta_F)$.

One checks easily that $\mc S_{\phi^\sharp} / \mc S_\phi \cong W_\chi$. The action of
$\mc S_{\phi^\sharp} / \mc S_\phi$ on $I_{B^\sharp}^{G^\sharp}(\chi)$ can be reduced to an 
action on $\mr{ind}_{\C[\C^n / \C]}^{\C[\C^n / \C] \rtimes W_\chi} (\mr{ev}_0)$, 
and then we can identify it with the right regular representation of $W_\chi$. 
With \eqref{eq:8.29} we obtain
\[
\pi (\phi^\sharp, \tau^*) = \mr{Hom}_{W_\chi} (\tau^*, I_{B^\sharp}^{G^\sharp}(\chi))
= I_{B^\sharp}^{G^\sharp}(\chi)_\tau .
\]
This is the parametrization of $\Pi_{\phi^\sharp}(G^\sharp)$ from 
\cite{Ree1,ABPSprin,SolLLCunip}. In contrast, $I_{B^\sharp}^{G^\sharp}(\chi)_\tau$ is 
matched with $(\phi^\sharp,\tau)$ in \cite{HiSa}.

Consider $t = \mr{diag}(\varpi_F,1,1,\ldots,1) \in PGL_n (F) = G^\sharp_\ad$. For 
$w = \langle 1 \, 2 \ldots n \rangle \in W_\chi$ and $v \in I_{B^\sharp}^{G^\sharp}(\chi)_\tau$ 
which comes from $\mr{ind}_{\C[\C^n / \C]}^{\C[\C^n / \C] \rtimes W_\chi} (\mr{ev}_0)$:
\[
(\mr{Ad}(t) T_w) \cdot v = (T_w \theta_{e_n - e_1}) \cdot v = T_w \cdot \zeta_n^{-1} v = 
\zeta_n^{-1} (T_w \cdot v) .
\]
It follows that $\mr{Ad}(t)^* I_{B^\sharp}^{G^\sharp}(\chi)_\tau =
I_{B^\sharp}^{G^\sharp}(\chi)_{\zeta_n^{-1} \tau}$, where we identify $\tau$ with its value at
$\langle 1 \, 2 \ldots n \rangle$. In terms of enhanced L-parameters:
\begin{equation}\label{eq:8.31}
\mr{Ad}(t)^* \pi (\phi^\sharp, \tau^*) = \pi (\phi^\sharp, \zeta_n \tau^* ) . 
\end{equation}
The cocycles $c_h \in Z^1 (\mb W_F,Z(SL_n \C))$ for $\phi^\sharp$, from \eqref{eq:2.16}, 
are unramified and depend only on the image of $h$ in 
$\mc S_{\phi^\sharp} / \mc S_\phi \cong W_\chi$. We compute
\begin{multline*}
c_{\langle 1 \, 2 \ldots n \rangle}(\Fr_F) = \\
\langle 1 \, 2 \ldots n \rangle \:
\mr{diag}(1,\zeta_n,\zeta_n^2, \ldots, \zeta_n^{n-1}) \: \langle 1 \, 2 \ldots n \rangle^{-1}
\: \mr{diag}(1,\zeta_n,\zeta_n^2, \ldots, \zeta_n^{n-1})^{-1} \\
= \mr{diag}(\zeta_n^{n-1},\zeta_n^{n-1},\ldots, \zeta_n^{n-1}) .
\end{multline*}
Consequently $c_{\langle 1 \, 2 \ldots n \rangle^k}(\Fr_F^d) = \zeta_n^{-kd} I_n$ for all
$k,d \in \Z$, and
\[
\tau_{\phi^\sharp,SL_n}(t) (\langle 1 \, 2 \ldots n \rangle^k) =
\langle t, c_{\langle 1 \, 2 \ldots n \rangle^k} \rangle = \zeta_n^{-k} .
\]
As characters of $W_\chi$, this can be abbreviated to 
$\tau_{\phi^\sharp,SL_n}(t) = \zeta_n^{-1}$. Thus 
\[
\pi (\phi^\sharp, \tau^* \otimes \tau_{\phi^\sharp,SL_n}(t)^{-1}) = 
\pi (\phi^\sharp, \zeta_n \tau^* ) ,
\]
in agreement with \eqref{eq:8.31} and Conjecture \ref{conj:A}.\\
\end{ex}

The map \eqref{eq:8.2} defines an isomorphism of $F$-groups
\begin{equation}\label{eq:8.19}
\eta_\tau : SL_m (D^{op}) \to SL_m (D) . 
\end{equation}
Let us investigate the functoriality of the LLC with respect to this isomorphism. The only
non-canonical step in the construction of the LLC for these groups is the choice of the 
isomorphisms \eqref{eq:8.9}, that is, of intertwining operators $I_s$ for $s \in \mc S_{\phi^\sharp}$.

\begin{lem}\label{lem:8.4} 
The intertwining operators $I_s$ can be chosen such that Conjecture \ref{conj:5.7} holds for
all isomorphisms between the $F$-groups $SL_m (D)$ and $SL_m (D^{op})$.
\end{lem}
\begin{proof}
By Proposition \ref{prop:8.3} and Corollary \ref{cor:3.3}, we only have to establish 
Conjecture \ref{conj:5.7} for $\eta_\tau$. 

When $D \cong \mh H$, we can compose with \eqref{eq:8.25} to obtain an automorphism 
$\overline{\eta_\tau}$ of $SL_m (D)$. Since \eqref{eq:8.25} intertwines the local Langlands 
correspondences for $GL_m (D)$ and $GL_m (D^{op})$, the definition \eqref{eq:8.29} implies
that it does the same for $SL_m (D)$ and $SL_m (D^{op})$. Thus the lemma reduces to the case of 
the automorphism $\overline{\eta_\tau}$. It is easy to check, with \cite[Example 2.11]{Vog} 
at hand, that $\overline{\eta_\tau}$ fixes both involved strong rational forms. Similarly one 
checks with \eqref{eq:8.30} that $\eta_\tau^\vee$ fixes $\mc S_{\phi^\sharp}$ pointwise. 
This means that the lemma for $SL_m (D)$ becomes equivalent to a statement about $GL_m (D)$
and $\overline{\eta_\tau}$. The latter was proven in Proposition \ref{prop:8.1}.

Now we consider $D \not\cong \mh H$. From Proposition \ref{prop:8.1} we know that 
\[
\eta_\tau^* \pi (\phi,\zeta_D) \cong \pi (\eta_\tau^\vee \circ \phi, \zeta_D^{-1}) \in
\Irr (GL_m (D^{op})) .
\]
We fix such an isomorphism $\lambda_\phi$. The choice does not matter, because it is unique
up to scalars. Now
\begin{multline}\label{eq:8.20}
I_s \in \mr{Hom}_{GL_m (D^{op})} \big( \eta_\tau^* (\pi (\phi,\zeta_D)), 
\eta_\tau^* (\pi (\phi,\zeta_D)) \otimes \eta_\tau^* (\chi) \big) \\
\cong \mr{Hom}_{GL_m (D^{op})} \big( \pi (\eta_\tau^\vee \circ \phi,\zeta_D^{-1}), 
\pi ( \eta_\tau^\vee \circ \phi,\zeta_D^{-1}) \otimes \eta_\tau^* (\chi) \big) .
\end{multline}
The image of $I_s$ on the second line of \eqref{eq:8.20} is 
\begin{equation}\label{eq:8.27}
\lambda_\phi \circ I_s \circ \lambda_\phi^{-1} =: I_{\eta_\tau^\vee (s)}^{op} .
\end{equation}
For $s \in \mc S_\phi$ we have $\eta_\tau^\vee (s) = s^{-1} \in 
\mc S_{\eta_\tau^\vee \circ \phi}$ and
\begin{equation}\label{eq:8.22}
\zeta_D^{-1}(s^{-1}) \mr{Id} = \zeta_D (s) \mr{Id} = I_s = \lambda_\phi \circ I_s \circ 
\lambda_\phi^{-1} = I_{\eta_\tau^\vee (s)}^{op} = I_{s^{-1}}^{op} .
\end{equation}
Together with \eqref{eq:8.9} this means that 
$\eta_\tau^\vee (s) \mapsto I_{\eta_\tau^\vee (s)}^{op}$ extends to an isomorphism 
\begin{equation}\label{eq:8.21}
e_{\zeta_D^{-1}} \C [\eta_\tau^\vee \mc S_{\phi^\sharp}] \to
\mr{End}_{SL_m (D^{op})} \pi (\eta_\tau^\vee \circ \phi, \zeta_D^{-1}) .
\end{equation}
When $D \neq D^{op}$, we can first choose the intertwining operators $I_s$ for $GL_m (D)$,
and then decree that the $I^{op}_{\eta_\tau^\vee (s)}$ are the intertwining operators for 
$GL_m (D^{op})$.

When $D = D^{op}$, we have to check that the new $I_s^{op}$ agree with the old $I_s$.
In that case $D = F$ and $G^\sharp = SL_n (F)$ is split. We pick an additive character $\chi_F :
F \to \C^\times$ and we endow $SL_n (F)$ with the Whittaker datum such that the nondegenerate 
character of the unipotent group of upper triangular matrices of is given by
\[
X \mapsto \chi_F \Big( \sum\nolimits_{i=1}^{n-1} X_{i,i+1} \Big) .
\]
We normalize the intertwining operators $I_s$ with respect to this Whittaker datum, as in 
\cite[Chapter 3]{HiSa}. This means that, whenever the L-packet $\Pi_{\phi^\sharp}(SL_n (F))$
contains a generic representation, it is $\pi (\phi^\sharp, \mr{triv})$ \cite[Lemma 12.8]{HiSa}.
As $\eta_\tau$ stabilizes the standard pinning of $SL_n (F)$, it also stabilizes this Whittaker 
datum. Hence the intertwining operators $I_{\eta_\tau^\vee (s)}^{op}$ are normalized with respect 
to the same Whittaker datum, and it follows that $I_{\eta_\tau^\vee (s)}^{op} = 
I_{\eta_\tau^\vee (s)}$ for every $s \in \mc S_{\phi^\sharp}$.

With that in order, we may apply \eqref{eq:8.21}:
\begin{align*}
\eta_\tau^* \pi (\phi^\sharp,\rho) & = \eta_\tau^* \big( \mr{Hom}_{e_{\zeta_D} 
\C [\mc S_{\phi^\sharp}]}(\rho,\pi (\phi,\zeta_D)) \big) \\
& = \mr{Hom}_{\mr{End}_{SL_m (D^{op})} \pi (\eta_\tau^\vee \circ \phi, \zeta_D^{-1})}
\big( \rho \circ (\eta_\tau^\vee )^{-1} , \pi (\eta_\tau^\vee \circ \phi, \zeta_D^{-1}) \big) \\
& = \pi (\eta_\tau^\vee \circ \phi^\sharp, \rho \circ (\eta_\tau^\vee )^{-1} ) 
= \pi ( \Phi_e ({}^L \eta_\tau)(\phi^\sharp,\rho)) . \qedhere
\end{align*}
\end{proof}

\begin{rem}
When $\zeta_D$ has order two, the division algebra $D$ is isomorphic to its opposite $D^{op}$.
However, when $F$ is non-archimedean we treat $SL_m (D)$ and $SL_m (D^{op})$ as different groups 
in Lemma \ref{lem:8.4}. This leaves a bit to be desired, because 
\eqref{eq:8.25} provides a canonical isomorphism between these groups. It is 
reasonable to require that \eqref{eq:8.25} intertwines the local Langlands correspondences
for $SL_m (D)$ and $SL_m (D^{op})$. Then the intertwining operators $I_s$ for $SL_m (D)$ are
transferred automatically to intertwining operators for $SL_m (D^{op})$.

We have to check that these agree with the intertwining operators $I_{\eta_\tau^\vee (s)}^{op}$ 
defined in the proof of Lemma \ref{lem:8.4} -- at least when $\eta_\tau^\vee \circ \phi^\sharp$
is $SL_n (\C)$-conjugate to $\phi^\sharp$, otherwise we can pick the $I_s$ for $\phi^\sharp$ and 
then define them for $\eta_\tau^\vee \circ \phi^\sharp$ by \eqref{eq:8.27}.
As this appears to be cumbersome, we sketch an alternative argument to establish Lemma 
\ref{lem:8.4} in the stronger sense when $\zeta_D$ has order two. In view of the Langlands
classification (see the remarks at the start of Paragraph \ref{par:classical}), it suffices
to consider bounded L-parameters and tempered representations.

An equivalent formulation of Lemma \ref{lem:8.4} comes from \cite[Lemma 5.1]{Xu1}:
\begin{equation}\label{eq:8.28}
\langle \eta_\tau (s) , \eta_\tau^* (\pi) \rangle_{\eta_\tau^\vee \circ \phi^\sharp} 
= \langle s , \pi \rangle_{\phi^\sharp} . 
\end{equation}
Here $\pi \in \Pi_{\phi^\sharp}(SL_m (D))$ and $\langle s , \pi \rangle_{\phi^\sharp} = \mr{tr}\, 
\rho_\pi (s)$ with $\rho_\pi$ as in \cite[Lemma 12.6]{HiSa}. By \cite[Theorem 12.7]{HiSa}
there exists a $c(s) \in \C^\times$ such that
\[
\mr{Tran}_{H^\sharp}^{G^\sharp} J(\phi_{H^\sharp}) = c(s) 
\sum\nolimits_{\pi \in \Pi_{\phi^\sharp}(G^\sharp)} \langle s, \pi \rangle_{\phi^\sharp} J(\pi) ,
\]
where J indicates the distribution associated to a representation or an L-packet.
As explained in the proof of \cite[Lemma 5.1]{Xu1}, the same arguments lead to
\[
\mr{Tran}_{H^\sharp}^{G^\sharp} J(\phi_{H^\sharp}) = c(\eta_\tau^\vee (s)) 
\sum\nolimits_{\eta_\tau^* \pi \in \Pi_{\eta^\vee \circ \phi^\sharp}(G^\sharp)} \langle 
\eta_\tau^\vee (s), \eta_\tau^* \pi \rangle_{\eta_\tau^\vee \circ \phi^\sharp} J(\pi) .
\]
If $c(s) = c(\eta_\tau^\vee (s))$ for all $s \in \mc S_{\phi^\sharp}$, then the argument from
\cite{Xu1} applies and proves \eqref{eq:8.28}. The constants $c(s)$ and $c(\eta_\tau^\vee (s))$
come from the global trace formulas in \cite[Chapter 17]{HiSa}. Since $\eta_\tau$ provides 
an automorphism of the entire setting, $c(\eta_\tau^\vee (s))$ should equal $c(s)$ when all
choices are made $\eta_\tau$-equivariantly. We prefer not to work the details out here.
\end{rem}

\subsection{Quasi-split classical groups} \
\label{par:classical}

In this paragraph $F$ is a local field of characteristic zero and $E$ is a quadratic extension 
of $F$. Let $\mc G$ be one of the quasi-split groups 
\begin{equation}\label{eq:8.23}
Sp_{2n}, GSp_{2n}, PSp_{2n}, SO_{2n+1}, SO_{2n}, GSO_{2n},
PSO_{2n}, SO_{2n}^*, GSO_{2n}^*, PSO_{2n}^*, U_n . 
\end{equation}
Write $G = \mc G (F)$ when $\mc G$ is split and $G = \mc G (E/F)$ when $\mc G$ is non-split. 
For these groups a LLC is known, thanks to the work of Arthur and his PhD students. But it
is only written down for tempered representations, and several useful properties of these
LLCs are not worked out. We take this opportunity to fill in some of the details.

\begin{thm}\label{thm:8.8}
A local Langlands correspondence exists for $G$ as above, and for all its Levi subgroups.
It takes the form of a bijection
\[
\begin{array}{ccc}
\Phi_e (G) & \longleftrightarrow & \Irr (G) \\
(\phi,\rho) & \mapsto & \pi (\phi,\rho)
\end{array}
\]
and it enjoys the following properties:
\enuma{
\item The central character of $\pi (\phi,\rho)$ equals the $Z(G)$-character $\chi_\phi$
associated to $\phi$ in \cite[\S 10.1]{Bor}.
\item For $\chi \in X_\nr (G)$ with parameter $\hat \chi \in H^1 (\mb W_F, Z(G^\vee)^\circ)$:
$\pi (\hat \chi \phi, \rho) = \chi \otimes \pi (\phi,\rho)$.
\item $\pi (\phi,\rho)$ is essentially square-integrable if and only if $\phi \in \Phi (G)$
is discrete.
\item $\pi (\phi,\rho)$ is tempered if and only if $\phi \in \Phi (G)$ is bounded.
}
\end{thm}
\begin{proof}
(d) First we only consider tempered representations and enhanced bounded L-parameters.
For those, a bijective LLC was established in \cite{Art} (for $Sp_{2n}, SO_{2n+1}$,
$SO_{2n}, SO_{2n}^*$), \cite{Mok} (for $U_n$) and \cite{Xu2} (for $GSp_{2n}, GSO_{2n},
GSO_{2n}^*$). For $PSp_{2n}$, $PSO_{2n}$ and $PSO_{2n}^*$ it is derived formally from the
LLC for tempered representations, just like the LLC for $PGL_n$ follows from that for
$GL_n$. 

Every Levi $F$-subgroup of $\mc G$ is the direct product of some factors $GL_m$ and one 
group of the same type as $\mc G$ (but of smaller rank). Thus the above references and
the LLC for $GL_n (F)$ also provide a bijective LLC for Levi subgroups of $G$.\\
(b) For the adjoint groups $PSp_{2n}, PSO_{2n}, PSO_{2n}^*$ this is trivial.
For $Sp_{2n}, SO_{2n}, SO_{2n}^*$ see \cite[Proposition 6.27]{Xu1}. For $U_n$,
see \cite[Theorem 2.5.1.c]{Mok}.

Recall that 
\begin{equation}\label{eq:8.35}
GSp_{2n} = \big( GL_1 \times Sp_{2n} \big) / \{ \pm 1 \} \quad \text{and} \quad
GSO_{2n}^{(*)} = \big( GL_1 \times SO_{2n}^{(*)} \big) / \{ \pm 1 \} .
\end{equation}
When $\mc G$ is one of these groups
\begin{equation}\label{eq:8.36}
G^\vee = \big( \C^\times \times Spin_N (\C) \big) / \{ \pm 1 \},
\end{equation}
where $N = 2n$ in the orthogonal case and $N = 2n+1$ in the symplectic case.
Let $\mc G^\sharp$ be the subgroup $Sp_{2n}, SO_{2n}$ or $SO_{2n}^*$ of $\mc G$.
The inclusion $\mc G^\sharp \to \mc G$ induces a surjection $G^\vee \to (G^\sharp)^\vee
= SO_N (\C)$, with kernel $Z(G^\vee) = \C^\times$.

L-packets for the similitude groups $GSp_{2n}, GSO_{2n}, GSO_{2n}^*$ are described in
\cite[Theorem 1.2]{Xu2}. Let $\phi \in \Phi_{\mr{bdd}}(G)$. The L-packet $\Pi_\phi (G)$ 
depends only on the image $\phi^\sharp$ of $\phi$ in $\Phi (G^\sharp)$ and on a character 
$\zeta : Z(G) \to \C^\times$ which extends the central 
character of any member of $\Pi_{\phi^\sharp}(G^\sharp)$. By part (a) for $G^\sharp$.
the latter equals $\chi_{\phi^\sharp} = \chi_\phi |_{Z(G^\sharp)}$. Hence one can choose
$\zeta = \chi_\phi$, and then \cite{Xu2} provides a LLC with the desired central characters.

Part (a) is also well-known for $GL_n$, see the references in Paragraph \ref{par:GLn}.
In view of the shape of the Levi subgroups of $\mc G$, the LLC for those has the same
property.\\
(b) Among the groups \eqref{eq:8.23}, only the similitude groups admit nontrivial
unramified characters. Let $\mc G$ be one of those and let $(\phi,\rho) \in \Phi_e (G)$
be bounded. For a unitary $\chi \in X_\nr (G)$, $\hat \chi \phi$ still bounded. Since
$\hat \chi$ lives in the factor $\C^\times$ of \eqref{eq:8.36}, $\hat \chi \phi$ has the
same image $\phi^\sharp$ in $\Phi (G^\sharp)$ as $\phi$. The constructions from \cite{Xu2}
entail that 
\[
\pi (\hat \chi \phi, \rho) = \zeta \otimes \pi (\phi,\rho)
\]
for some $Z(G)$-character $\zeta$ which is trivial on $Z(G^\sharp)$. With the conventions
from the proof of part (a), this central character $\zeta$ equals 
\[
\chi_{\hat \chi \phi} \chi_\phi^{-1} = \chi_{\hat \chi} = \chi .
\]
For $GL_n$, part (b) is an important and intrinsic part of the LLC, see the references in
Paragraph \ref{par:GLn}. Clearly part (b) is inherited by direct products of groups, so
it also holds for Levi subgroups of $G$.\\
(c) This is the starting point of the construction of the LLC, see \cite{Art,Mok,Xu2}. 
It also holds for $GL_n$, and hence for all Levi subgroups of $G$ as well.

Now we consider all $\pi \in \Irr (G)$ and all $(\phi,\rho) \in \Phi_e (G)$. We say that
$\pi$ is essentially tempered if its restriction to $G_\der$ is tempered. Then there is a
unique (necessarily unramified) character $\nu : G \to \R_{>0}^\times$ such that 
$\pi \otimes \nu^{-1}$ is tempered. Multiplying $\phi_b \in \Phi_{\mr{bdd}} (G)$ by $\hat \nu$
does not change $\mc S_{\phi_b}$. We extend the tempered LLC to such $\pi$ by defining
\[
\pi (\hat \nu \phi_b, \rho) = \pi (\phi_b,\rho) \otimes \nu .
\]
The same can be done for all Levi subgroups of $G$. 

Now part (b) holds for all essentially tempered representations and all unramified 
characters. This brings us in the right position to apply the Langlangs classification,
both for $G$-representations \cite{Lan,BoWa,Ren} and for (enhanced) L-parameters \cite{SiZi}.
As in \cite{ABPSnontemp}, this allows us to extend the LLC canonically to the whole of
$\Irr (G)$ (and the same for Levi subgroups of $G$).

Roughly speaking, this entails the following. Given any $(\phi,\rho) \in \Phi_e (G)$,
find a parabolic subgroup $P = MU$ of $G$ such that $\phi$ factors through ${}^L M$ and,
as an element of $\Phi (M)$, is of the form $\phi = \hat \nu \phi_b$ with $\phi_b \in
\Phi_{\mr{bdd}}(M)$ and $\nu : M \to \R_{>0}^\times$ strictly positive (with respect to $P$).
This can be done, $P$ is unique up to conjugation and
\[
\mc S_{\phi,G} \cong \mc S_{\phi,M} \cong \mc S_{\phi_b,M} .
\]
Above we associated to $\phi$ an essentially tempered representation
\[
\pi_M (\phi,\rho) = \pi_M (\phi_b,\rho) \otimes \nu \in \Irr (M) .
\]
With normalized parabolic induction we form $I_P^G (\pi_M (\phi,\rho)) \in \Rep (G)$.
Then $\pi (\phi,\rho) \\ \in \Irr (G)$ is the unique irreducible quotient of 
$I_P^G (\pi_M (\phi,\rho))$. 

This operation preserves the bijectivity of the LLC, by \cite[Proposition 7.3]{SiZi}. We also 
need to check that the properties (a), (b) and (c) still hold for this version of the LLC.

By the tempered case of part (a), $\pi_M (\phi,\rho)$ admits the $Z(M)$-character
\[
\chi_{\phi_b} \otimes \nu = \chi_{\hat \nu \phi_b} = \chi_\phi . 
\]
Then $I_P^G (\pi_M (\phi,\rho))$ admits the $Z(G)$-character $\chi_\phi |_{Z(G)}$. But by 
\cite[\S 10.1]{Bor} $\chi_\phi |_{Z(G)}$ is just $\chi_\phi$ for $G$. 

The argument for part (b) in the tempered case still applies.

Part (c) follows from the tempered case by two properties of the Langlands classification.
Firstly, $I_P^G (\pi_M (\phi,\rho)$ can only be essentially square-integrable if $P = G$
and $\pi_M (\phi,\rho)$ is essentially square-integrable. Secondly, $\phi$ is discrete
if and only if the $M$ constructed above equals $G$.
\end{proof}

For most homomorphisms between these groups, Conjecture \ref{conj:5.7} has already been 
proven with endoscopic methods. Here we collect the relevant results from various references.

\begin{lem}\label{lem:8.6}
Conjecture \ref{conj:5.7} holds for the canonical homomorphisms
\[
\xymatrix{
Sp_{2n} \ar[r] \ar[dr] & GSp_{2n} \ar[d] \\ & PSp_{2n} 
},
\xymatrix{
SO_{2n} \ar[r] \ar[dr] & GSO_{2n} \ar[d] \\ & PSO_{2n} 
},
\xymatrix{
SO_{2n}^* \ar[r] \ar[dr] & GSO_{2n}^* \ar[d] \\ & PSO_{2n}^* 
}.
\]
\end{lem}
\begin{proof}
By the canonicity of the extension of the LLC from tempered representations to the whole
of $\Irr (G)$ (see above), it suffices to prove the lemma for bounded L-parameters.

For $Sp_{2n} \to GSp_{2n}$ and $SO_{2n}^{(*)} \to GSO_{2n}^{(*)}$, see 
\cite[Proposition 6.12]{Xu1}. The assumptions made in \cite{Xu1} were verified in \cite{Xu2}. 
The same statement can be found in \cite[Theorem 4.20]{Cho}, but over there the working 
hypotheses are so strong that it is not clear whether they are fulfilled completely by 
\cite{Art,Xu2}.

For $GSp_{2n} \to PSp_{2n}$ and $GSO_{2n}^{(*)} \to PSO_{2n}^{(*)}$ Conjecture
\ref{conj:5.7} holds by construction, compare with \eqref{eq:8.6} and \eqref{eq:8.7}. Hence
it also holds for the composed maps $Sp_{2n} \to PSp_{2n}$ and $SO_{2n}^{(*)} \to PSO_{2n}^{(*)}$.
\end{proof}

To investigate Conjecture \ref{conj:5.7} for automorphisms of the groups \eqref{eq:8.23}
we tabulate some useful data:
\[
\begin{array}{cccc}
G & \mr{Inn}(G) & G_\ad / G & |\mr{Out}(G)| \\
\hline 
Sp_{2n} & PSp_{2n} & F^\times / F^{\times 2} & 1 \\
GSp_{2n} & PSp_{2n} & 1 & 2 \\
SO_{2n+1} & SO_{2n+1} & 1 & 1 \\
SO_{2n} & PSO_{2n} & F^\times / F^{\times 2} & 2  \; (6 \text{ if } 2n = 8) \hspace{-2cm} \ \\
GSO_{2n} & PSO_{2n} & 1 & 4  \; (12 \text{ if } 2n = 8) \hspace{-2cm} \ \\
SO_{2n}^* & PSO_{2n}^* & F^\times / F^{\times 2} & 2 \\
GSO_{2n}^* & PSO_{2n}^* & 1 & 4 \\
U_n & PU_n & 1 & 2
\end{array}
\]
We note in particular that, among the groups in this paragraph, the actions of $G_\ad / G$
on $\Irr (G)$ and $\Phi_e (G)$ can only be nontrivial for $Sp_{2n}, SO_{2n}$ and $SO_{2n}^*$.

\begin{lem}\label{lem:8.8}
Let $\mc G$ be one of the groups \eqref{eq:8.23} and let $g \in G_\ad$. 
For any $(\phi,\rho) \in \Phi_e (G)$: 
\[
\mr{Ad}(g)^* \pi (\phi,\rho) = \pi (\phi, \rho \otimes \tau_{\phi}(g)^{-1}) 
= \bigoplus_{\tilde \rho \in \Irr (\mc S_{\phi})} \Hom_{\mc S_\phi} \big( \rho, 
{}^S \mr{Ad}(g)^*(\tilde \rho) \big) \otimes \pi (\phi, \tilde \rho) . 
\]
\end{lem}
\begin{proof}
The second equality is formal, see the last lines of Section \ref{sec:principal}. Like in the
two previous lemmas, it suffices to prove the claims for bounded L-parameters.

The first equality is shown in \cite[\S 3]{Kal} and \cite[Conjecture 2 and \S 3.5--3.6]{Xu1}.
The assumptions needed in \cite{Xu1} can be found in \cite{Art,Mok,Xu2}. We note that actually
the outcome of \cite{Kal,Xu1} is 
\begin{equation}\label{eq:8.24}
\mr{Ad}(g)^* \pi (\phi,\rho) = \phi (\phi, \rho \otimes \tau_{\phi}(g)) .
\end{equation}
The above table shows that every element of $G_\ad / G$ has order one or two, so $\tau_\phi (g)
= \tau_\phi (g)^{-1}$.
\end{proof}

\begin{rem}
Interestingly, the proof of Lemma \ref{lem:8.8} reveals a discrepancy between the endoscopic
methods referred to in this section and the Hecke algebra techniques prominent in Sections
\ref{sec:1}--\ref{sec:unip}. With the former one arrives at \eqref{eq:8.24}, whereas the 
latter leads to 
\[
\mr{Ad}(g)^* \pi (\phi,\rho) = \pi (\phi, \rho \otimes \tau_{\phi}(g)^{-1}) .
\]
Of course there was no real issue in Lemma \ref{lem:8.8}, because Ad$(g)^*$ had order one or
two. A problem of this kind could have appeared in Proposition \ref{prop:8.3}, if we had followed
\cite{HiSa} naively, see Example \ref{ex:SLn}. Fortunately we could avoid the trouble,
by replacing $I_\chi$ with $I_\chi^{-1}$. 

This points to the source of the discrepancy: different conventions for intertwining operators.
Suppose that $s \in G^\vee, \phi \in \Phi (G)$ and $s \phi s^{-1} = a_s \phi$ with
$a_s \in H^1 (\mb W_F, Z(G^\vee))$. Via \cite[\S 10.1]{Bor}, $a_s$ determines a character
$\chi_s : G \to \C^\times$. In endoscopy it appears to be common to associate to $s$ and $\chi_s$
an intertwining operator
\[
\pi (s \phi s^{-1},\rho) \cong \pi (\phi,\rho) \otimes \chi_s \to \pi (\phi,\rho) ,
\]
see for instance \cite[p. 40]{HiSa} and \cite[(3.10)]{Xu1}. On the other hand, for modules
constructed via Hecke algebras, $s$ and $\chi_s$ typically give rise to an intertwining
operator $\pi (\phi,\rho) \to \phi (s\phi s^{-1}, \rho) \cong \pi (\phi,\rho) \otimes \chi_s$,
see for example \cite[(61)]{ABPSprin}. This dichotomy permeates to the action of 
$\mc S_{\phi}$ on a standard module associated to $\phi$, and causes different parametrizations
of the L-packets by $\Irr (\mc S_\phi)$.
\end{rem}

For a $\mb W_F$-automorphism of the based root datum of $\mc G$, let $\eta_\tau$
be as in Theorem \ref{thm:3.2}.a: the unique automorphism of $G$ which lifts $\tau$ and 
stabilizes a given pinning.

From \eqref{eq:8.35} we see that any of the similitude groups $\mc G$ has a unique automorphism 
$\eta_Z$ which is the identity on $\mc G^\sharp$ and the inverse map on $Z(\mc G) \cong GL_1$. 
This explains why $\mc G$ has twice as many automorphisms as $\mc G^\sharp$.

\begin{lem}\label{lem:8.7}
Let $\mc G$ be one of the algebraic groups from \eqref{eq:8.23} and let $\tau$ be a 
$\mb W_F$-equivariant automorphism of its based root datum. (For $SO_8, PSO_8, GSO_8$ we
assume that $\tau$ is not exceptional, that is, fixes the first two simple roots of
$D_4$ in the standard numbering.) Then, for any $(\phi,\rho) \in \Phi_e (G)$:
\[
\eta_\tau^* \pi (\phi,\rho) = \pi \big( \Phi_e ({}^L \eta_\tau)(\phi,\rho) \big) . 
\] 
\end{lem}
\begin{proof}
Recall from the proof of Theorem \ref{thm:8.8} that that the LLC is derived from the 
tempered case with the Langlands classification. The functoriality properties of that 
procedure entail that it suffices to establish the lemma for tempered representations
and enhanced bounded L-parameters.

First we consider the case where $\mc G$ is a similitude group ($GSp_{2n}, GSO_{2n},
GSO_{2n}^*$) and $\eta_\tau = \eta_Z$.
Then $\eta_Z^* \pi (\phi,\rho)$ and $\pi (\phi,\rho)$ have the same restriction to
$G^\sharp$, whereas their central characters are inverse to each other. 

From the above description of $\eta_Z$ and \eqref{eq:8.36} we see that $\eta_Z^\vee$ 
is the identity on $Spin_N (\C)$ and the inverse map on $Z(G^\vee) \cong \C^\times$. 
Since the connected group $Z(G^\vee)$ becomes trivial in any of the $\mc S_\phi$,
$\Phi_e ({}^L \eta_Z)$ does not affect the enhancements. Further $\eta_Z^\vee \circ \phi$ 
and $\phi$ have the same image $\phi^\sharp$ in $\Phi (G^\sharp)$. In the setting of 
\cite{Xu2}, this means that $\pi (\eta_Z^\vee \circ \phi, \rho)$ and $\pi (\phi,\rho)$
are both obtained from $(\phi^\sharp, \rho)$ be choosing different extensions of
$\chi_{\phi^\sharp} \in \Irr (Z(G^\sharp))$ to $Z(G)$. In the proof of Theorem \ref{thm:8.8}.a
we imposed that $\pi (\phi,\rho)$ gets the central character $\chi_\phi$. Then
$\pi (\Phi_e ({}^L \eta_Z) (\phi,\rho)) = \pi (\eta_Z^\vee \circ \phi,\rho)$
gets the central character $\chi_{\eta_Z^\vee \circ \phi} = \chi_\phi^{-1}$, while
$\pi (\eta_Z^\vee \circ \phi, \rho)$ and $\pi (\phi,\rho)$ have the same restriction to
$G^\sharp$. In other words, $\pi (\Phi_e ({}^L \eta_Z) (\phi,\rho)) = \eta_Z^* \pi (\phi,\rho)$.

If the lemma holds for $\eta_\tau$ and $\eta_{\tau'}$, then it also holds for $\eta_{\tau \tau'}$.
This means that (upon possibly replacing $\eta_\tau$ by $\eta_\tau \eta_Z$), we may assume that
$\eta_\tau$ fixes $Z(\mc G)^\circ$ pointwise. For every $\mc G$ this leaves at most one
nontrivial $\eta_\tau$.

For $SO_{2n}^{(*)}$, $\eta_\tau$ comes from conjugation by an element of $O_{2n}^{(*)}$, and
the lemma is built into Arthur's constructions, see \cite[Theorem 8.4.1 and Corollary 8.4.5]{Art}. 
In the other remaining cases $\tau$ is trivial or $\mc G = U_n$ or $\mc G$ is a similitude group.
For those instances we refer to \cite[Lemma 5.1]{Xu1}. The assumptions for that result were 
proven in \cite{Mok,Xu2}.
\end{proof}

\end{document}